\documentclass[english,letterpaper,12pt,leqno]{article}

\usepackage{tikz}
\usetikzlibrary{decorations.markings}
\usetikzlibrary{calc}

\usepackage[latin1]{inputenc}
\usepackage{tipa}
\usepackage{dsfont}
 
\usepackage{babel}

\usepackage{xr}
 
\usepackage{amsxtra}
\usepackage{amsmath}
\usepackage{amssymb}
\usepackage{amsfonts}
\usepackage[all,2cell]{xy}
\UseAllTwocells
\usepackage{mathrsfs}
\usepackage{amsthm}
\usepackage{enumitem}
\usepackage{hyperref}
\usepackage{subfigure}

\setlist[enumerate]{leftmargin=*,labelindent=.5pc}
\newcommand{\Par}{\vspace{5pt}\noindent\textbf{ 
\refstepcounter{equation}
(\theequation)}
  \textbf}

\newtheorem{thm}[equation]{Theorem}
\newtheorem{cor}[equation]{Corollary}

\newtheorem{prop}[equation]{Proposition}

\newtheoremstyle{example}{\topsep}{\topsep}%
     {}
     {}
     {\bfseries}
     {.}
     {2pt}
     {\thmname{#1}\thmnumber{ #2}\thmnote{ #3}}

   \theoremstyle{example}
   
   \newtheorem{Defi}[equation]{Definition}
   \newtheorem{defi}[equation]{Definition}
   \newtheorem{rem}[equation]{Remark}

   \newtheorem{exas}[equation]{Examples}
   \newtheorem{ex}[equation]{Example}

\newtheorem{exa}[equation]{Example}

\def\AAA{\mathbb{A}}
\def\AA{\mathbb{A}}

\def\CC{\mathbb{C}}

\def\PP{\mathbb{P}}
\def\RR{\mathbb{R}}
\def\ZZ{\mathbb{Z}}
\def\QQ{\mathbb{Q}}

\def\scrF{\mathscr{F}}
\def\F{\scrF}

\def\Ac{\mathcal{A}}
\def\Bc{\mathcal{B}}
\def\Cc{\mathcal{C}}

\def\Dc{\mathcal{D}}

\def\Fc{\mathcal{F}}

\def\Oc{\mathcal{O}}

\def\Tc{\mathcal{T}}

\def\Cb{\mathbf{C}}

 \def\End{\operatorname{End}\nolimits}
  \def\Ext{\operatorname{Ext}\nolimits}

\def\Ho{\operatorname{Ho}}
 \def\Hom{\operatorname{Hom}}
 
 \def\id{\operatorname{id}}
 \def\IHom{\underline{\Hom}}
 
 \def\Id{\operatorname{Id}\nolimits}

 \def\Lie{\operatorname{Lie}\nolimits}
 \def\Map{\operatorname{Map}}

 \def\Mod{\operatorname{Mod}}
\def\NC{\operatorname{NC}}
\def\Ob{\operatorname{Ob}\nolimits}
\def\on{\operatorname}

\def\pt{\operatorname{pt}\nolimits}

\def\Vect{\operatorname{Vect}}

\def\3x3{\operatorname{3}\times\operatorname{3}}
\def\1{{\bf 1}}
\def\lra{\longrightarrow}

\def\(({(\hskip -1mm (}
\def\)){)\hskip -1mm )}
\def\((({(\hskip -1mm (}
\def\))){)\hskip -1mm )}

\def\llb{ (\hskip -1mm (}
\def\rlb{ )\hskip -1mm )}

\def\be{\begin{equation}}
\def\ee{\end{equation}}
\def\ed{\end{document}}

\def\la{{\langle}}
\def\ra{{\rangle}}
\def\ind{\varinjlim{}}
\def\pro{\varprojlim{}}
\def\hoind{\operatorname{ho}\!\varinjlim{}}
\def\hopro{\operatorname{ho}\!\varprojlim{}}
 \def\k{\mathbf{k}}

\usepackage{MnSymbol}
\usepackage{euscript}

\def\op{{\operatorname{op}}}

\def\Fun{\operatorname{Fun}}

\def\h{\operatorname{h}\!}
\def\ho{\operatorname{ho}\!}

\def\hopro{\ho \underleftarrow{\lim}{}}
\def\Hom{\operatorname{Hom}}

\def\O{{\mathcal O}}

\def\J{{\mathcal J}}

\def\Kan{\operatorname{Kan}}
\def\h{\operatorname{h}\!}

\def\Map{\operatorname{Map}}

\def\N{\operatorname{N}}

\def\PF{{P^\triangleright}}

\def\PI{{P^\triangleleft}}

\def\s{{\mathscr S}}

\def\Rib{ {\mathcal Rib}}

\def\cm{\langle m \rangle}
\def\cn{\langle n \rangle}

 \def\SW{\mathcal {S}}
 
\def\T{{\mathcal T}}

\def\Set{ {\mathcal Set}}

\def\FC{\operatorname{FC}}

\def\cone{\operatorname{cone}}

 \def\dg{\on{dg}}
\def\Ndg{\N_{\on{dg}}}

\def\A{ {\EuScript A}}
\def\P{{\EuScript P}}
\def\B{ {\EuScript B}}
\def\C{{\EuScript C}}
\def\D{{\EuScript D}}
\def\E{{\EuScript E}}
\def\wD{\widetilde{{\EuScript D}}}

\def\R{{\EuScript R}}

\def\J{ {\EuScript J}}
\def\K{ {\EuScript K}}

\def\Set{ {\mathcal Set}}

\def\<<{\langle {}\hskip -.1cm {}\langle}
\def\>>{\rangle \hskip -.1cm \rangle}

\def\FC{\on{FC}}

\def\Heq{\mathsf{Hqe}}
\def\Qeq{\mathsf{Qe}}
\def\Mor{\mathsf{Mo}} 
\def\Hmo{\mathsf{Hmo}}
\def\Hmot{\mathsf{Hmo}^{(2)}}

\def\Qis{\mathsf{Qis}}

\setlist[enumerate,1]{label=(\arabic{*})}
\setlist[enumerate,2]{label=(\alph{*})}
\setlist[enumerate,3]{label=(\roman{*})}

\def\centerarc[#1](#2)(#3:#4:#5)  { \draw[#1] ($(#2)+({#5*cos(#3)},{#5*sin(#3)})$) arc (#3:#4:#5); }

\def\Zt{\mathbb{Z}/2\mathbb{Z}}

\def\LCat{{\L\text{-}\EuScript Cat}}

\def\dgcat{\mathsf{dgcat}}
\def\dgcatt{\dgcat^{(2)}}
\def\Vectt{\Vect^{(2)}_{\k}}
\def\Modt{\Mod^{(2)}}
\def\Perf{\on{Perf}}
\def\Perft{\Perf^{(2)}}
\def\IRHom{R \underline{\on{Hom}}}
\def\IFun{\underline{\on{Fun}}}
\def\MF{\on{MF}}
\def\LF{\on{LF}}

\def\E{ {\EuScript E}}
\def\T{ {\EuScript T}}
\def\Q{ {\EuScript Q}}
\def\L{ {\ZZ_+}}
\def\mod{\on{-mod}}
\def\LSet{\L\text{-}\Set}
\def\sSet{{\Set}_{\Delta}}
\def\cm{\langle m \rangle}
\def\cn{\langle n \rangle}
\newcommand{\co}[1]{ {\langle #1 \rangle}}
\def\Zp{{\mathbb Z}_+}

\title{Triangulated surfaces in triangulated categories}
 \author{T. Dyckerhoff\footnote{Department of Mathematics, Yale University,
	  10 Hillhouse Avenue, New Haven CT 06520 USA, email:
	   {\tt tobias.dyckerhoff@yale.edu, mikhail.kapranov@yale.edu}  }, M.
	   Kapranov\footnotemark[1]
}

\begin{document}
%

\maketitle
\begin{abstract}
For a triangulated category $\Ac$ with a $2$-periodic dg-enhancement and a triangulated oriented marked
surface $S$, we introduce a dg-category $\F(S,\Ac)$ parametrizing systems of exact triangles in $\Ac$ labelled
by triangles of $S$. Our main result is that $\F(S,\Ac)$ is independent of the choice of a triangulation
of $S$ up to essentially unique Morita equivalence. In particular, it admits a canonical action of the
mapping class group. The proof is based on general properties of cyclic $2$-Segal spaces.

In the simplest case, where $\Ac$ is the category of $2$-periodic complexes of vector spaces,
$\F(S,\Ac)$ turns out to be a purely topological model for the Fukaya category of the surface $S$.
Therefore, our construction can be seen as implementing a $2$-dimensional instance of Kontsevich's
program on localizing the Fukaya category along a singular Lagrangian spine.
\end{abstract}

\tableofcontents

\addcontentsline{toc}{section}{Introduction}

\vfill\eject

\numberwithin{equation}{section}

\section*{Introduction} 
\paragraph{} 
The goal of this paper is to study a certain ``2-dimensional symmetry" built into the very foundations
of triangulated categories and thus of homological algebra more generally.
To make it manifest, we represent exact triangles in a triangulated category $\Dc$, in the dual fashion:
\begin{equation}\label{eq:geom-triangle}
  \begin{tikzpicture}[>=stealth,baseline=(current  bounding  box.center),decoration={
    markings,
    mark=at position 0.55 with {\arrow{>}}}]
\usetikzlibrary{calc}
\def\centerarc[#1](#2)(#3:#4:#5)  { \draw[#1] ($(#2)+({#5*cos(#3)},{#5*sin(#3)})$) arc (#3:#4:#5); }

    \node (left) at (-1.5,0) {};
 \fill (left) circle (0.05);
    \node (middle) at (0,2) {};
 \fill (middle) circle (0.05);
    \node (right) at (1.5,0) {};
 \fill (right) circle (0.05);
    
    
    \draw[postaction={decorate}] (left.center) -- (middle.center);
      \draw[postaction={decorate}] (middle.center) -- (right.center);
       \draw[postaction={decorate}] (left.center) -- (right.center);
       
       \node at (-1,1.3){$A$};
          \node at (1,1.3){$C$};
          \node at (0, -0.5) {$B$};

        \centerarc[->](-1.5,0)(46:2:.6)
         \node at (-.8, .4) {$\alpha$};
         
         \centerarc[->](1.5,0)(180:130:0.6)
         \node at (0.8, 0.4) {$\beta$};
         
         \centerarc[->](0,2)(300:240:0.6)
         \node at (0,1.2) {$\gamma$}; 

	 \node at (3.5,1){$\Longleftrightarrow$};

         \node at (7,1){
	$A\buildrel \alpha \over \lra  B \buildrel \beta \over \lra C\buildrel \gamma \over \lra A[1].$
	};
   \end{tikzpicture}
\end{equation}

That is, we assign objects to oriented edges of geometric triangles, and morphisms
to their angles. A morphism of degree $1$ is represented by an angle formed by two edges with
different directions (one incoming, one outgoing). The advantage of this dual point of view is that
the most fundamental types of diagrams are now represented by collections of geometric triangles of
the most basic shapes.

\begin{ex}\label{ex:octahedron-flip}
The two halves of an octahedron are represented by two triangulations of a 4-gon.
The octahedral axiom is now interpreted as switching from one triangulation to the other (flip):

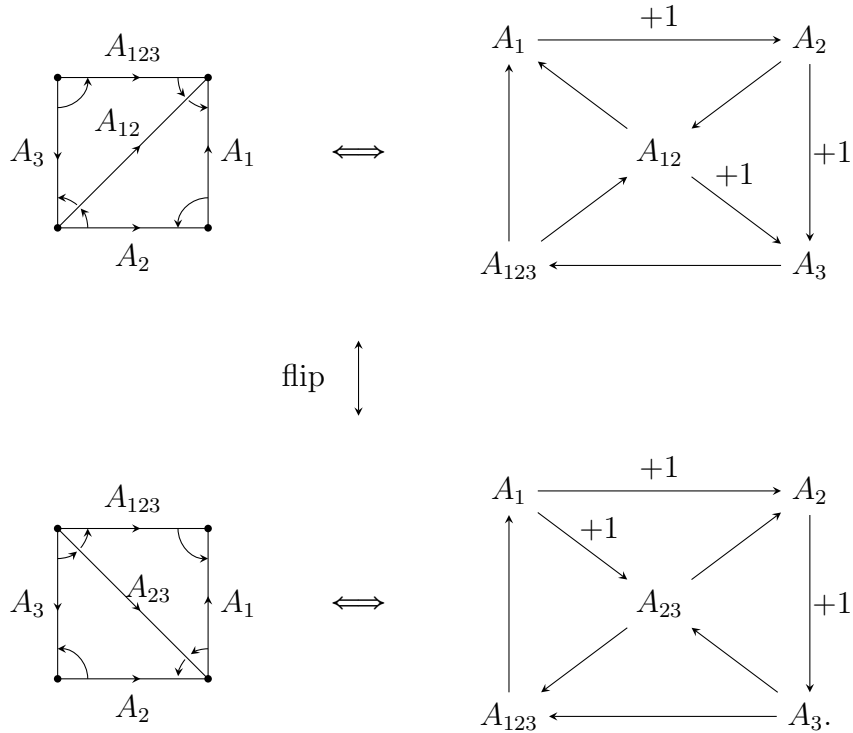
\centering
\begin{tikzpicture}[>=stealth,baseline=(current  bounding  box.center),decoration={
    markings,
    mark=at position 0.55 with {\arrow{>}}}
    ] 
	\begin{scope}[xshift=0cm,yshift=-.5cm]
\node (1) at (0,0){}; 
 \fill (1) circle (0.05);
 \node (2) at (0,-2){};
 \fill (2) circle (0.05);
 \node (3) at (-2,0){};
 \fill (3) circle (0.05);
 \node (4) at (-2,-2){};
 \fill (4) circle (0.05);
 
 \draw[postaction={decorate}] (2.center) -- (1.center); 
 \draw[postaction={decorate}] (4.center) -- (2.center); 
 \draw[postaction={decorate}] (3.center) -- (4.center); 
 \draw[postaction={decorate}] (4.center) -- (1.center); 
 \draw[postaction={decorate}] (3.center) -- (1.center); 
 
 \node at (0.4, -1){$A_1$}; 
 \node at (-1, -2.4) {$A_2$}; 
 \node at (-2.4, -1) {$A_3$}; 
 \node at (-1, 0.4) {$A_{123}$}; 
 \node at (-1.2, -0.6) {$A_{12}$}; 
 
 \centerarc[->](0,-2)(90:180:0.4)
 \centerarc[->](-2,-2)(0:40:0.4)
 \centerarc[->](0,0)(230:270:0.4)
 \centerarc[->](0,0)(180:220:0.4)
 \centerarc[->](-2,-2)(50:90:0.4)
 \centerarc[->](-2,0)(270:360:0.4)

 \node at (2, -1) {$\Longleftrightarrow$}; 

 \draw[<->] (2,-3.5) -- (2,-4.5); 
 \node[left] at (1.7,-4){\text{flip}};
 
\end{scope} 
\begin{scope}[xshift=1cm]
 
 \node (A1) at (3,0){$A_1$}; 
 \node (A2) at (7,0) {$A_2$}; 
 \node (A123) at (3, -3) {$A_{123}$}; 
 \node (A3) at (7, -3) {$A_3$}; 
 \node (A12) at (5, -1.5) {$A_{12}$}; 
 
 \draw[->] (A1) -- (A2); 
 \draw[->] (A2) -- (A3);  
  \draw[->] (A3) -- (A123); 
   \draw[->] (A123) -- (A1); 
 \draw[->] (A2) -- (A12); 
  \draw[->] (A123) -- (A12); 
   \draw[->] (A12) -- (A1); 
    \draw[->] (A12) -- (A3); 
 
 \node at (5,0.3) {+1};
 \node at (7.3, -1.5) {+1};
 \node at (6, -1.8) {+1};

\end{scope}
 
\begin{scope}[xshift=0cm,yshift=-6.5cm]

\node (1) at (0,0){}; 
 \fill (1) circle (0.05);
 \node (2) at (0,-2){};
 \fill (2) circle (0.05);
 \node (3) at (-2,0){};
 \fill (3) circle (0.05);
 \node (4) at (-2,-2){};
 \fill (4) circle (0.05);
 
 \draw[postaction={decorate}] (2.center) -- (1.center); 
 \draw[postaction={decorate}] (4.center) -- (2.center); 
 \draw[postaction={decorate}] (3.center) -- (4.center); 
 \draw[postaction={decorate}] (3.center) -- (1.center); 
 \draw[postaction={decorate}] (3.center) -- (2.center); 
 
 \node at (0.4, -1){$A_1$}; 
 \node at (-1, -2.4) {$A_2$}; 
 \node at (-2.4, -1) {$A_3$}; 
 \node at (-1, 0.4) {$A_{123}$}; 
 \node at (-0.8, -0.8) {$A_{23}$}; 
 
 \centerarc[->](0,-2)(90:125:0.4)
 \centerarc[->](0,-2)(140:175:0.4)

 \centerarc[->](0,0)(180:270:0.4)
 \centerarc[->](-2,-2)(0:90:0.4)
 \centerarc[->](-2,0)(270:310:0.4)
  \centerarc[->](-2,0)(320:360:0.4)
 
 \node at (2, -1) {$\Longleftrightarrow$}; 

\end{scope}
 
\begin{scope}[xshift=1cm,yshift=-6cm]
 
  \node (A1) at (3,0){$A_1$}; 
 \node (A2) at (7,0) {$A_2$}; 
 \node (A123) at (3, -3) {$A_{123}$}; 
 \node (A3) at (7, -3) {$A_3.$}; 
 \node (A23) at (5, -1.5) {$A_{23}$}; 
 
 \draw[->] (A1) -- (A2); 
 \draw[->] (A2) -- (A3);  
  \draw[->] (A3) -- (A123); 
   \draw[->] (A123) -- (A1); 
 \draw[->] (A23) -- (A123); 
  \draw[->] (A23) -- (A2); 
   \draw[->] (A1) -- (A23); 
    \draw[->] (A3) -- (A23); 
 
 \node at (5,0.3) {+1};
 \node at (7.3, -1.5) {+1};
 \node at (4.2, -0.5) {+1};

 \end{scope}
\end{tikzpicture}
\end{ex}

\begin{ex} A {\em Postnikov system} in $\Dc$ is a diagram of exact triangles representing an object
$A_{12...n}$ as an iterated extension of (``tower of fibrations" with fibers being) the
given objects $A_1, ..., A_n$, see \cite{gelfand-manin}, Ch. 4, \S 2.  Note that there are
several possible types of Postnikov systems, see {\em loc. cit}. In our approach, these
correspond to different triangulations of the $(n+1)$-gon. The octahedral axiom thus
allows us to pass from any one type to any other by a sequence of flips on $4$-gons.

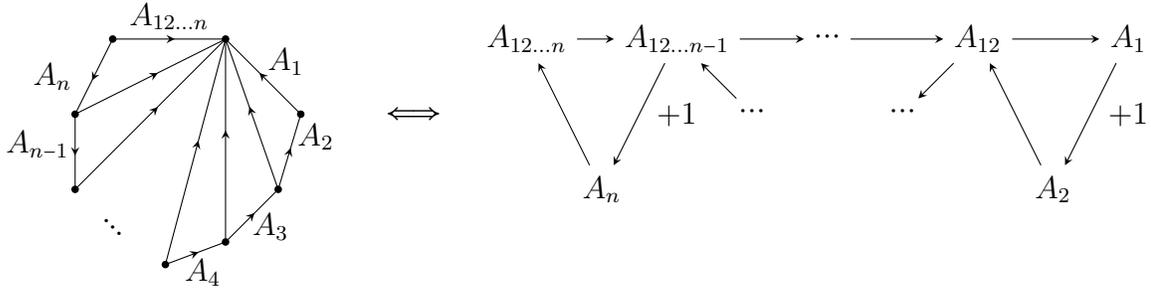
\vspace*{.5cm}
\begin{tikzpicture}[>=stealth,baseline=(current  bounding  box.center),decoration={
    markings,
    mark=at position 0.55 with {\arrow{>}}}
    ] 
\node (P0) at (0:0){};
 \fill (P0) circle (0.05);
 \node (P1) at (1,-1){};
 \fill (P1) circle (0.05);
  \node (P2) at (0.7, -2){}; 
 \fill (P2) circle (0.05);
  \node (P3) at (0, -2.7){};
 \fill (P3) circle (0.05);
   \node (P4) at (-.8, -3){};
 \fill (P4) circle (0.05);
   \node (dots) at (-1.5, -2.5){$\ddots$};
   \node(Pn-2) at (-2, -2){};
 \fill (Pn-2) circle (0.05);
   \node(Pn-1) at (-2, -1){};
 \fill (Pn-1) circle (0.05);
   \node (Pn) at (-1.5,0){};
 \fill (Pn) circle (0.05);

  \draw[postaction={decorate}] (P1.center) -- (P0.center); 
\draw[postaction={decorate}] (P2.center) -- (P1.center); 
\draw[postaction={decorate}] (P3.center) -- (P2.center); 
\draw[postaction={decorate}] (Pn.center) -- (P0.center); 
\draw[postaction={decorate}] (Pn.center) -- (Pn-1.center); 
\draw[postaction={decorate}] (Pn-1.center) -- (Pn-2.center); 
 \draw[postaction={decorate}] (P2.center) -- (P0.center); 
  \draw[postaction={decorate}] (P3.center) -- (P0.center); 
   \draw[postaction={decorate}] (Pn-2.center) -- (P0.center); 
    \draw[postaction={decorate}] (Pn-1.center) -- (P0.center); 
     \draw[postaction={decorate}] (P4.center) -- (P3.center); 
     
     \draw[postaction={decorate}] (P4.center) -- (P0.center);

\node at (-0.75, .3){$A_{12...n}$};
\node at (0.8, -0.3){$A_1$}; 
\node at (1.2, -1.3){$A_2$};
\node at (0.6, -2.5){$A_3$};
\node at (-0.3, -3.1){$A_4$};
\node at (-2.3, -0.5){$A_n$};
\node at (-2.5, -1.4){$A_{n-1}$}; 

\node (A12n) at (4,0){$A_{12...n}$}; 
\node (A12n-1) at (6,0){$A_{12...n-1}$}; 
\node (cdots) at (8,0){$\cdots$}; 
\node (A12) at (10,0){$A_{12}$};
\node(A1) at (12,0){$A_1$};
\node (An) at (5,-2){$A_n$};
\node (meddots) at (7,-1){$\cdots$};
\node (other dots) at (9,-1){$\cdots$}; 
\node (A2) at (11, -2){$A_2$}; 

\draw[->] (A12n) -- (A12n-1); 
\draw[->] (A12n-1) -- (cdots); 
\draw[->](cdots) -- (A12); 
\draw[->] (A12) -- (A1); 
\draw[->] (A1) -- (A2); 
\draw[->] (A2) -- (A12); 
\draw[->] (An) -- (A12n); 
\draw[->] (A12n-1) -- (An); 
\draw[->]  (meddots) -- (A12n-1); 
\draw [->] (A12) -- (other dots); 

\node at (12, -1){+1};
\node at (6, -1){+1}; 

\node at (2.5, -1){$\Longleftrightarrow$};

\end{tikzpicture}
\end{ex}
 
\paragraph{}
This 2-dimensional symmetry becomes even more pronounced, if $\Dc$ is 2-periodic, i.e., the shift
functor $\Sigma: A\mapsto A[1]$ squares to the identity. In this case we can freely switch the
directions of edges in the geometric triangle representing an exact one as above, by postulating
that such switches amount to applying $\Sigma$:
  
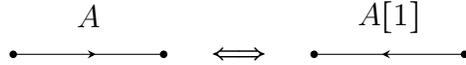
\begin{center}
  \begin{tikzpicture}[>=stealth,baseline=(current  bounding  box.center),decoration={
    markings,
    mark=at position 0.55 with {\arrow{>}}}]
   
\node (P0) at (0,0){};
 \fill (P0) circle (0.05);
 \node (P1) at (2,0){};
  \fill (P1) circle (0.05);
  \node (P2) at (4,0){};
  \fill (P2) circle (0.05);
\node (P3) at (6,0){};
  \fill (P3) circle (0.05);
  
  \draw[postaction={decorate}]  (P0.center) -- (P1.center);
  \draw[postaction={decorate}] (P3.center) -- (P2.center);
  
  \node at (1,0.5) {$A$}; 
  \node at (5,0.5){$A[1]$};
  
  \node at (3,0) {$\Longleftrightarrow$}; 
   
  \end{tikzpicture}

 \end{center}
The really important remaining datum is purely 2-dimensional: it is the orientation of the
geometric triangle itself, which determines the directions of the morphisms between the objects on
its edges. We can therefore consider {\em surface Postnikov systems}: diagrams of exact triangles
in $\Dc$, whose associated geometric triangles form a curvilinear triangulation $\Tc$ of an
oriented topological surface $S$, possibly with boundary.
 
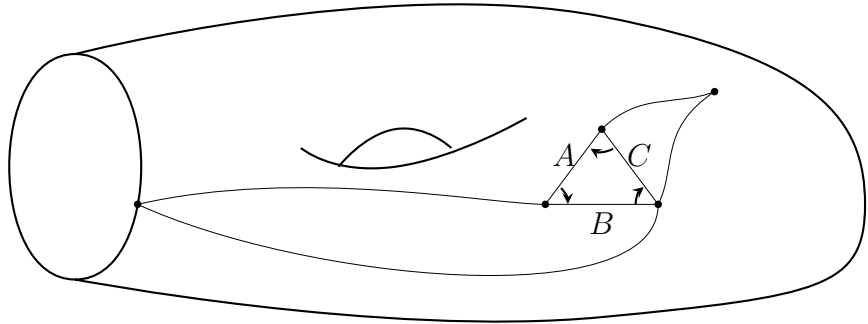
\begin{center}
   \begin{tikzpicture}[>=stealth,baseline=(current  bounding  box.center),decoration={
    markings,
    mark=at position 0.55 with {\arrow{>}}}, thick,scale=0.5]

  \filldraw[even odd rule,fill=gray!25, fill opacity=0]
 plot[smooth,tension=1] coordinates {(-14,4)  (0,5)     (7,0) 
  (0,-3)   (-14, -2)}
 
 (-14,4) to[out=0,in=0] (-14, -2)
 
 plot[smooth, tension=1] coordinates{(-7,1) (-5.5, 2) (-4, 1.5)}
 
 plot[smooth, tension=1] coordinates{(-8,1.5) (-5.5, 1) (-2, 2.3)};
 
 \draw  (-14,-2) to[out=180,in=180] (-14,4);

 \node (left) at (-1.5,0) {};
    \node (middle) at (0,2) {};
    \node (right) at (1.5,0) {};
    
      \fill (left) circle (0.1);
    \fill (middle) circle (0.1);
    \fill (right) circle (0.1);  
    
    \draw[line width=.1mm] (left.center) -- (middle.center);
      \draw[line width=.1mm] (middle.center) -- (right.center);
       \draw[line width=.1mm] (left.center) -- (right.center);
       
       \node at (-1,1.3){$A$};
          \node at (1,1.3){$C$};
          \node at (0, -0.5) {$B$};

        \centerarc[->](-1.5,0)(46:1:.6)

         \centerarc[->](1.5,0)(180:130:0.6)
                  
         \centerarc[->](0,2)(300:240:0.6)
         
         \node (N1) at (3,3){};
         \fill (N1) circle (0.1);

         \draw[line width=.1mm] (middle.center) .. controls (1,3) and (2,2.6) .. (N1.center);
         
         \draw[line width=.1mm] (right.center) .. controls (2,1) and (1.5,2) .. (N1.center);
         
         \node (N2) at (-12.34,0){};
         \fill (N2) circle (0.1);
         
         \draw[line width=.1mm] (N2.center) .. controls (-8,1) and (-3,0) .. (left.center);
         
         \draw[line width=.1mm] (N2.center) .. controls (-8,-2) and (1.5,-3) .. (right.center);

 \end{tikzpicture}
 \end{center}

Ordinary Postnikov systems are obtained when $S$ is a disk and all the vertices of $\Tc$ are on
$\partial S$. Standard results of Teichm\"uller theory imply that any two triangulations of $S$
with the same underlying set of vertices $M$ are connected by a sequence of flips. This suggests that an
appropriately defined classifying space of surface Postnikov systems depends, in a very canonical
way, only on the oriented surface $(S,M)$ and not on a chosen triangulation $\Tc$, in particular, 
that it is acted upon by the mapping class group of $(S,M)$. In the present paper we make this
statement precise and provide a proof. The resulting theory turns out to be related to subjects 
such as Fukaya categories, matrix factorizations and mirror symmetry. 
 
In order to have good classifying spaces of exact diagrams in $\Dc$, it seems
unavoidable to assume that $\Dc$ comes with an {\em enhancement}, a certain refinement of the
graded abelian groups $\Hom_\Dc(A, \Sigma^\bullet B)$. In this paper we mostly work with
dg-enhancements (\S \ref{subsec:morita}) which allows us to use techniques from the Morita homotopy
theory of dg-categories \cite{tabuada, toen-morita} such as model structures, simplicial mapping
spaces, homotopy limits, etc.

In this setting, given any triangulation $\Tc$ of $(S,M)$, we can form the {\em universal Postnikov
system} of type $\Tc$ which is a $2$-periodic dg-category $L\E^\Tc$ with the following ``universal
property'': Given any perfect $2$-periodic dg category $\Ac$, enhancing a triangulated category $\Dc$,
the classifying space of surface Postnikov systems of type $\Tc$ with values in $\Ac$ is given as
the simplicial mapping space 
\begin{equation}\label{eq:post}
    \on{Post}^\Tc(\Ac) := \Map(L\E^\Tc, \Ac)
\end{equation}
in the category of $2$-periodic dg categories, localized along Morita equivalences.

Our main result says that, up to Morita equivalence, $L\E^\Tc$ does not depend on $\Tc$, so that we
obtain an object
\[
\F^{(S,M)} \simeq L\E^\Tc \in \Hmo^{(2)}. 
\]
which, up to unique isomorphism, only depends on $(S,M)$. Here, $\Hmo^{(2)}$ is the Morita homotopy
category of $2$-periodic dg-categories. In particular, the mapping class group of $(S,M)$ acts on
$\F^{(S,M)}$ by automorphisms in $\Hmo^{(2)}$. 

We can refine construction \eqref{eq:post} to form the
{\em classifying dg-category of Postnikov systems of type $\Tc$ in $\Ac$} 
\[
	\underline{\on{Post}}^\Tc(\Ac) \,\,:=\,\, \IRHom(L\E^\Tc, \Ac)
\]
where $\IRHom$ denotes To\"en's internal Hom for the category of dg-categories. For the same reasons
as above, this $2$-periodic dg-category is acted upon by the mapping class group of the surface
$(S,M)$. In fact, in both cases, the action of the mapping class group is coherent in the sense of
homotopy theory.
 
As pointed out to us by M. Kontsevich, the dg-category $\F^{(S,M)}$ is nothing but a version of the Fukaya category of
the surface $S-(M\cap S^\circ)$ obtained by removing the points of $M$ lying in the interior of $S$.
The representation of $\F^{(S,M)}$ as $L\E^\Tc$ provides a rigorous implementation of an instance of 
his program of ``localizing the Fukaya category along a singular Lagrangian spine" 
\cite{kontsevich-miami, kontsevich-sym-homol}.
More generally, he considered a $2d$-dimensional symplectic manifold $(U,\omega)$ which can be
contracted onto a possibly singular Lagrangian subvariety $L \subset U$ by the flow along a vector
field $\xi$ satisfying  $\Lie_\xi(\omega)=-\omega$.  In such a situation he suggested to construct a
``cosheaf of dg-categories" $\Phi_L$ on $L$, refining the Fukaya category $\Fc(U)$, which should be
recovered as the category of global sections $\Phi_L(L)$. In particular, different choices of $L$ 
should lead to different realizations of $\Fc(U)$. 

Our situation corresponds to the simplest case $d=1$ when $U=S-M$, where $(S,M)$ is a marked surface
with $\partial S=\emptyset$, which we consider as a symplectic manifold with respect to some 2-form $\omega$.
A triangulation $\T$ of $(S,M)$ gives then a 3-valent {\em dual graph} $L\subset U$ defined up to
isotopy and Lagrangian because $\dim(L)=1$. 
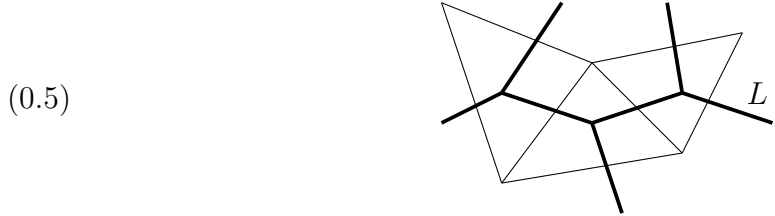
\begin{equation}
\label{eq:dual graph}           
\begin{tikzpicture}[scale=0.4, baseline=(current  bounding  box.center)]

\draw (0,0) -- (6,1); 
\draw (6,1) -- (3,4);
\draw (3,4) -- (0,0); 
\draw (6,1) -- (8,5); 
\draw (3,4) -- (8,5); 
\draw (3,4) -- (-2,6);
\draw (-2,6) -- (0,0); 

\draw[line width=0.5mm] (6,3) -- (9,2); 
\draw[line width=0.5mm] (6,3) -- (5.5,6); 
\draw[line width=0.5mm] (6,3) -- (3,2); 
\draw[line width=0.5mm] (3,2) -- (4, -1); 
\draw[line width=0.5mm] (3,2) -- (0,3); 
\draw[line width=0.5mm] (0,3) -- (-2,2); 
\draw[line width=0.5mm] (0,3) -- (2,6); 

\node at (8.5,3){$L$};
\end{tikzpicture}
\end{equation} 

Our $L\E^\Tc$ corresponds to $\Phi_L(L)$. Further, the ``local" nature of $\Phi_L$ in Kontsevich's proposal
corresponds to our construction of $L\E^\Tc$ by gluing it out of local data, a certain system of 2-periodic 
dg-categories $\E^\bullet = (\E^n)_{n\geq 0}$ such that:
\begin{enumerate}
\item $\E^\bullet$ is a cocyclic object, in the sense of A. Connes \cite{connes}, in the category
	$\dgcat^{(2)}$ of 2-periodic dg-categories. In particular, the group $\ZZ/(n+1)$ acts on
	$\E^n$ by automorphisms of dg-categories. 
\item For every $n \ge 0$, the dg-category $\E^n$ is Morita equivalent to a dg-enhancement of
	$D^{(2)}(A_n\mod)$, the 2-periodic derived category of representations of the quiver $A_n$.
	The action of the generator of $\ZZ/(n+1)$ corresponds to the Coxeter functor. 
\item The cosimplicial object underlying $\E^\bullet$ is 2-coSegal in the sense of \cite{HSS1}.
\end{enumerate}
Property (1) is responsible for the fact that the construction $L\E^\Tc$ does only depend on the
orientation of each triangle of $\Tc$ induced from the orientation of the surface $S$ 
and not on any particular orientations of its edges.
Condition (2) corresponds to the requirement in \cite{kontsevich-miami} that the stalk of $\Phi_L$
at a ramification point of $L$ with valency $n+1$ should be a version of $D(A_n\mod)$. 
The 2-coSegal property (3) ensures the coherent independence of $L\E^\Tc$ 
on $\Tc$ (or, equivalently, of $\Phi_L$ on $L$).

We call the dg-category $\F^{(S,M)}$ the {\em topological coFukaya category of the marked
oriented surface $(S,M)$}. Dually, for any perfect $2$-periodic dg-category $\Ac$, the dg-category
\[
  \IRHom(\F^{(S,M)}, \Ac)
\]
is called the {\em topological Fukaya category of $(S,M)$ with coefficients in $\Ac$}. In the case
where $\Ac$ is the dg category $\Perf_{\k}^{(2)}$ of $2$-periodic perfect complexes of $\k$-vector spaces, 
we introduce the notation 
\[
  \F_{(S,M)} = \IRHom(\F^{(S,M)}, \Perf_{\k}^{(2)})
\]
and refer to this category as the {\em topological Fukaya category of $(S,M)$}. In other words, the
dg-category $\F_{(S,M)}$ is the Morita dual of $\F^{(S,M)}$.
The terminology is chosen to reflect the descent properties of these constructions: The choice of a
spanning Ribbon graph $\Gamma$ of the surface $(S,M)$ can be regarded as a combinatorial way of
encoding an open covering of the surface. The Morita equivalences
\begin{align}
	\label{eq:cosheaf} \F^{(S,M)} & \simeq L\E^{\Gamma} \simeq \hoind^{\dgcatt}_{\{\Lambda^n \to \Lambda^{\Gamma}\}}
  \E^n\\
  \label{eq:sheaf} \F_{(S,M)} & \simeq R\E_{\Gamma} \simeq \hopro^{\dgcatt}_{\{\Lambda^n \to \Lambda^{\Gamma}\}}
  \E_n
\end{align}
are immediate by our construction of the topological Fukaya category as a homotopy Kan extension,
and assign a precise meaning to the statement that the topological (co)Fukaya category is a
homotopy (co)sheaf with values in dg-categories. 
The homotopy limits in \eqref{eq:cosheaf} and \eqref{eq:sheaf} are taken with respect to the Morita
model structure and can be effectively computed using standard techniques from the theory of 
model categories. We illustrate this in \S \ref{subsection:fukayaexamples} where we investigate some examples 
appearing on Kontsevich's list \cite[Pictures]{kontsevich-sym-homol}.

For our constructions to work, it is crucial that the system $\E^\bullet$ of dg-categories
satisfies conditions (1),(2), and (3) above. Note that the most immediate dg-enhancements of
$D^{(2)}(A_n\mod)$ do not have manifest cyclic symmetry. From the symplectic point of view, it
is known that $D^{(2)}(A_n{\textrm{-}} \on{mod})$ is the ``Fukaya-Seidel category of the unit disk
$|z|\leq 1$ equipped with the potential $z^{n+1}$" (the 1-dimensional $A_n$-singularity).  However,
one does not obtain a suitable definition of $\E^n$ on this path either. Indeed, the definition of
Seidel \cite{seidel-book} requires choosing, first, a deformation of the singularity, i.e., a
generic polynomial $f(z)= z^{n+1}+\sum_{i=0}^n a_i z^i$ and, second, an ordered basis of
($0$-dimensional) vanishing cycles of $f$, which, again, breaks the cyclic symmetry.

Instead, we define $\E^n$ in terms of matrix factorizations, in the sense of D. Eisenbud, of
$z^{n+1}$, slightly modifying the setup of \cite{takahashi}. Our motivation for this approach is
that this matrix factorization category can be interpreted as the homological mirror of the above
mentioned Fukaya-Seidel category, thus mirror symmetry between Landau-Ginzburg models and matrix
factorizations is locally built into our constructions from the very outset. Our construction uses a
new concept of {\em loop factorization} in V. Drinfeld's $\L$-categories and is explained in detail in \S
\ref{section:loop}.

The concept of a 2-Segal object was introduced in \cite{HSS1} as a unifying concept for various
situations when some object is defined in terms of a choice of a triangulation but ends up not
depending on this choice in a coherent way. In the case of 2-Segal simplicial objects, treated in
{\em loc. cit.}, we deal with triangulations of plane polygons and related instances of
associativity, such as, e.g., in the context of Hall algebras. 
The example that motivated our study of 2-Segal spaces in \cite{HSS1} was the {\em Waldhausen
S-construction}, a simplicial space which plays a fundamental role in algebraic K-theory, see
\cite{gillet}. In \cite{HSS1} we introduced a generalization of the S-construction encompassing
arbitrary stable $\infty$-categories \cite{lurie.stable, lurie.algebra}. The present work grew out
of our heuristic observation that for $2$-periodic perfect dg categories, the S-construction 
 has   a cyclic, and not just a simplicial structure.  
Passing from simplicial to cyclic objects allows one to extend the polygon triangulations to
triangulations of arbitrary marked oriented surfaces in a non-ambiguous way which, applied to the
S-construction, leads to a precise variant of the surface Postnikov systems described above.  The
relevant constructions for the present work are provided in \S \ref{section:cyclicsegal}. A more
detailed account of the general theory will be given in \cite{HSS2}. 

The starting point of this project was a suggestion of J. Lurie to rigorously establish the
additional cyclic symmetry of the S-construction by constructing a cocyclic dg-category which
corepresents it in $\dgcatt$. 
The object $\E^\bullet$ provides a solution, in the sense that, given a $2$-periodic perfect
dg-category $\Ac$, the simplicial space underlying the cyclic space $\Map(\E^{\bullet}, \Ac)$ is
weakly equivalent to the Waldhausen S-construction of $\Ac$. This relies on a comparison result
between pre-triangulated dg-categories and stable $\infty$-categories which has been carried out by
G. Faonte \cite{faonte}. A more detailed analysis will be given in \cite{HSS2}.

In conclusion, we find it remarkable that the observation
\[
\left( {\text{Axioms of homological}\atop\text{algebra}}\right) \,\,\longleftrightarrow\,\,
\left( {\text{Flips of}\atop\text{ $2d$ triangulations} }\right)
\]
naturally leads to a topological variant of the Fukaya category. This phenomenon seems to be
potentially appealing even to someone with no symplectic motivation whatsoever.

We would like to point out that there have been various projects addressing Kontsevich's
localization program for $2$-dimensional symplectic manifolds. We refer the reader to
\cite{sibilla-treumann-zaslow} and the references therein.  In higher dimensions, the general
problem of localization on a given spine is treated in \cite{nadler-higher}.  Very recently, a
construction similar to ours has been given in the context of $A_{\infty}$-categories \cite{nadler}.
The main novelty in our approach is the $2$-Segal property which reflects, in a conceptually clear
way, the fact that the category we construct is a {\em topological} invariant of the marked surface,
coherently independent of a chosen spine.
Finally, we wish to mention that, as we were informed by J. Lurie, he himself has, in joint work
with A. Preygel, found a coparacyclic version of the cocyclic object $\E^{\bullet}$ which is
suitable for an analysis from the point of view of $\infty$-categories and relates to classical
concepts from homotopy theory such as the J-homomorphism.\\ 

{\bf Acknowledgements.}
We are very grateful to J. Lurie for his interest and for many inspiring discussions regarding the
theory of higher Segal spaces in general. In particular, his proposal of corepresenting the
S-construction via a cocyclic object was a key idea in this project.
We would further like to thank A. Goncharov, and M. Kontsevich for important conversations which
influenced the direction of this work. Finally, we would like to acknowledge a conversation
involving D. Ben-Zvi, J. Lurie, D. Nadler, and A. Preygel, in which ideas related to this work were
discussed. The work of T.D. was supported by a Simons Postdoctoral Fellowship. The work of M.K. was
partially supported by an NSF grant and parts of it were carried out during visits to the
Max-Planck-Institut f\"ur Mathematik in Bonn and to Universit\'e Paris-13, whose hospitality and
financial support are gratefully acknowledged.        

\vfill\eject

\numberwithin{equation}{subsection}
  
\section{Background on the homotopy theory of dg-structures}
     
\subsection{Model structures on the category of differential $\ZZ$-graded categories. }
\label{subsec:morita}
 
Let $\k$ be a field and $\Vect_\k^\ZZ$ be the category of $\ZZ$-graded $\k$-vector spaces.
We denote by $\Sigma^n, n\in\ZZ$, the functor of shift of grading: $(\Sigma^n V)^i=V^{i+n}$.
We denote by $\Mod_\k$ the category of cochain complexes of $\k$-vector spaces.
The usual tensor product of complexes makes $\Mod_\k$ into a symmetric monoidal category,
and the shift functor $\Sigma$ is defined by $\Sigma(V^\bullet) = \Sigma(\k)\otimes V^\bullet$,
where  $\Sigma(\k)$ is the vector space $\k$ in degree $(-1)$ with zero differential. 
By a ($\ZZ$-graded) {\em dg-category} we will mean a category $\Ac$ enriched in the symmetric monoidal
category $\Mod_\k$. Note that $\Mod_\k$ itself is a dg-category. 
We denote by $\dgcat$ the category formed by small $\k$-linear dg-categories and their dg-functors.
The category $\dgcat$ has a symmetric monoidal structure $\otimes$ given by the tensor product 
$\Ac\otimes \Bc$ of dg-categories $\Ac$ and $\Bc$:
\[
\begin{gathered}
\Ob(\Ac\otimes\Bc) \,\,=\,\,\Ob(\Ac)\times\Ob(\Bc),\\
\Hom^\bullet_{\Ac\otimes\Bc} \bigl( (x,y), (x',y')\bigr) \,\,=\,\,
\Hom^\bullet_\Ac (x,x')\otimes_\k \Hom^\bullet_\Bc(y, y').
\end{gathered}
\]
Recall that for dg-categories $\Ac, \Bc$ the category $\IHom(\Ac,\Bc)$
of dg-functors $\Ac\to\Bc$
is naturally a dg-category 
so that we have an adjunction
\[
\Hom_{\dgcat}(\Ac, \IHom(\Bc, \Cc)) \cong \Hom_{\dgcat}(\Ac\otimes \Bc, \Cc). 
\]

\begin{Defi} A dg-functor $f: \Ac\to\Bc$ of dg-categories is called:
\begin{itemize}
\item {\em fully faithful}, resp. {\em quasi-fully faithful}, if for any $x,y\in\Ob(\Ac)$ the morphism of complexes
\[
f_{x,y}:\Hom_\Ac^\bullet(x,y) \lra\Hom_\Bc(f(x), f(y))
\]
is an isomorphism, resp. a quasi-isomorphism. 

\item a {\em quasi-isomorphism}, resp. {\em quasi-equivalence}, if $H^\bullet(f): H^\bullet(\Ac)\to H^\bullet(\Bc)$
is an isomorphism, resp.  an equivalence of graded $\k$-linear categories. 
\end{itemize}
\end{Defi} 
 
Dg-functors $\Ac^\op \to \Mod_\k$ will be called (right) {\em dg-modules} over $\Ac$, and the
dg-category formed by them will be denoted $ \Mod_\Ac$. We will use the dg-version of the Yoneda
embedding
\begin{equation}
	\Upsilon_\Ac: \Ac \lra \Mod_\Ac,
\end{equation}
which is a fully faithful dg-functor. 
For background on model categories, see, e.g., \cite{hovey}, \cite{dhks} and Appendix A to \cite{lurie.htt}.
The category $\dgcat$ carries two model structures introduced by Tabuada \cite{tabuada}.
The first one, which we call the {\em quasi-equivalence model structure}, is characterized as follows:
\begin{itemize}
\item[(QW)] Weak equivalences are quasi-equivalences.

\item[(QF)] Fibrations are dg-functors $f: \Ac\to\Bc$ such that:
\begin{enumerate}
  \item $f$ is surjective on $\Hom$-complexes.
  \item For any $x\in\Ob(\Ac)$ and any homotopy equivalence $v: f(x)\to z$ in $\Bc$,
  there is a homotopy equivalence $u: x\to y$ in $\Ac$ such that $f(u)=v$ 
  (in particular, $f(y)=x$). 
\end{enumerate}

\item[(QC)] Cofibrations are defined by the left lifting property with respect to trivial fibrations. 
\end{itemize}
The initial object in $\dgcat$ is the empty dg-category $\emptyset$ (no objects).
The final object is the zero dg-category $\bf 0$ with one object $\pt$ and $\Hom^\bullet(\pt, \pt)=0$.
Note that (QF) implies that every dg-category is fibrant. Let $\Qeq$ be the class of
quasi-equivalences in $\dgcat$, and let
\[
\Heq = \dgcat[\Qeq^{-1}]
\]
denote the homotopy category of the quasi-equivalence model structure. 

It follows from the results of  To\"en \cite{toen-morita} that $\otimes$ defines a {\em closed}
symmetric monoidal structure on $\Heq$, so that we have dg-categories $\IRHom(\Ac, \Bc)$ together
with natural isomorphisms (in $\Heq$)
\begin{equation}\label{eq:toen-rhom1}
\Hom_{\Heq}(\Ac\otimes \Bc, \Cc)\,\,\cong \,\,
\Hom_{\Heq} (\Ac,\IRHom(\Bc, \Cc)). 
\end{equation}
More precisely, To\"en considers the situation when $\k$ is allowed to be 
an arbitrary commutative ring and uses $\otimes^L$, the derived
functor of $\otimes$. In our case when $\k$ is a field,
$\otimes$ preserves quasi-equivalences and hence does not need to be derived. 
Note that the internal Hom is not obtained as a derived functor, in the sense of model categories,
of the bifunctor $\IHom(\Ac,\Bc)$, since the latter does not take quasi-equivalences of cofibrant
dg-categories into quasi-equivalences (see {\em loc. cit.} p.631). 
By the main result of \cite{toen-morita}, the dg-category $\IRHom(\Ac, \Bc)$ can be explicitly
described as the full dg-subcategory of $\Mod_{\Ac^{\op} \otimes \Bc}$
formed by those dg-modules $M$ satisfying:
\begin{enumerate}
\item $M$ is cofibrant.
\item $M$ is {\em right quasi-representable}, i.e., for each $x\in\Ob(\Ac)$ the right dg-module
\[
M(x, -): \Bc^\op\lra \Mod_\k, \quad y\mapsto M(x,y)
\]
is quasi-isomorphic to a representable dg-module $\Upsilon_\Bc(f(x))$ for some object
$f(x)\in\Bc$. 
\end{enumerate}

Recall that, as any model category, $\dgcat$ is equipped with {\em simplicial mapping spaces}
$\Map(\Ac, \Bc)$, obtained by Dwyer-Kan localization. However, as $\dgcat$ is not known to carry a
simplicial model structure in the sense of \cite{hovey}, the computation of $\Map(\Ac, \Bc)$ is
non-trivial.
It was shown in \cite{toen-morita} that the mapping spaces can be computed via simplicial framings,
leading to the explicit formula
\[
  \Map(\Ac, \Bc) \simeq N(\Mod^{\on{rqr}}_{\Ac^{\op} \otimes \Bc}, \mathsf{W}), 
\]
the nerve of the category formed by all right quasi-representable $\Ac^{\op} \otimes \Bc$-modules and
their weak equivalences. The isomorphisms \ref{eq:toen-rhom1} can then be refined to the adjunction
\begin{equation}\label{eq:simplicial-adjunction}
	\Map(\Ac\otimes\Bc, \Cc) \simeq \Map(\Ac, \IRHom(\Bc, \Cc))
\end{equation}
of simplicial mapping spaces (see \cite{toen-morita}).

Let $x, z$ be objects of a dg-category $\Ac$, and $m\in\ZZ$. We say that $z$ is realized as an
$m$-{\em fold shift} of $x$, and write $z\simeq\Sigma^m x$, if we are given an isomorphism of
dg-functors
\[
	\Hom^\bullet(-, z) \lra \Sigma^m \Hom^\bullet(-, x), \quad \Ac^\op\to C_\k.
\]
Note that $\Sigma^m x$, if exists, is defined uniquely up to a unique isomorphism. 
 
We recall (e.g., \cite{toen-vaquie}) that $\Mod_\Ac$, equipped with the projective model
structure, is a $\Mod_\k$-enriched model category in which weak equivalences are quasi-isomorphisms
of dg-modules, and all objects are fibrant. We denote by $\Mod_\Ac^\circ\subset\Mod_\Ac$ the full
dg-subcategory of cofibrant (and automatically fibrant) objects. We also denote by
\[
	D(\Ac) \,\,=\,\, H^0(\Mod_\Ac) \,\,=\,\,\Mod_\Ac[\Qis^{-1}]
\] 
the homotopy category of $\Mod_\Ac$, which is commonly called the {\em derived category} of $\Ac$.
Thus, we have an equivalence
\[
D(\Ac) \,\,\simeq \,\, H^0(\Mod_\Ac^\circ). 
\]
This can be rephrased by saying that we have natural complexes $R\Hom^\bullet_\Ac(M,N)$, given for
each $M,N\in\Mod_\Ac$ and satisfying
\[
H^iR\Hom^\bullet_\Ac(M,N) \,\,\cong \,\, \Hom_{D(\Ac)} (M,\Sigma^i N). 
\]
We further recall that a dg-module $M\in\Mod_\Ac$ is called {\em perfect}, if $M$ is compact in
$D(\Ac)$ in the categorical sense, i.e., the functor $\Hom_{D(\Ac)}(M,-)$ commutes with infinite
direct sums. We denote by $\Perf_\Ac$ the dg-category whose objects are perfect dg-modules and
\[
\Hom^\bullet_{\Perf_\Ac}(M,N)\,\,=\,\, R\Hom^\bullet_\Ac(M,N).
\]
Any dg-functor $\Ac\to\Bc$ gives rise to a Quillen adjunction
\[
f_!: \Mod_\Ac \longleftrightarrow \Mod_\Bc: f^*
\]
where $f^*$ is obtained by composing dg-functors $\Bc^{\op} \to \Mod_\k$ with $f^{\op}$. This induces a
dg-functor
\[
f_!: \Perf_\Ac\to\Perf_\Bc
\]
and a triangulated functor
\[
f^*: D(\Bc)\to D(\Ac). 
\]
The dg-Yoneda embedding factors through a  faithful dg-embedding 
\[
\beta_\Ac: \Ac \lra \Perf_\Ac. 
\]

\begin{Defi}
A dg-category $\Ac$ is called {\em perfect}, if $\beta_\Ac$ is a quasi-equivalence. 
\end{Defi}
It is known that for a perfect $\Ac$, we have that $H^0(\Ac)$ is an idempotent-complete
triangulated category. 

We now define a second model structure on $\dgcat$ which we call the
{\em Morita model structure}. First, we recall that a dg-functor $f: \Ac\to\Bc$
with $\Ac\neq\emptyset$
is called a {\em Morita equivalence}, if the following equivalent conditions
are satisfied:
\begin{enumerate}
\item $f_!: \Perf_\Ac\to\Perf_\Bc$ is a dg-equivalence of dg-categories.
\item $f^*: D(\Bc)\to D(\Ac)$ is an equivalence of triangulated categories. 
\end{enumerate}
See \cite{keller-deriving} and \cite{tabuada}, \S 2.5,
for details, including the treatment of  the case when $\Ac$ is empty.
 The Morita model structure on $\dgcat$ is defined by
(\cite{tabuada}, Th. 0.7):
\begin{itemize}
\item[(MW)] Weak equivalences are Morita equivalences.

\item[(MC)] Cofibrations are the same as for the quasi-equivalence model structure.

\item[(MF)] Fibrations are determined by the right lifting property with respect to
trivial cofibrations. 
\end{itemize}
We denote by $\Mor$ the class of Morita equivalences and by $\Hmo=\dgcat[\Mor^{-1}]$ the homotopy
category of the Morita model structure. There is a Quillen adjunction
\[
\id: (\dgcat,\Qeq) \longleftrightarrow (\dgcat, \Mor) : \id
\]
which exhibits the Morita model structure on $\dgcat$ as a left Bousfield localization (see
\cite{hirschhorn}) of the quasi-equivalence model structure and hence induces an adjunction of
homotopy categories
\[
 F: \Heq \longleftrightarrow \Hmo : G
\]
where $G$ is fully faithful.

\begin{prop}
(a) A dg-category is fibrant for the Morita model structure, if and only if it is perfect.

(b) For any dg-category $\Ac$ the canonical dg-functor
\[
\beta_{\Perf_\Ac}: \Perf_\Ac \lra \Perf_{\Perf_{\Ac}}
\]
is a quasi-equivalence. In particular, $\Perf_\Ac$ is perfect, and $\beta_\Ac: \Ac\to\Perf_\Ac$
is a Morita equivalence.
\end{prop}
\begin{proof} (a) is Proposition 0.9 of \cite{tabuada}. Part (b) is Lemma 7.5 in \cite{toen-morita} 
\end{proof}

Note that, as a consequence, we obtain that the Yoneda embedding $\beta_{\Ac}: \Ac \to \Perf_{\Ac}$ exhibits
$\Perf_{\A}$ as a fibrant replacement of $\Ac$ in the Morita model structure.

\begin{ex}[(Morita duality)] 
Considering $\k$ as a 1-object dg-category, we see that
$\Perf_\k\subset \Mod_\k$ is the full dg-subcategory of complexes with total cohomology space
finite-dimensional. By the above, it is a Morita fibrant replacement of $\k$. 
The derived tensor product $\otimes^L$ makes $\Hmo$ into a symmetric monodical category with unit object
$\k$. This monoidal structure is closed, with internal $\Hom$ objects given by
\[
\IRHom_{\Hmo}(\Ac, \Bc) \,\,=\,\,\IRHom(\Ac, \Perf_\Bc)
\]
(fibrant replacement of the second argument), see \cite{tabuada}, Cor. 0.12.
Accordingly, for a dg-category $\Ac$, we will call 
\[
\Ac^\vee = \IRHom(\Ac, \Perf_\k)
\]
the {\em Morita dual} of $\Ac$. By the computation of $\IRHom$ in \cite{toen-morita}, the
dg-category $\Ac^\vee$ is identified with the full dg-subcategory in $\Mod_{\Ac^{\op}}$ formed by
dg-modules $M$, which are cofibrant and {\em pseudo-perfect}, i.e., such that each $M(x)$ is a perfect complex.
\end{ex}

Note that passing to the dual object is a contravariant functor
\begin{equation}
	(-)^\vee: \Hmo^\op \lra \Hmo. 
\end{equation}
As in any closed monoidal category, we say that a dg-category $\Ac$ is {\em dualizable} in $\Hmo$,
if the canonical dg-functor
\[
	\Ac^\vee \otimes \Bc \lra \IRHom_{\Hmo} (\Ac, \Bc)
\]
is a Morita equivalence for any $\Bc$. 

\begin{Defi} A dg-category $\Ac$ is called
\begin{itemize}
\item {\em proper}, if each complex $\Hom^\bullet_\Ac(x,y)$ belongs to $\Perf_\k$.

\item {\em smooth}, if the diagonal $\Ac^\op\otimes\Ac$-module
\[
\Ac: (x,y) \mapsto \Hom^\bullet_\Ac(x,y)
\]
belongs to $\Perf_{\Ac^\op\otimes\Ac}$. 
\end{itemize}
\end{Defi}

We recall the following result from \cite{toen-vaquie}.

\begin{prop}\label{prop:dg-dualizable}
A dg-category $\Ac$ is dualizable in $\Hmo$, if and only if it is smooth and proper. In this case a
dg-module over $\Ac$ is perfect if and only if it is pseudo-perfect, and so
\[
	\Ac^\vee \simeq \Perf_{\Ac^{\op}}
\]
is Morita equivalent to $\Ac^\op$. 
\end{prop}
\begin{proof}
Lemma 2.8 in \cite{toen-vaquie}.
\end{proof}

\subsection{The $2$-periodic case}

Let $\Vectt$, resp. $\Modt_\k$, be the category of $\Zt$-graded $\k$-vector spaces, resp. cochain
complexes, equipped with the $\Zt$-graded tensor product. The functor $\Sigma$ of shift of grading
on these categories satisfies $\Sigma^2=\Id$. 

We have an obvious $\Zt$-graded analogue of the concept of a dg-category: a small category enriched
over $\Modt_{\k}$. We refer to these structures as $2$-periodic, or $\Zt$-graded, dg-categories and
will leave out the extra adjective when it is obvious from the context. 
We will denote by $\dgcat^{(2)}$ the category of $\Zt$-graded dg-categories and their dg-functors. 

All the aspects of the homotopy theory of dg-categories and their dg-modules,
as recalled in \S \ref{subsec:morita}, can be extended to the $\Zt$-graded case
without any substantial changes. A convenient way to compare to the $\ZZ$-graded theory of 
\S \ref{subsec:morita} is as follows \cite[\S 5.1]{dyckerhoff-compact}. Note that
objects of $\Modt_\k$ can be seen as $2$-periodic $\ZZ$-graded cochain complexes over $\k$,
i.e., as dg-modules over the commutative dg-algebra
\[
	\k[u, u^{-1}], \quad \deg(u)=2, \, du=0.
\]
Under this identification, the $\Zt$-graded tensor product corresponds to $\otimes_{\k[u^{\pm 1}]}$.
We have an adjunction
\begin{equation}
	P: \Mod_\k \longleftrightarrow \Mod_{\k[u^{\pm 1}]}=\Modt_\k: F
\end{equation}
where $F$ is the forgetful functor, and $P$ is the functor of 2-periodization given by
\[
P(V^\bullet) = V^\bullet\otimes_\k \k[u^{\pm 1}], \quad P(V^\bullet)^{\bar i} = \bigoplus_{i\equiv \bar i \, \on{mod}\,  2} V^i,\,\,
\bar i\in \Zt.
\]
As explained in {\em loc. cit.}, this is a Quillen adjunction of model categories.
Applying this adjunction on the level of Hom-complexes, we get an adjunction
\begin{equation}\label{eq:adjunction:periodic}
P: \dgcat \longleftrightarrow \dgcatt: F
\end{equation}
The {\em quasi-equivalence model structure} on $\dgcatt$ is defined by:
\begin{itemize}
\item[(QW$^{(2)}$)] Weak equivalences are quasi-equivalences, i.e., morphisms taken by $F$
into quasi-equivalences in $\dgcat$.

\item[(QF$^{(2)}$)] Fibrations are defined by the right lifting property with respect to the set of
generating trivial cofibrations that is obtained by applying $P$ to the generating set in \cite{tabuada}.
This leads to the description of fibrations which is the $\Zt$-graded version of (QF).

\item[(QC$^{(2)}$)] Cofibrations are defined by the left lifting property with respect to trivial
fibrations. 
\end{itemize}

As observed in \cite[\S 5.1]{dyckerhoff-compact}, this indeed defines a model structure such that
\eqref{eq:adjunction:periodic} becomes a Quillen adjunction. We denote by $\Heq^{(2)}$ the homotopy
category of this model structure. 

All results and definitions recalled in \S \ref{subsec:morita} have obvious 2-periodic case
analogues. In particular, we will denote by $\Hmo^{(2)}$ the Morita homotopy category of $\dgcatt$
and will refer to $2$-periodic versions of other statements in \S \ref{subsec:morita} without
further explanation.

\section{Loop factorizations}
\label{section:loop}
\subsection{$\L$-categories and the cyclic category}
\label{subsec:lcat}

\begin{Defi}
 By a $\L$-{\em category} we mean a pair $(\C, w)$, where $\C$ is a category and $w: \Id_\C\Rightarrow \Id_\C$
is a natural transformation. Thus $w$ is a system of morphisms  $w_x: x\to x$ given for each $x\in\Ob(\C)$
and such that $fw_x=w_yf$ for each morphism $f: x\to y$. If $w$ is clear from the context, we will omit it from the notation. 
\end{Defi}

This definition is due to Drinfeld \cite{drinfeld}. Here are two reformulations. 
First, let us denote by $\L=\{0,1,2,...\}$ the additive monoid of non-negative integers, and let $B\L$ be the category
with one object corresponding to $\L$. Since $\L$ is commutative, $B\L$ is a symmetric monoidal category. 
 A  $\L$-category is the same as a category with action of $B\L$.

Second, let $\LSet$ be the category of $\L$-sets, i.e., of sets with a $\L$-action. 
Given $\L$-sets $A$ and $B$, we define the $\L$-set
\[
  A \otimes_{\L} B = A \times B / \{ (n+a, b) \sim (a, n+b), \,\, n\in\L \}.
\]
This operation makes   $\LSet$ into a   symmetric monoidal  category, with unit  object $\L$
(considered as an $\L$-set). Thus a morphism of $\L$-sets $A\otimes_\L B\to C$ is the same as an
$\L$-bilinear map $A\times B\to C$. 

\begin{prop}
Let $\C$ be a small category. The two following sets of data are in bijection:
\begin{enumerate}

\item  Structures of a $\L$-category on $\C$, i.e., natural transformations 
$w: \Id_\C\Rightarrow \Id_\C$.

\item Enrichments of $\C$ in  
$\LSet$, i.e.,  ways of defining $\L$-action on each $\Hom_\C(x,y)$ so that the composition is $\L$-bilinear. 

\end{enumerate}

\end{prop}

\begin{proof}
Given $w$, we define, for any $f: x\to y$, the morphism $n+f$ as $fw_x^n=w_y^nf$. Given an enrichment, i.e.,
system of actions of $\L$ on each $\Hom_\C(x,y)$, we define $w_x = 1+ \Id_x$. The details are left to the reader.
\end{proof}

\begin{exa} \label{exa:circular} Let $n \ge 0$ and consider the circular quiver $Q^n$ with set of vertices given 
	by $\ZZ/(n+1)$ and, for every $i \in \ZZ/(n+1)$, an arrow from $i$ to $i+1$. Let $\Q^n$ be
	the category freely generated by $Q^n$. Thus $\Q^0=B(\L)$. 
	The category $\Q^n$ is admits a natural
	$\L$-category structure where $w_i: i\to i$  is the cycle
	of degree $1$ at $i$.
	\end{exa}

 We  define an $\L$-{\em functor} between $\L$-categories  as  an enriched functor. Explicitly, if we write our $\L$-categories
 as $(\C, w)$ and $(\C',w')$, then a $\L$-functor between them is an ordinary functor $F: \C\to \C'$ such that
 $F(w_x)=w'_{F(x)}$ for each $x\in\Ob(\C)$. 
  The set of $\L$-functors between $\L$-categories $\C$ and $\C'$ is denoted by $\Fun_{\L}(\C,\C')$.
 We denote by $\LCat$ the
  category of $\L$-categories with morphisms given by $\L$-functors. 
  
  Further, $\Fun_{\L}(\C,\C')$ is the set of objects of a category $\IFun_{\L}(\C,\C')$
  whose morphisms are natural transformations $\eta: F \to G$ of $\L$-functors. Note that 
   $\IFun_{\L}(\C,\C')$ itself carries a structure of a $\L$-category: for $\eta$ as above we
   define $(n+\eta)_x = n+\eta_x: F(x)\to G(x)$ for $x\in\Ob(\C)$ and $n\in\ZZ_+$.

  \vskip .2cm
  
  Given $\L$-categories $\C$ and $\C'$, we define a $\L$-category $\C \otimes_{\L} \C'$ with set of
objects $\Ob (\C) \times \Ob (\C)$ and morphisms given by 
\[
  \Hom_{\C \otimes_{\L} \C'}( (x,x'), (y,y')) := \Hom_{\C}(x,y) \otimes_{\L} \Hom_{\C'}(x',y')
\]
where $x,y$ and $x',y'$ are objects of $\C$ and $\C'$, respectively. This operation provides a
monoidal structure on $\LCat$ with unit given by the $\L$-category $\Q^0$.   We have an adjunction
\[
  \Fun_{\L}(\C \otimes_{\L} \C', \D) \cong
  \Fun_{\L}(\C, \IFun_{\L}(\C', \D))
\]
which shows that the monoidal structure on $\LCat$ is closed.

\Par{The cyclic category, cyclic ordinals and cyclic objects.}
We recall Connes' definition of the cyclic category $\Lambda$, see \cite{connes}. The objects of $\Lambda$ are given by
nonnegative integers where we denote the object corresponding to $n \ge 0$ by $\cn$. We use the map 
\begin{equation}\label{cyc-ordinals-roots}
  \ZZ/(n+1) \to \CC, \quad  k \mapsto \exp\bigl (\frac{2 \pi i k}{n+1}\bigr)
\ee
to identify the elements of $\ZZ/(n+1)$ with the set of ($n+1$)st roots of unity contained in the unit circle
$S^1 \subset \CC$. A map $f: \cm \to \cn$ in $\Lambda$ is given by a homotopy class of continuous
maps $f: S^1 \to S^1$ of degree $1$ mapping $\ZZ/(m+1)$ into $\ZZ/(n+1)$.

Following Drinfeld \cite{drinfeld}, we provide an alternative description of   $\Lambda$.

\begin{prop}\label{prop:quiver} There is a fully faithful functor
	\[
		\FC: \Lambda \lra \LCat, \quad \cn\mapsto  \Q^n
	\]
	which embeds   $\Lambda$ into the category of $\L$-categories. \qed
\end{prop}

More generally, by a  {\em (total) cyclic order} on a finite set $I$, $|I|=n+1>0$, we mean a class of total orders up to
the action of the group $\ZZ/(n+1)$ of cyclic rotations. Alternatively, a cyclic order can be defined
as a   ternary relation  of a certain kind \cite{huntington, sibilla-treumann-zaslow}.
 A finite set with a cyclic order will be called a {\em finite cyclic ordinal}. Each finite cyclic ordinal $I$ is isomorphic to
 some $\ZZ/(n+1)$ and so gives
 rise to a $\ZZ_+$-category $\Q^I$ as in Example \ref{exa:circular}.  We will sometimes replace $\Lambda$ by an
 equivalent large category $\bf\Lambda$ whose objects are all finite cyclic ordinals and
 \be
 \Hom_{\bf\Lambda}(I,J) := \Hom_{\LCat}(\Q^I, \Q^J). 
 \ee
 By a {\em cyclic} (resp.{\em cocyclic}) {\em object} in a category $\Cb$, we mean a contravariant (reap. covariant)
 functor $X:\Lambda\to\Cb$. Note that such functor canonically extends to $\bf\Lambda$, so we can speak about
 objects $X(I)\in\Cb$ for any finite cyclic ordinal $I$.

\subsection{Loop  factorizations  in $\L$-categories}\label{subsec:loopfac}
Let $(\C,w)$ (or simply $\C$) be an $\L$-category. 
 We define a {\em loop factorization in $\C$} to be a functor of $\L$-categories
$F: \Q^1 \to \C$. Explicitly, a loop factorization can be viewed a datum
\[
F=\bigl\{	\xymatrix{
	x_1 \ar@/_1ex/[r]_\varphi & x_0 \ar@/_1ex/[l]_\psi\\
}\bigr\}
\]
of objects and morphisms in $\C$
such that $\varphi \psi = w_{x_0}$ and $\psi \varphi =  w_{x_1}$.

\begin{ex}\label{ex:MF}
 Let $R$ be an associative ring, and $w\in R$ be a central element.
Let $\Perf_R$ be the category of  finitely generated projective left $R$-modules. Multiplication by $w$
makes $\Perf_R$ into a $\L$-category. A loop factorization in $\Perf_R$ is the same as a matrix
factorization of $w$ in the standard sense \cite{eisenbud}. Further, the construction of dg-categories of matrix
factorization extends to our context as follows. 

\end{ex}

Let $\k$ be a field and $\C$ be a  $\L$-category. 
Denote by $\k[\C]$ the $\k$-linear envelope of $\C$, i.e., the category with
the same objects as $\C$ and $\Hom_{\k[\C]}(x,y)$ being the $\k$-vector space spanned by the set $\Hom_\C(x,y)$.
For any two 
  loop factorizations: $F$ as above and $F'=\bigl\{	\xymatrix{
	x'_1 \ar@/_1ex/[r]_{\varphi'} & x'_0 \ar@/_1ex/[l]_{\psi'}\\
}\bigr\}
$ in
$\C$, we define the $\Zt$-graded $\k$-module $\Hom^\bullet(F,F')$ given by
\begin{equation}\label{eq:pathen}
	\begin{split}
		\Hom^0( F, F') &  = \Hom_{\k[\C]}(x_0,x_0') \oplus \Hom_{\k[\C]}(x_1,x_1')\\
		\Hom^1( F, F') & = \Hom_{\k[\C]}(x_1,x_0') \oplus \Hom_{\k[\C]}(x_0,x_1')
	\end{split}
\end{equation}
  Any element of $\Hom^\bullet(F,F')$ can be represented by a matrix
\[
  \begin{pmatrix}\alpha & \gamma\\ \delta & \beta \end{pmatrix}=
  \begin{pmatrix}\alpha & 0\\ 0 & \beta \end{pmatrix}+
  \begin{pmatrix}0 & \gamma\\ \delta & 0 \end{pmatrix}
\]
with $(\alpha,\beta) \in \Hom^0(F, F')$ and $(\gamma,\delta) \in \Hom^1(F, F')$. We define the differential on
$\Hom^\bullet(F,F')$ by the formulas
\begin{equation}\label{eq:pathendg}
  \begin{split}
  \begin{pmatrix} \alpha & 0\\ 0 & \beta \end{pmatrix} & \mapsto 
  \begin{pmatrix} 0 & \varphi'\\ \psi' & 0 \end{pmatrix} 
  \begin{pmatrix} \alpha & 0\\ 0 & \beta \end{pmatrix} -
  \begin{pmatrix} \alpha & 0\\ 0 & \beta \end{pmatrix} 
  \begin{pmatrix} 0 & \varphi\\ \psi & 0 \end{pmatrix}\\
  \begin{pmatrix} 0 & \gamma\\ \delta & 0 \end{pmatrix} & \mapsto 
  \begin{pmatrix} 0 & \varphi'\\ \psi' & 0 \end{pmatrix} 
  \begin{pmatrix} 0 & \gamma\\ \delta & 0 \end{pmatrix} +
  \begin{pmatrix} 0 & \gamma\\ \delta & 0 \end{pmatrix} 
  \begin{pmatrix} 0 & \varphi\\ \psi & 0 \end{pmatrix},
  \end{split}
\end{equation}
In this way we get a $\Zt$-graded dg-category, denoted  $\LF(\C)$,  whose objects are loop factorizations in $\C$
and the $\Hom$-complexes are given by the $\Hom^\bullet(F,F')$.

\begin{thm}\label{prop:lf} Associating to a $\L$-category $\C$ the dg category $\LF(\C)$,
gives a functor
	\[
		\LF: \LCat \lra \dgcat^{(2)}.
	\]
	 \end{thm}
\begin{proof} This follows directly from the definitions,  since the
	formulas \eqref{eq:pathen} and \eqref{eq:pathendg} are intrinsically functorial.
\end{proof}

\subsection{The cyclic model for $D^{(2)}(A_n{\textrm{-}} \on{mod})$ via matrix factorizations}
\label{subsec:graded-factorizations}

 \Par{The root category and its cyclic symmetry.}
 Consider the Dynkin quiver of type $A_n$ with its standard orientation:
  \[
	A_n = \bigl\{ \xymatrix{
		\overset{1}{\bullet} \ar[r] & \overset{2}{\bullet} \ar[r] & \dots \ar[r] &
		\overset{n}{\bullet}
	}\bigr\}.
\]
 Denote $A_n\mod$ the category of finite-dimensional representations of $A_n$ over $\k$, i.e.,
of diagrams $V_1\to \cdots\to V_n$ of finite-dimensional $\k$-vector spaces,
and let 
  $D^{(2)}(A_n\mod)$  be the derived category of 2-periodic complexes over $A_n\mod$.
The latter category is known as the {\em root category} for $A_n$ because its indecomposable
objects are in bijection with  roots of the root system of type $A_n$, see \cite{happel}. More precisely, 
for each $1\leq i\leq j\leq n$ we  denote $\k_{[i,j]}\in A_n\mod$ 
the indecomposable object
having
$\k$ in positions from the interval $[i,j]$ and $0$ elsewhere. 
Then indecomposable objects in $D^{(2)}(A_n\mod)$ are
 \be
  {\bf e}_{ij} = \begin{cases}
  \k_{[n-j+1, n-i]}, \text{ if } i<j;\\
  \Sigma\,  \k_{[n-i+1, n-j]}, \text{ if } i>j,
  \end{cases}, \quad i,j\in \{0,1,\cdots, n\}, \, i\neq j,
  \ee
  so that $\Sigma {\bf e}_{ij}\simeq {\bf e}_{ji}$ in all cases. 
The Grothendieck group $K(D^{(2)}(A_n\mod))$ is identified with the root lattice for $A_n$, and the
class of ${\bf e}_{ij}$ is the standard root $e_{ij}$. 

Further, $D^{(2)}(A_n\mod)$ carries a self-equivalence known as the (derived) {\em Coxeter functor}
\[
C_n: D^{(2)}(A_n\mod) \lra D^{(2)}(A_n\mod), \quad C_n^{n+1} \simeq \Id. 
\]
It can be defined either as the composition of derived reflection functors \cite{gelfand-manin}, or
characterized intrinsically by the condition that $\Sigma\circ C_n^{-1}$ is the {\em Serre functor}
of $D^{(2)}(A_n\mod)$, i.e., we have natural isomorphisms
\[
\Hom_{D^{(2)}(A_n\mod)}(V^\bullet, W^\bullet)^* \simeq  \Hom_{D^{(2)}(A_n\mod)}
(W^\bullet, \Sigma C_n^{-1}(V^\bullet)).
\]
The   automorphism of $K(D^{(2)}(A_n\mod))$ induced by $C_n$, is the Coxeter transformation $c_n$
from the Weyl group $W_{A_n}=S_{n+1}$. This transformation is the  $(n+1)$-cycle: $c_n=(012\cdots n)$.  
Being an equivalence, $C_n$ preserves indecomposable objects and the action on such objects corresponds
to the action of $c_n$ on the roots. In particular, the action on simple modules is
\begin{equation}\label{eq:coxeter-irreducible}
\k_{[n,n]} \longmapsto \k_{[n-1,n-1]} \longmapsto \cdots \longmapsto 
\k_{[1,1]} \longmapsto \Sigma \, \k_{[1,n]}. 
\end{equation}

 \begin{ex}\label{ex:triangles-in-polygon}
  Let $P_{n+1}$ be the convex plane $(n+1)$-gon with vertices labelled by
 $0,1,\cdots, n$ in the counterclockwise order. 
 We can represent ${\bf e}_{ij}$ as an oriented arc (side or diagonal) in $P_n$
 going from the vertex $i$ to the vertex $j$, so applying $\Sigma$ corresponds to change of
 orientation. 
  We say that a triple of distinct elements $(i,j,k)\in \cn \simeq \ZZ/(n+1)$ is 
 {\em in the counterclockwise cyclic order},
 if it can be brought by a cyclic rotation to a triple $(i',j',k')$ with $0\leq i'<j'<k'\leq n$. For any such triple we have a distinguished triangle
 in $D^{(2)}(A_n\mod)$:
 \[
 {\bf e}_{ij} \lra {\bf e}_{ik}\lra {\bf e}_{jk} \lra \Sigma {\bf e}_{ij} 
 \]
 which can be depicted as a triangle inscribed into $P_{n+1}$, similarly to
 \eqref{eq:geom-triangle}: 
 
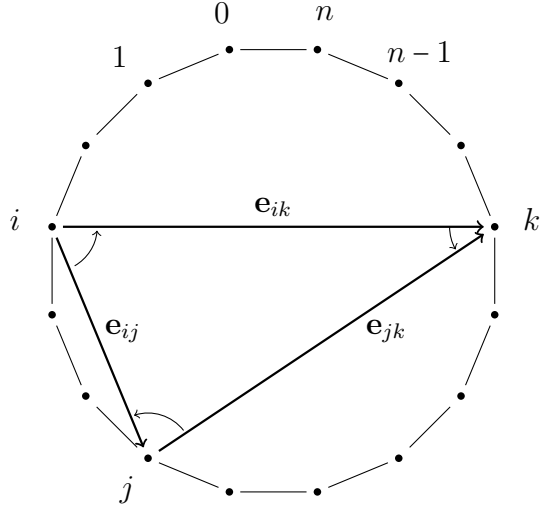
 \begin{center}
 \begin{tikzpicture}
 \node (A) at (11.25:3){}; 
 \node (B) at (33.75:3){}; 
 \node (C) at (56.25:3){}; 
 \node (D) at (78.75:3){}; 
 \node (E) at (101.25:3){}; 
 \node (F)  at (123.75:3){}; 
 \node (G) at (146.25:3){}; 
 \node (H) at (168.75:3){}; 
 \node (I) at (191.25:3){};
 \node (J) at (213.75:3){};
 \node (K) at (236.25:3){}; 
 \node (L) at (258.75:3){}; 
 \node (M) at (281.25:3){}; 
 \node (N) at (303.75:3){}; 
 \node (O) at (326.25:3){}; 
 \node (P) at (348.75:3){};

  \fill (A) circle (0.05);
  \fill (B) circle (0.05);
  \fill (C) circle (0.05);
   \fill (D) circle (0.05);
\fill (E) circle (0.05);
 \fill (F) circle (0.05);
 \fill (G) circle (0.05);
 \fill (H) circle (0.05);
\fill (I) circle (0.05);
 \fill (J) circle (0.05);
\fill (K) circle (0.05);
\fill (L) circle (0.05);
  \fill (M) circle (0.05);
 \fill (N) circle (0.05);
 \fill (O) circle (0.05);
 \fill (P) circle (0.05);
 
 \draw (A) -- (B) --(C) --(D)  -- (E) -- (F) -- (G) -- (H) -- (I) -- (J) -- (K) -- (L) -- (M) -- (N) --
 (O) -- (P) -- (A);  
 
 \node at (101.25:3.5){$0$}; 
 \node at (125.75:3.5){$1$}; 
 \node at (78.75:3.5){$n$}; 
 \node at (56.25:3.5){$n-1$}; 
 \node at (168.75:3.5){$i$}; 
 \node at (236.25:3.5){$j$}; 
 \node at (11.25:3.5){$k$}; 

\draw[->, line width = .3mm] (H) -- (K); 
\draw[->, line width = .3mm] (K) -- (A); 
\draw[->, line width = .3mm] (H) -- (A); 

 \centerarc[->](H)(-60:-5:.6)
 \centerarc[->](A)(180:210:.6)
 \centerarc[->](K)(38:108:.6)
 
 \node at (0,.85){${\bf e}_{ik}$};
\node at (-2, -.8){${\bf e}_{ij}$}; 
\node at (1.5, -.8){${\bf e}_{jk}$};

 \end{tikzpicture}
 \end{center}
 \end{ex}
  
\Par {Graded matrix factorizations.} 
 Let $L$ be an abelian group, and $R=\bigoplus_{a\in L} R_a$ be an $L$-graded associative $\k$-algebra (with $\k\subset R_0$). 
Let $w\in R_0$ be a central element, and $\Perf^L_R$ be the category of finitely generated projective $L$-graded left
$R$-modules (and their morphisms of degree 0).
As in Example \ref{ex:MF}, $\Perf^L_R$ is then a $\k$-linear $\L$-category. We define the 
  dg-category
$\MF^L(R,w)$
to have, as objects, loop factorizations in $\Perf^L_R$ and $\Hom$-complexes 
 defined
analogously to \eqref{eq:pathen} but with $\Hom_{\k[\C]}$, replaced by $\Hom_{\Perf^L_R}$.
We call $\MF^L(R,w)$
the dg-category of  {\em  $L$-graded matrix factorizations} of $w$. 
As with any category of matrix factorizations, the dg-category $\MF^L(R,w)$ is  perfect. In particular, the category  
 $H^0\MF^L(R,w)$ is triangulated.  
 For $i \in L$,  and $M\in\Perf^L_R$, we denote by $M(i)$ the graded $R$-module with $M(i)_n = M_{i+n}$.

Let $L = \ZZ/(n+1)$ where $n \ge 0$. We consider the polynomial ring $R = \k[z]$ as an
	$L$-graded $\k$-algebra with $\deg(z) = 1$,  and take  $w = z^{n+1}$. We 
	introduce the notation
	\[
	\T^n = \MF^{\ZZ/(n+1)}(\k[z], z^{n+1}).
	\]
	 The {\em rank} of an object 
	$F=\bigl\{	\xymatrix{
	M_1 \ar@/_1ex/[r]_\varphi & M_0 \ar@/_1ex/[l]_\psi\\
}\bigr\}\in\T^n$  is, by the definition, the rank of $M_0$ and $M_1$ as free
  $\k[z]$-modules  (these ranks are equal). 
	The shift of grading gives an equivalence of dg-categories
	\[
	\Pi_n: \T^n \lra \T^n, \quad\Pi_n \bigl\{	\xymatrix{
	M_1 \ar@/_1ex/[r]_\varphi & M_0 \ar@/_1ex/[l]_\psi\\
}\bigr\} = 
\bigl\{	\xymatrix{
	M_1(1) \ar@/_1ex/[r]_{\varphi(1)} & M_0(1) \ar@/_1ex/[l]_{\psi(1)}\\
}\bigr\}, \quad \Pi_n^{n+1}=\Id. 
	\]

	  The following is an  adaptation of the main result of \cite{takahashi}.  	We omit the proof as well as the proofs of
	  the next few followup statements, as they are similar to {\em loc. cit.}
	  
	\begin{thm}\label{thm:A-n-factorizations}
	 
	 (a) The triangulated category $H^0\T^n$ is equivalent to $D^{(2)}(A_n\mod)$. 
	 
	 (b) Under this equivalence, the functor induced by $\Pi_n$, corresponds to the derived
	Coxeter functor $C_n$. \qed
	 	 \end{thm} 
		 
		 The rank one objects of $\T^n$ have the form 
		  \begin{equation}\label{eq:rank-1-fact}
	\begin{gathered}
	 \, [i,j] =   \bigl\{\xymatrix{
		    R(i) \ar@<.5ex>[r]^{z^{j-i}} & \ar@<.5ex>[l]^{z^{i-j}} R(j)
	    }\bigr\}, \quad i\neq j, \\
[i,i] = \bigl\{\xymatrix{
		    R(i) \ar@<.5ex>[r]^{\Id} & \ar@<.5ex>[l]^{z^{n+1}} R(i)	}\bigr\},
\quad
  [i,i]'= \bigl\{ \xymatrix{
		    R(i) \ar@<.5ex>[r]^{z^{n+1}} & \ar@<.5ex>[l]^{\Id} R(i)
	    }\bigr\},  
	     \quad i\in \ZZ/(n+1),    
	     \end{gathered} 
	\ee
	where the exponents are to be interpreted via the identification $\ZZ/(n+1) \cong
	\{0,1,\dots,n\}$. It is clear that
	\[
	\Sigma[i,j] = [j,i], \quad i\neq j, \quad \Sigma[i,i]=[i,i]'. 
	\]
	
	One verifies by computing Hom-complexes that
	\[
	[i,i] \simeq [i,i'] \simeq 0 \in H^0\T^n
	\]
	are zero objects in $H^0\T^n$. 
	
	\begin{prop}\label{prop:rank-1-factorizations}
	The  $[i,j]$ exhaust all isomorphism classes of 
	indecomposable objects in $H^0\T^n$
	\qed
	\end{prop}
	
	\begin{prop} \label{prop:cones} 
	Let $i,j,k\in\ZZ/(n+1)$ be distinct elements, in counterclockwise cyclic order. Then the degree 0 morphisms
	\[
	\begin{gathered}
	\alpha_{ijk} = {\begin{pmatrix} 1 & 0\\0&z^{k-j}\end{pmatrix}}\in\Hom^0([i,j], [i,k]), \quad
	\beta_{ijk} = {\begin{pmatrix} z^{j-i} & 0\\0&1\end{pmatrix}}\in \Hom^0([i,k], [j,k]), \\
	\gamma_{ijk} = {\begin{pmatrix} 0 & z^{i-k}\\1&0\end{pmatrix}}\in\Hom^0([j,k], [k,i]),
	\end{gathered}
	\]	 
	are  closed and induce a distinguished triangle in $H^0\T^n$:
	\[
	[i,j] \overset{\alpha_{ijk}}{\lra} [i,k]\overset{\beta_{ijk}}{\lra} [j,k]
	\overset{\gamma_{ijk}}{\lra} [k,i] =  \Sigma [i,k]. \qed
	\]
	\end{prop}	 
	The equivalence of Theorem \ref{thm:A-n-factorizations} can be chosen
	so that 
	\begin{equation}\label{eq:[i,j]-e-ij}
	{\bf e}_{ij} \mapsto [i,j], \quad i\neq j, 
	\ee	 
	and the triangles of Example \ref{ex:triangles-in-polygon}
	correspond to those of Proposition \ref{prop:cones}.
	For this, note that each of the following two diagrams
	\begin{equation}\label{eq:R-n-triangles} 
	\begin{gathered}
	\R_n^\triangleleft \,=\,\bigl\{ [0,1] \buildrel \alpha_{012}\over\lra [0,2]
	\buildrel \alpha_{023}\over\lra \cdots \buildrel\alpha_{0, n-1, n}\over\lra [0,n]\bigr\},\\
	\R_n^\triangleright \,=\,\bigl\{ [0,n] \buildrel \beta_{01n}\over\lra [1,n]
	\buildrel \beta_{12n}\over\lra \cdots \buildrel\beta_{n-2, n-1, n}\over\lra [n-1,n]\bigr\},
\end{gathered}
	\ee
	can be considered as a representation of $A_n$ in $\T^n$. The dg-functor
	\[
	\Phi_n^\triangleleft: C^{(2)}(A_n\mod) \lra \T^n, \quad V^\bullet\mapsto \on{RHom}_{A_n} (\R_n^\triangleleft, V^\bullet)
	\]
	establishes an equivalence satisfying \eqref{eq:[i,j]-e-ij}. A different equivalence $\Phi_n^\triangleright$
	 can be constructed using $\R_n^\triangleright$.

	\subsection{The dg-categories $\E^n$ and their cocyclic structure}
	
	We denote by $\E^n=\LF(\Q^n)$ the dg-category of loop factorizations of the $\L$-category
	$\Q^n$ from Example \ref{exa:circular}. 
		 Because the construction $\C\mapsto\LF(\C)$ is covariantly functorial in the $\L$-category $\C$,  we obtain
		 immediately: 

	  \begin{prop}\label{prop:E-n-cocyclic}
	 The collection $\E^\bullet = (\E^n)_{n\geq 0}$, forms a cocyclic object in $\dgcat^{(2)}$. 
	  \qed
	 \end{prop}

	 For $i,j\in \ZZ/(n+1)$, $i\neq j$, let
	 $\varphi_{ij}$ be the shortest oriented path from $i$ to $j$ in the circular quiver $Q^n$.
	 Recall that $w_i$ denotes the full circle path beginning and ending at $i$.
	The objects of $\E^n$ are exhausted by the following:
	 \begin{equation}
		 \begin{gathered}
			 E_{ij} = \bigl\{	\xymatrix{
				 i \ar@/_1ex/[r]_{\varphi_{ij}} & j\ar@/_1ex/[l]_{\varphi_{ji}}\\
			 }\bigr\}, \quad i\neq j, \\
			 E_{ii} = \bigl\{	\xymatrix{
				 i \ar@/_1ex/[r]_{w_i} & i\ar@/_1ex/[l]_{\Id}\\
			 }\bigr\},
			 \quad 
			 E'_{ii} = \bigl\{	\xymatrix{
				 i \ar@/_1ex/[r]_{\Id} & i\ar@/_1ex/[l]_{w_i}\\
			 }\bigr\},  \quad i\in \ZZ/(n+1). 
		 \end{gathered} 
	 \end{equation}
	  	 
	 \begin{prop}\label{prop:E-T-fully}
	 The correspondence
	 \[
	 E_{ij}\mapsto [i,j], \quad E'_{ii}\mapsto [i,i]'
	 \]
	 extends to a fully faithful dg-functor $\epsilon_n: \E^n\to\T^n$. 
	 \end{prop}
	 
	 \begin{proof} The definitions imply at once the identifications
	 of the $\Hom$-complexes. 
	 \end{proof}

	 Let $D^{(2)}_{\on{ind}}(A_n\mod)$ be the full subcategory in $D^{(2)}(A_n\mod)$ formed by all indecomposable
	 objects, including the zero object. The above proposition, together with Theorem \ref{thm:A-n-factorizations}(a),
	 implies an equivalence of categories
	 \begin{equation}
		 H^0(\E^n) \simeq D^{(2)}_{\on{ind}}(A_n\mod). 
	 \end{equation}
	 	 
	Let $\A^n$ be the $\k$-linear envelope of the quiver $A_n$, considered as a differential
 $\ZZ/2$-graded category concentrated in even degree with zero differential. 
  The diagrams \eqref{eq:R-n-triangles} can be considered as dg-functors
\begin{equation}
	r_n^\triangleleft, r_n^\triangleright: \A^n\lra\E^n. 
\end{equation}

 \begin{prop} \label{prop:r-n-morita}
 The dg functors $r_n^\triangleleft, r_n^\triangleright, \epsilon_n$ are Morita equivalences.
\end{prop}

\begin{proof} For $\epsilon_n$, which is an embedding of a full dg-category, the statement follows from
Proposition \ref{prop:rank-1-factorizations}: each object of $H^0 \T^n$ is isomorphic to a direct sum of
objects from $\on{Im}(\epsilon_n)$. To prove the statement for $r_n^\triangleleft$,  we note that is
is quasi-fully faithful: induces quasi-isomorphisms on Hom-complexes. By the above, it is enough to
prove that $\epsilon_n\circ r_n^{\triangleleft}$ is a Morita equivalence. This follows because, by
  Proposition \ref{prop:cones}, each object of $H^0\epsilon_n(\E^n)$ and therefore, by the above,
  each object of $H^0\T^n$,
  is obtained from objects in $\on{Im}(\epsilon_n\circ r_n^\triangleleft)$ by taking iterated cones of
  morphisms.
  The case of $r_n^\triangleright$ is similar. 
\end{proof}

\subsubsection{$\E^2$ and distinguished triangles} 

In Section \ref{sec:sconstruction} below, we provide a thorough analysis of the universal property of
the cyclic object $\E^{\bullet}$ in terms of Waldhausen's S-construction. Here, we use a more
explicit approach identifying the dg category $\E^2$ as a universal distinguished triangle category.

Let $\T$ be a triangulated category equipped with a $2$-periodic dg enhancement. Due to the
$2$-periodicity, a triangle
\[
	\xymatrix{
		A \ar[rr] & & B \ar[dl] \\
		& C \ar[ul]^{+1}&
	}
\]
in $\T$ can be depicted more symmetrically via the diagram
\[
	\xymatrix{
		A \ar[rr]^{+1} & & \Sigma B \ar[dl]^{+1} \\
		& C \ar[ul]^{+1}.&
	}
\]
We will refer to the latter diagram as a {\em symmetric triangle in $\T$}. A symmetric triangle in
$\T$ is called {\em distinguished} if its corresponding asymmetric triangle is distinguished.
We denote by $\F$ the $2$-periodic dg category with three objects $0$,$1$, and $2$, freely generated
by closed morphisms $p_{10}: 0 \to 1$, $p_{21}: 1 \to 2$, and $p_{02}: 2 \to 0$ of degree $1$. Given
$2$-periodic dg categories $\A$, $\B$, we denote by $[\A,\B]$ the set of morphisms from $\A$ to $\B$
in the homotopy category $\Ho(\dgcatt)$. 

\begin{prop}\label{prop:universal1}
	Let $\T$ be a triangulated category with $2$-periodic dg enhancement $\A$. 
	Then the set $[\F, \A]$ of morphisms in $\Ho(\dgcatt)$ is in canonical bijection with the
	set of isomorphism classes of triangles in $\T$.
\end{prop}
\begin{proof}
	The dg category $\F$ is obtained by iterated pushouts along generating cofibrations of
	$\dgcatt$ and hence cofibrant. By the usual model category formalism, we can therefore 
	compute the set $\Hom_{\Ho(\dgcatt)}(\F, \A)$ as homotopy classes of maps from $\F \to \A$.
	Here, two dg functors $f: \F \to \A$ and $g: \F \to \A$ are homotopic, if there exists a
	commutative diagram 
	\[
		\xymatrix{
			& \A\\
			\F \ar[r] \ar[ur]^{f} \ar[dr]_{g} & P(\A) \ar[u]_{p_1}\ar[d]^{p_2}\\
			& \A
		}
	\]
	in $\dgcatt$.
	Here, $P(\A)$ denotes a path object for $\dgcatt$ which can be explicitly constructed as
	follows (cf. \cite{tabuada}). The objects of $P(\A)$ are pairs $(x,y)$ of objects of $\A$ equipped with a closed
	morphism $f: x \to y$ of degree $0$ which becomes an isomorphism in $H^0(\A)$. The mapping
	complex between objects $(x,y,f)$ and $(x',y',f')$ is defined as the suspension of the cone
	of the map
	\[
		\Hom_{\A}(x,x') \oplus \Hom_{\A}(y,y') \overset{(f'_*, -f^*)}{\lra} \Hom_{\A}(x,y').
	\]
	Now consider the obvious map of sets
	\[
		\Hom_{\dgcatt}(\F, \A) \lra \left\{ \text{symmetric triangles in $\T = H^0(\A)$} \right\}
	\]
	which is easily seen to be surjective. Unravelling the definition of the path dg category, it is straightforward 
	to verify that two dg functors from $\F$ to $\A$ are homotopic if and only if the
	corresponding symmetric triangles are isomorphic as diagrams in $\T$.
\end{proof}

\begin{rem}\label{rem:classifying1} Given $2$-periodic dg categories $\A$, $\B$, the set of maps $[\A,\B]$ can be
	canonically identified with the set of connected components of the mapping space
	$\Map(\A,\B)$ of the simplicial localization of $\dgcatt$ along quasi-equivalences.
	Therefore, Proposition \ref{prop:universal1} provides an interpretation of 
	$\Map(\F,\A)$ as a classifying space for triangles in $\T$.
\end{rem}

By Proposition \ref{prop:cones}, the diagram in $H^0(\E^2)$ 
\begin{equation}\label{eq:trianglediag}
	\xymatrix@C=.5ex{
		E_{01} \ar[rr]^-{[f_1]} & & E_{20}
		\ar[dl]^-{[f_2]} \\
		& \ar[ul]^-{[f_3]} E_{12}
	}
\end{equation}
with 
\[
	f_1 = \begin{pmatrix} 0 & z\\1&0\end{pmatrix}\quad\quad 
	f_2 = \begin{pmatrix} 0 & z^2\\1&0\end{pmatrix}\quad\quad
	f_3 = \begin{pmatrix} 0 & z\\1&0\end{pmatrix}
\]
is a symmetric distinguished triangle. Analyzing the null homotopies in $\E^2$ of the pairwise
composites of the morphisms $f_i$, we can lift \eqref{eq:trianglediag} to a diagram in the dg category $\E^2$
\[
	\xymatrix@C=.5ex{
		E_{01} \ar[rr]^-{f_1} \ar@/_4ex/[dr]_{h_{21}}& & E_{20}\ar@/_4ex/[ll]_{h_{32}}
		\ar[dl]^-{f_2} \\
		& \ar[ul]^-{f_3} \ar@/_4ex/[ur]_{h_{13}}E_{12}
	}
\]
with 
\[
	h_{21} = \begin{pmatrix} 0 & 0\\1&0\end{pmatrix}\quad\quad
	h_{32} = \begin{pmatrix} 0 & 0\\1&0\end{pmatrix}\quad\quad
	h_{13} = \begin{pmatrix} 0 & 0\\1&0\end{pmatrix}
\]
satisfying
\begin{equation}\label{eq:rel1}
	d(h_{ij}) = f_i f_j.
\end{equation}
Further, we observe that there are relations in $\E^2$
\begin{equation}\label{eq:rel2}
	h_{32} f_1 + f_3 h_{21} = \id_{E_{01}}\quad\quad
	h_{13} f_2 + f_1 h_{32} = \id_{E_{20}}\quad\quad
	h_{21} f_3 + f_2 h_{13} = \id_{E_{12}}.
\end{equation}
Equivalently, we can reformulate our observation by saying that there exists a dg functor $f: \D \to \E^2$ 
where $\D$ is the $2$-periodic dg category with objects $0$, $1$, $2$, generated
by closed morphisms $f_1: 0 \to 1$, $f_2: 1 \to 2$, $f_3: 2 \to 0$ of degree $1$, and morphisms
$h_{21}:0 \to 2$, $h_{32}:1 \to 0$, $h_{13}:2 \to 1$ of degree $1$ satisfying \eqref{eq:rel1} and
\eqref{eq:rel2}.

\begin{prop}\label{prop:moritad}
	The functor $f:\D \to \E^2$ is a Morita equivalence.
\end{prop}
\begin{proof} Left to the reader. 
\end{proof}
 
We will now give an interpretation of the set of morphisms $[\E^2, \A]$ in $\Ho(\dgcatt)$ in terms of
distinguished triangles in $H^0(\A)$. By Proposition \ref{prop:moritad}, this set is in natural
bijection with the set of morphisms $[\D, \A]$. To get a more explicit hold on this morphism set, we
may utilize the model structure on $\dgcatt$ and pass to a cofibrant replacement $p: \wD \to \D$.

The dg category $\wD$ can be explicitly described as follows. Denote by $\Q^2$ the cyclic quiver
with vertices $0$,$1$, and $2$, and arrows $0 \to 1$, $1 \to 2$, and $2 \to 0$. Let $\wD$ denote the $2$-periodic dg
category with objects $0$, $1$, and $2$, obtained by adjoining, for every path $\gamma$ in $\Q^2$ a
(not necessarily closed) morphism $p_{\gamma}: s(\gamma) \to t(\gamma)$ of degree $1$, with differential given by 
\begin{align*}
	d(p_{\gamma}) = 
	\begin{cases}
		\sum_{\beta \circ \alpha = \gamma} p_{\beta} p_{\alpha} - \id_i & \text{if $\gamma$
				is a degree $1$ cycle centered at $i$},\\
		0 & \text{if $\gamma$ has length $1$},\\
		\sum_{\beta \circ \alpha = \gamma} p_{\beta} p_{\alpha} & \text{for all other cases.}
	\end{cases}
\end{align*}
Since the category $\F$ is freely generated by paths in $\Q^2$ of length $1$, we have a natural embedding 
$i: \F \to \wD$. The functor $i$ is a cofibration, since it can be obtained by iterated pushouts 
along generating cofibrations. In particular, the dg category $\wD$ is cofibrant. The functor $p:
\wD \to \D$ is simply obtained by sending all morphisms corresponding to paths of length $\ge 3$ to $0$.

\begin{prop}\label{prop:cofibrantrep}
	The functor $p:\wD \to \D$ is a quasi-equivalence.
\end{prop}
\begin{proof} Left to the reader.
\end{proof}

In summary, we have a dg functors
\[
	\F \overset{i}{\lra} \wD \overset{p}{\lra} \D \overset{f}{\lra} \E^2
\]
with $\F$ and $\wD$ cofibrant, $p$ a quasi-equivalence, and $f$ a Morita equivalence.

\begin{prop}\label{prop:universal2}
	Let $\T$ be a triangulated category with $2$-periodic dg enhancement $\A$. 
	\begin{enumerate}
		\item A dg functor $\F \to \A$ representing a triangle in $\T$ lifts to a dg functor $\wD
			\to \A$ if and only if the triangle is distinguished.
		\item The pullback map 
			\[
				[\wD, \A] \lra [\F, \A] 
			\]
			is an inclusion with image given by the subset of isomorphism classes of distinguished symmetric
			triangles in $\T$. 
	\end{enumerate}
\end{prop}
\begin{proof}
	By definition, a distinguished symmetric triangle in $\T$ is a symmetric triangle in $\T$ that is
	isomorphic as a diagram in $\T$ to a {\em cone triangle} of the form
	\[
		\xymatrix@C=.5ex{
			A \ar[rr]^{+1} & & \Sigma B \ar[dl]^{+1} \\
			& \cone(f) \ar[ul]^{+1}&
		}
	\]
	where $f: A \to B$ is a closed degree $0$ morphism in $\A$ and $\cone(f)$ denotes the mapping cone
	construction in $\A$ (which is pre-triangulated in the sense of \cite{bondal-kapranov}).
	Therefore, by the argument of Proposition \ref{prop:universal1}, a symmetric triangle,
	represented by a dg functor $F: \F \to \A$, is distinguished, if and only if there exists
	a commutative diagram 
	\[
		\xymatrix{
			& \A\\
			\F \ar[r]^H \ar[ur]^{F} \ar[dr]_{G} & P(\A) \ar[u]_{p_1}\ar[d]^{p_2}\\
			& \A
		}
	\]
	in $\dgcatt$ where the symmetric triangle represented by $G$ is a cone triangle. We
	explicitly verify that any cone triangle in $\T$ can be lifted to a dg functor
	$\wD \to \A$ as follows. By construction of the object $\cone(f)$ as a twisted
	complex, it comes equipped with morphisms in $\A$ as depicted in the diagram
	\[
		\xymatrix@C=.5ex{
			A \ar@/_4ex/[dr]_-{r} \ar[rr]^{f_1} & & \Sigma B
			\ar[dl]^-{f_2} \\
			& \cone(f) \ar[ul]^{f_3} \ar@/_4ex/[ur]_-{s}  &
		}
	\]
	with $f_1,f_2,f_3$ closed morphisms of degree $1$, satisfying $f_2 f_1 = dr$, $f_1 f_3 =
	ds$, $f_3 f_2 = 0$, $r f_3 + f_2 s = \id_{\cone(f)}$, $f_3 r = \id_A$, $s f_2 = \id_{\Sigma
	B}$. Comparing these formulas to the defining formulas \eqref{eq:rel1} and \eqref{eq:rel2}
	of $\D$, there is an apparent dg functor $\D \to \A$ representing this triangle which can be
	precomposed with $p: \wD \to \D$ to obtain a dg functor $\wD \to \A$.
	
	Hence, we obtain a commutative diagram
	\[
		\xymatrix{
			\F \ar[r]^H \ar[d]_i & P(\A) \ar[d]^{p_2} \\
			\wD \ar@{-->}[ur] \ar[r]_{\widetilde{G}} & \A
		}
	\]
	in $\dgcatt$.
	Since $i$ is a cofibration and $p_2$ a trivial fibration, we can fill in the indicated
	functor $\wD \to P(\A)$. Composing this functor with $p_1$, we obtain a functor
	$\widetilde{F}: \wD \to \A$ lifting $F$. On the other hand, an easy calculation shows that
	any functor $\F \to \A$ which is a restriction of a functor $\wD \to \A$ represents a
	distinguished symmetric triangle. We have shown that the map $[\wD,\A] \to [\F,\A]$ surjects
	onto those functors which represent distinguished triangles in $\T$. It remains to show that
	the map is injective. We have a commutative diagram in $\dgcatt$
	\[
		\xymatrix{ & [\F, \A] \ar[dr]^{j^*} & \\
		[\E^2, \A] \cong [\wD, \A] \ar[ru]^{i^*} \ar[rr]_{\cong} &  & [\A^2, \A]
		}
	\]
	where the horizontal map is a bijection by Proposition \ref{prop:r-n-morita}. In particular, $i^*$
	must be injective which concludes the argument.
\end{proof}

\begin{rem}\label{rem:classifying2} In analogy with Remark \ref{rem:classifying1}, 
	Proposition \ref{prop:universal2} implies that the connected components 
	of the mapping space $\Map(\E^2, \A) \simeq \Map(\D,\A) \simeq \Map(\wD,\A)$ 
	are in bijective correspondence with isomorphism classes of distinguished triangles in $\T$. 
\end{rem}
	
\begin{rem} The analogue of statement (1) in Proposition \ref{prop:universal2} for
	$A_\infty$-categories is a result due to M. Kontsevich (cf. Proposition 3.8 in
	\cite{seidel-book}). While in the context of dg-categories we are forced to work with the
	(rather large) cofibrant replacement $\wD$ of $\D$, using $A_\infty$-categories and
	$A_\infty$-functors, one does not have to replace $\D$. On the contrary, one can pass to an
	even more economic {\em minimal model} of $\D$: a simple homological perturbation
	calculation shows that a minimal model $\overline{\D}$ is given by the $\A_{\infty}$-category with objects
	$0$, $1$, $2$, obtained by adjoining closed morphisms 
	$f_1: 0 \to 1$, $f_2: 1 \to 2$, $f_3: 2 \to 1$ of degree $1$ whose pairwise composition
	equals $0$, equipped with triple operations
	\[
		m_3(f_3,f_2,f_1) = \id_0 \quad\quad 
		m_3(f_1,f_3,f_2) = \id_1 \quad\quad 
		m_3(f_2,f_1,f_3) = \id_2 .
	\]
	The natural morphism  $\F \to \overline{\D}$ has the property that, given a
	$2$-periodic $A_\infty$-category $\A$, a (symmetric) triangle in $H^0(\A)$ represented by an $A_\infty$-functor $\F \to
	\A$, is exact if and only if this functor admits a lift to an $A_\infty$-functor
	$\overline{\D} \to \A$.
	Of course, the simplicity of the category $\overline{\D}$, when compared to $\wD$, comes at a price: the
	complexity is now hidden in the amount of data needed to specify an $A_\infty$-functor.
\end{rem}

\subsection{Cyclic duality and Morita duality}
\label{subsec:cyc-morita}

Given a $\L$-category $\C$,  the {\em dual} $\L$-category is defined as  $\C^{\vee} := \IFun_{\L}(\C, \Q^0)$, see \cite{drinfeld}.
The duality operation provides a functor
 \[
(-)^\vee:  \LCat^\op \lra \LCat.
\]
It is not a perfect duality, i.e., the canonical functor $\C\to \C^{\vee\vee}$ need not be an
isomorphism, nor an equivalence of $\L$-categories. However, it 
induces a perfect duality functor on cyclic ordinals
\[
 (-)^*:  {\bf\Lambda}^\op\lra{\bf\Lambda}, \quad I\mapsto I^*, \quad I^{**}=I, \quad \Q^{I^*}\simeq (\Q^I)^\vee. 
\]
Explicitly,  for a cyclic ordinal $I$, the dual cyclic ordinal $I^*$ is the set of {\em interstices}, i.e.,  of minimal cyclic intervals in $I$:

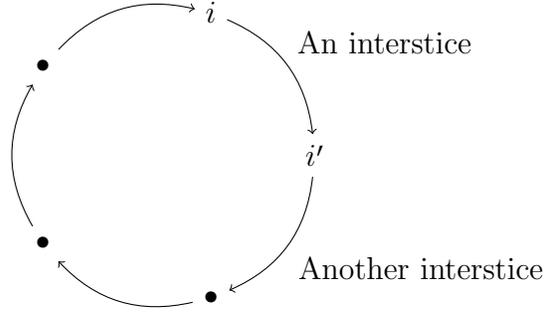
\begin{center}
\begin{tikzpicture}
   
 \node (origin) at (0,0) (origin) {}; 
  \node (origin') at (5,0) (origin') {}; 

\node[xshift=5cm]  (P0) at (0:2cm)  {$i'$}; 
\node [xshift=5cm] (P2) at (1*72:2cm) {$i$};
 \node [xshift=5cm] (P4) at (2*72:2cm) {$\bullet$}; 
 \node [xshift=5cm] (P1) at (3*72:2cm) {$\bullet$}; 
  \node [xshift=5cm] (P3) at (4*72:2cm) {$\bullet$}; 
  
 \draw[bend left,->]  (P2) to node [auto] {An interstice} (P0);
 \draw[bend left,->]  (P0) to node [auto] {Another interstice} (P3);
\draw[bend left,->]  (P3) to node [auto] {} (P1);
\draw[bend left,->]  (P1) to node [auto] {} (P4);
\draw[bend left,->]  (P4) to node [auto] {} (P2);
               
 \end{tikzpicture}
\end{center}

On the other hand, as we have seen in \S \ref{subsec:morita}, the Morita homotopy category $\Ho(\dgcat^{(2)})$ admits a
 duality functor
\[
(-)^\vee: \Ho(\dgcat^{(2)})^{\op} \lra \Ho(\dgcat^{(2)}), \quad 
\A \mapsto \A^\vee =    \IRHom(\A, \Perf^{(2)}_\k). 
\]
The goal of this section is to understand how the functor
\[
  \MF: \LCat \lra \dgcat^{(2)}
\]
relates these duality functors.

To this end, for  a $\L$-category $\C$ we define a dg-category $\widetilde \LF(\C)$ completely analogous
to $\LF(\C)$ but with the substitutions $(\varphi,\psi) \mapsto  (\psi, -\varphi)$ and 
$(\varphi',\psi') \mapsto  (\psi', -\varphi')$
in the formulas \eqref{eq:pathendg} for the differentials on the mapping complexes. To signify this
sign change in a suggestive way, we denote the objects of $\widetilde \LF(\C)$ by 
\[
	\xymatrix{
	  x_1 \ar@/_1ex/[r]_{\psi} & x_0 \ar@/_1ex/[l]_{-\varphi}
	}.
\]

\begin{rem} \label{rem:intduality}  Assume that $\C$ itself has a
``duality"  functor which a $\L$-equivalence  $\C^{\op} \lra \C^{\vee}, x \mapsto x^{\vee}$.  Then dualizing an object 
	  \[
		  \xymatrix{
		  x_1 \ar@/_1ex/[r]_\varphi & x_0 \ar@/_1ex/[l]_\psi
	  }
	  \]
	  of $\LF(\C)$ using the Koszul signs rule yields the object
	  \[
		\xymatrix{
			x_1^{\vee} \ar@/_1ex/[r]_{\psi^{\vee}} & x_0^{\vee} \ar@/_1ex/[l]_{-\varphi^{\vee}}
		}
	  \]
	  of $\widetilde \LF(\C^{\vee})$. In such a case, this association extends to an
	  isomorphism of dg-categories 
	   $ \LF(\C)^{\op} \cong \widetilde \LF(\C^{\vee})$. 
         
\end{rem}

Given a $\L$-category $\C$, there is a natural functor of categories
\begin{equation}\label{eq:tensor}
    \C^{\vee} \times \C \lra \on{Vect}_\k,  \,\, (F,x) \mapsto F \otimes x
\end{equation}
where we define $F \otimes x$ to be the free $\k$-vector space on the set $\Hom_{\Q^0}(0,F(x))$.

 \begin{defi} Let $\C$ be a $\L$-category.  We define a dg functor
  \[
    \eta_{\C}: \widetilde \LF(\C^{\vee}) \otimes \LF(\C) \to \Mod^{(2)}_\k
  \] 
  by the formula
  \[
	(\xymatrix{
	  F_1 \ar@/_1ex/[r]_{\xi} & F_0 \ar@/_1ex/[l]_{-\theta}
	},
	  \xymatrix{
	  x_1 \ar@/_1ex/[r]_\varphi & x_0 \ar@/_1ex/[l]_\psi
	  }
	) \mapsto 
	(\xymatrix{
	  F_1 \ar@/_1ex/[r]_{\xi} & F_0 \ar@/_1ex/[l]_{-\theta}
	}) \otimes (
	  \xymatrix{
	  x_1 \ar@/_1ex/[r]_\varphi & x_0 \ar@/_1ex/[l]_\psi
	  })
  \]
  where the tensor product is given by \eqref{eq:tensor} and differential determined by the Koszul
  rule where $F_0$ and $x_0$ are considered of even degree while $F_1$ and $x_1$ are considered of odd
  degree. 
\end{defi}
	
\begin{prop} \label{prop:dualfunctoriality} For any $\L$-functor $F:\C \to \D$ of $\Zp$-categories, the diagram of dg categories
  \[
    \xymatrix{
      \widetilde \LF(\D^{\vee}) \otimes \LF(\C)\ar[d]^{(F^{\vee},\id)} \ar[r]^{(\id,F)} & \widetilde \LF(\D^{\vee}) \otimes \LF(\D)
      \ar[d]^{\eta_{\D}}\\
      \widetilde \LF(\C^{\vee}) \otimes \LF(\C) \ar[r]^{\eta_{\C}} & \Mod^{(2)}_\k
    }
  \] 
  commutes.
\end{prop}
\begin{proof}
	This follows immediately from the fact that the diagram of categories
	  \[
	    \xymatrix{
		    \D^{\vee} \times \C \ar[r]^{(\id,F)}\ar[d]_{(F^{\vee}, \id)} &  \D^{\vee} \times \D \ar[d]^{\otimes}\\
		    \C^{\vee} \times \C \ar[r]^{\otimes} &  \on{Vect}_\k
	    }
	  \]
	commutes.
\end{proof}

\begin{prop} 
	Let $\C=\Q^n$ Then the Yoneda
	embedding
	\[
		\C^{\op} \lra \IFun_{\L}(\C, \LSet)
	\]
	factors over the canonical embedding $\C^{\vee} = \IFun_{\L}(\C, \co{0}) \to \IFun_{\L}(\C,
	\LSet)$ and induces a $\L$-equivalence between $\C^\op$ and the dual $\L$-category $\C^\vee$,
	 and hence, by Remark \ref{rem:intduality}, an isomorphism of dg-categories
	\[
		\LF(\C)^{\op} \cong \widetilde \LF(\C^{\vee}).   \hskip 2cm\qed
	\]
	
\end{prop}

\begin{prop} Let $\C=\Q^n$. Then there is
	a commutative diagram
  \[
	  \xymatrix@C=7pc{
		  \widetilde \LF(\C^{\vee}) \otimes \LF(\C) \ar[r]^-{\eta_{\C}} & \Mod^{(2)}_\k \\
		  \LF(\C)^{\op} \otimes \LF(\C) \ar[r]_-{\Hom_{\LF(\C)}(-,-)} \ar[u]^{\cong} &
		  \Mod^{(2)}_\k\ar[u].
	  }
  \] 
  In particular, the functor
  \[
	\widetilde \LF(\C^{\vee}) \lra \Fun(\LF(\C), \Mod^{(2)}_\k),
  \]
  adjoint to $\eta_{\C}$, factors over $\IRHom(\LF(\C), \Perf_\k) \subset \Fun(\LF(\C), \Mod_\k^{(2)})$. The
  induced functor 
  \[
	\LF(\C^{\vee}) \lra \IRHom(\LF(\C), \Mod^{(2)}_\k) 
  \]
  is a Morita equivalence which exhibits the dg-category $\widetilde \LF(\C^{\vee})$ as the Morita dual of
  $\LF(\C)$. 
\end{prop}

\begin{cor}\label{cor:cyc-mor}  We have  a commutative diagram 
  \[
	  \xymatrix@C=7pc{
		  \bf\Lambda^{\op} \ar[r]^{\LF} \ar[d]_{*} &
		  \Ho(\dgcat^{(2)})^{\op} \ar[d]^{\vee} \\
		  \bf\Lambda \ar[r]^{\LF^{\op}} & \Ho(\dgcat^{(2)})
	  }
  \] 
  	relating the duality functors on $\bf\Lambda$ and $\Ho(\dgcat^{(2)})$.
\end{cor}

\section{Cyclic 2-Segal objects}
\label{section:cyclicsegal} 
\subsection{The 1- and 2-Segal conditions}\label{subset:simplicial-cyclic}
We denote by $\Delta$ the category of finite nonempty ordinals $[n]=\{0,1,...,n\}$ and monotone
maps. For a category $\Cb$, a simplicial object $X$ in $\Cb$ is defined to be a functor $X:
\Delta^{\op} \to \Cb$ and we denote by $\Cb_\Delta$ the category of simplicial objects in $\Cb$.
Similarly, a cyclic object $X$ in $\Cb$ is defined to be a functor $X: \Lambda \to \Cb$, where
$\Lambda$ denotes Connes' cyclic category. We denote by $\Cb_{\Lambda}$ the category of
cyclic objects in $\Cb$.
There is an embedding
\[
    \varsigma: \Delta \lra \Lambda
\]
which associates to a finite ordinal the cyclic ordinal corresponding to it by
cyclic closure. That is, $\varsigma([n]) = \cn$. Thus a cyclic object $X$ in $\Cb$ gives rise to a
simplicial object $\varsigma^*X$ (often also denoted $X$) together with endomorphisms
(cyclic rotations)
$t_n: X_n\to X_n$,  satisfying the well known identities \cite{connes}, Ch. III, App. A:
\[
\begin{gathered}
t_n^{n+1}=\Id,  \cr
\partial_i t_n = t_{n-1}\partial_{i-1}, \quad \text{for} \quad
1\leq i\leq n, \quad \text{while}\quad \partial_0 t_n =\partial_n,
 \cr
 s_i t_n = t_{n+1}s_{i-1} 
\quad \text{for} \quad
 1\leq i\leq n, 
 \quad \text{while}\quad
 s_0 t_n = t_{n+1}^2 s_n.  
\end{gathered}
\]
%
%
%
%
Suppose now that $\Cb$ has limits. For a simplicial set $K$ and a simplicial object $X\in\Cb_\Delta$ we denote,
following \cite{HSS1}, the {\em space of $K$-membranes} in $X$ as the object
\begin{equation}\label{eq:pointwise}
(K,X)  = \pro_{\{\Delta^p \to K\}}^{\Cb} X_p \in \Cb\text{.}
\end{equation}
Here the limit is taken over all simplices of $K$. The functor 
\[
  \Upsilon_*X: \Set_\Delta^\op \lra \Cb, \quad K\mapsto (K,X),
\]
is the right Kan extension of $X:\Delta^\op\to\Cb$ along the Yoneda embedding
$\Upsilon: \Delta^\op\to\Set_\Delta^\op$. 

Suppose now that $\Cb$ carries a model structure. Then we can define derived functors
of the projective limit of $\Cb$-valued diagrams \cite{dhks}, and will refer to them
as {\em homotopy limit} functors. We will use the notation
$\hopro_{a\in A}^{\Cb} Z_a$ for the homotopy limit of a diagram $(Z_a)_{a\in A}$. 
Similarly for the derived functors of  Yoneda extension functors,
see {\em loc. cit.}  In
particular, we define the {\em derived space of $K$-membranes} in $X$, denote
$(K,X)_R$ as 
\begin{equation}\label{eq.dpointwise}
  (K,X)_R \simeq \hopro_{\{\Delta^p \to K\}}^{\Cb} X_p \in \Cb.
\end{equation}
See  \cite[\S 5.1]{HSS1} for more details. 
We will need two particular examples of simplicial sets.

\begin{exas}
(a) 
 We denote by $I[n]\subset\Delta^n$ be the simplicial set (``subdivided interval") corresponding to the oriented graph
 \[
	\xymatrix{
		\overset{0}{\bullet} \ar[r] & \overset{1}{\bullet} \ar[r] & \dots \ar[r] &
		\overset{n}{\bullet}
	}.
\]
 (b) Let $P_{n+1}$ be the standard plane $(n+1)$-gon with the set of vertices $M=\{0,1,\cdots , n\}$,
 as in Example \ref{ex:triangles-in-polygon}. Let
 $\Tc$ be any triangulation of $P_{n+1}$. By lifting any triangle $\sigma\in\Tc$ with vertices $i,j,k$ to the
 triangle $\Delta^\sigma\subset\Delta^n$ with vertices $\{i\}, \{j\},\{k\}$, we associate to $\Tc$ a
 2-dimensional simplicial subset $\Delta^\Tc\subset \Delta^n$ homeomorphic to $P_{n+1}$. 
\end{exas}

We now recall the main definitions of \cite{HSS1}, the first one being a modification of that of Rezk \cite{rezk}. 

\begin{Defi}\label{def:segal} 
Let $\Cb$ be a combinatorial model category, and let $X\in\Cb_\Delta$ be a simplicial object.
\begin{enumerate}

	\item We say that $X$ is {\em 1-Segal} if, for every $n\geq 1$, the morphism
	\[
	f_n: X_n \lra (I[n], X)_R = X_1\times^R_{X_0} X_1 \times^R_{X_0} \cdots \times^R_{X_0} X_1,
	\]
	induced by the embedding $I[n]\hookrightarrow \Delta^n$, is a weak equivalence in $\Cb$. 

	\item We say that $X$ is {\em 2-Segal} if, for every $n\geq 2$ and every triangulation $\Tc$ of $P_{n+1}$,
	the morphism
	\[
	f_\Tc: X_n \lra (\Delta^\Tc, X)_R,
	\]
	induced by the embedding $\Delta^\Tc\hookrightarrow \Delta^n$, is a weak equivalence in $\Cb$. 

\end{enumerate}
\end{Defi}

Note the particular case when $\Cb$ is an ordinary category with trivial model structure.
In this case the conditions involve underived membrane spaces
and say that the corresponding morphisms $f_n$, resp. $f_\Tc$, are isomorphisms. 

We now recall the {\em path space criterion}, a result from \cite{HSS1} which is very useful to
establish the 2-Segal property in many cases. For ordinals $I$ and $J$ their {\em join} is the
set $I\sqcup J$ ordered so that each element of $I$ precedes each element of $J$.  For a simplicial
object $X$ its {\em initial and final path spaces} are the simplicial objects $\PI X$ and $\PF X$
induced from $X$ by pullback along the endofunctors
\[
i, f: \Delta\lra\Delta, \quad i(I) = [0]*I, \quad f(I) = I*[0]. 
\]
\begin{thm}[6.3.2 in \cite{HSS1}]\label{thm-path-crit}
A simplicial object $X$ is 2-Segal if and only if both $\PI X$ and $\PF X$
are 1-Segal.\qed
\end{thm}

A cyclic object  $X$ in $\Cb$ will be called $p$-Segal, if $\varsigma^*X$ is a $p$-Segal simplicial object. 
We will say that a covariant functor $\Delta\to\Cb$ or $\Lambda\to\Cb$ is $p$-{\em coSegal,}
if the corresponding contravariant functor with values in $\Cb^\op$ is $p$-Segal.

\subsection{Examples of cyclic 2-Segal objects}

\begin{ex}[(The cyclic nerve and the $\L$-nerve)]
For a small  category $\C $ its {\em  cyclic nerve} $\NC(\C)$, is the cyclic set defined by
\[
\NC_n(\C) = \Fun(\Q^n, \C)
\]
where $\Q^n$ denotes the underlying category of the $\L$-category from Example \ref{exa:circular}.
In other words, $\NC_n(\C)$ is the set of cyclic chains of morphisms
\[
x_0 \to x_1 \to x_2 \to \cdots \to x_n \to x_0
\]  
in $\C$. 
Similarly, for a small $\L$-category $\C=(\C,w)$ its {\em $\L$-nerve} $\N^\L(\C)$, is the cyclic set
defined by
\[
\N^\L_n(\C) = \Fun_\L(\Q^n, \C). 
\]
In other words, $\N^\L_n(\C)$ is the set of cyclic chains of morphisms as above, which form a
factorization of $w$, i.e., each composition $x_i \to\cdots\to x_i$ around the circle is equal to
$w_{x_i}$.   
\end{ex}
\begin{prop} (a) For any small category $\C$ the cyclic set $\NC \C$ is unital 2-Segal. 

(b) For any small $\L$-category $\C$
the cyclic set $\N^\L\C$ is unital 2-Segal. 
\end{prop}
\begin{proof}
The first statement is proved in \cite[Th. 3.2.3]{HSS1}, the second statement is proved similarly.
\end{proof}
  
%
%
%
%

We can now state the main result of this section. 
Consider the cocyclic object $\E^\bullet$ in $\dgcatt$ from Proposition \ref{prop:E-n-cocyclic}.
We equip $\dgcatt$ with the Morita model structure introduced in
\S \ref{subsec:morita}, so that weak equivalences are Morita equivalences.
We define a cyclic object $\E_{\bullet}$ which is defined by the formula
\[
	\E_{\bullet}: \Lambda^{\op} \to \dgcatt, \; \cn \mapsto (\E^{\cn^*})^{\op}
\]
where $\cn \mapsto \cn^*$ denotes the cyclic duality from \S \ref{subsec:cyc-morita}. Note that, by
the compatibility of cyclic and Morita duality established in Corollary \ref{cor:cyc-mor}, we have a
levelwise Morita equivalence 
\begin{equation}\label{eq:duale}
	\E_{\bullet} \simeq \IRHom(\E^{\bullet}, \Perft_{\k})
\end{equation}
of cyclic objects in $\dgcatt$.

\begin{thm}\label{thm:E-cosegal} The following statements hold:
	\begin{enumerate}
		\item[(a)] The cocyclic object $\E^\bullet$ in $(\dgcatt, \Mor)$ is $2$-coSegal.
		\item[(b)] The cyclic object $\E_{\bullet}$ in $(\dgcatt, \Mor)$ is $2$-Segal.
	\end{enumerate}
\end{thm}
\begin{proof} From the adjunction \eqref{eq:simplicial-adjunction}, we deduce that $\IRHom(-,-)$
	maps homotopy colimits in the first variable to homotopy limits. Therefore, in light of
	\eqref{eq:duale}, (b) follows immediately from (a).
	To show (a), we use Theorem \ref{thm-path-crit} to reduce to the statement that the cosimplicial objects $\PI(\E^\bullet), \PF(\E^\bullet)$
	are $1$-coSegal. We consider the case of $\PI(\E^\bullet)$. By definition, we have
	\[
	(\PI(\E^\bullet))^n = \E^{n+1}.
	\]
	The key point of the argument is now that the Morita equivalences
	\[
		r_{n+1}^\triangleleft: \A^{n+1} \overset{\simeq}{\lra} \E^{n+1}
	\]
	from Proposition \ref{prop:r-n-morita} assemble to give a weak equivalence of cosimplicial objects
	\[
		\A^{\bullet+1} \overset{\simeq}{\lra} \PI(\E^\bullet)
	\]
	in $(\dgcatt, \Mor)$ where the cosimplicial structure of $\A^{\bullet + 1}$ is obtained in an
	obvious way with coface maps given by composing morphisms and codegeneracies by filling in
	identity morphisms. Therefore, it suffices to show that the cosimplicial object
	$\A^{\bullet+1}$ is $1$-Segal. 
	Since homotopy fiber products in $(\dgcatt)^\op$ translate to homotopy pushouts in $\dgcatt$
	this amounts to verifying that, for every $n \ge 1$, the $1$-coSegal map
	\[
	\hoind \bigl\{ \A^2\leftarrow \A^1 \to \A^2\leftarrow\A^1\to \cdots \leftarrow\A^1\to\A^2\bigr\} 
	\lra \A^n
	\]
	is a Morita equivalence. Since the maps $\A^1 \to \A^2$
	appearing in the homotopy colimit are cofibrations in $(\dgcatt, \Qeq)$ (and hence in
	$(\dgcatt, \Mor)$), we may replace the homotopy colimit by an ordinary colimit. The
	resulting statement is clearly true.
\end{proof}

\subsection{Background on triangulated surfaces and ribbon graphs}\label{subsec:back-triang-surf}
Here we collect some well-known material on surfaces and their triangulations.
More details can be found in \cite{fock-goncharov, fomin-shapiro-thurston} and references therein. 
   
\subsubsection{Marked oriented surfaces}
By a {\em surface} we mean a compact, connected, oriented $2$-dimensional smooth manifold $S$
with boundary, denoted $\partial S$. We denote by $T^2, S^2$ and $D^2$ the 2-dimensional torus,
sphere, and disk, respectively.
   
\begin{Defi} A {\em stable marked surface} is a pair $(S,M)$ where $S$ is a surface and $M\subset
	S$ is a nonempty finite subset of points such that:
\begin{enumerate}
	\item Each component of $\partial S$ contains at least one point from $M$.

	\item The following {\em unstable cases} are excluded:
		\begin{enumerate}
			\item $S$ is diffeomorphic to  $S^2$, $|M|\leq 2$,
			\item $S$ is diffeomorphic to $D^2$, $|M| =  1$, or $|M|=2$ and
			$M\subset\partial S$.
		\end{enumerate}
\end{enumerate}
\end{Defi}

In the sequel all marked surfaces will be assumed stable, unless
indicated otherwise. 
For a marked surface $(S,M)$ we have the groups
\[
\on{Diffeo}^+(S,M), \quad   \mathfrak{G}(S, M)\,=\,\pi_0  \on{Diffeo}^+(S,M)
\]
of orientation preserving diffeomorphisms $S\to S$ preserving $M$ as a set,
and of isotopy classes of such diffeomorphisms. The group $ \mathfrak{G}(S, M)$
is known as the {\em mapping class group} of $(S,M)$.

\begin{rem}\label{rem:real-blow-up}
It is often convenient to view interior marked points $x\in M-\partial S$
as {\em punctures}, by removing them to form the noncompact surface $S-(M-\partial S)$. 

Further, when representing surfaces by Ribbon graphs (see \S \ref{subsub:rib}), it will be
convenient to transform $(S,M)$ into a new compact surface where
\begin{enumerate}
	\item the interior marked points become closed boundary components,
	\item the marked points on $\partial S$ become closed intervals on the boundary.
\end{enumerate}
The construction which naturally performs the modifications (1) and (2) is called the {\em real
blowup} $S_M$ of $S$ along $M$ (\cite[\S 2.1]{andre}), obtained by adding the set of inward tangent directions
at each $x\in M$. It is further convenient, to form a noncompact surface by removing the open
boundary intervals in the complement of the blown up marked boundary points creating {\em open ends}
of the surface. For example, Figure \ref{fig:rb} displays the marked surface given by a disk with
one interior and one boundary marked point, as well as its real blowup which is an annulus with an open end
on one of its boundary components.

\begin{figure}[ht]
\centering
\begin{tikzpicture}[baseline=-10ex,>=latex,scale=0.5]
\begin{scope}

	\fill[fill=green!20!white] (0:2cm) arc (0:360:2cm);
	\draw[thick] (0:2cm) arc (0:360:2cm);
	\fill (0:0cm) circle [radius=0.2cm];
	\fill (0:2cm) circle [radius=0.2cm];

	\node at (0:5cm) {$\Rightarrow$};
\end{scope}
\begin{scope}[decoration={
    markings,
    mark=at position 0.5 with {\arrow{>}}},xshift=10cm,
    ] 

\fill[fill=green!20!white] (0:1cm) -- (0:2cm)
	arc (0:360:2cm) -- (0:1cm)
	arc (360:0:1cm) -- cycle;

\fill[fill=green!20!white] (160:2cm) .. controls (165:2cm) and (175:2cm) .. (170:2.5cm)
		-- (-170:2.5cm) .. controls (-175:2cm) and (-165:2cm) .. (-160:2cm)
		-- cycle;

	\draw[thick] (0:1cm) arc (0:360:1cm); 
	\draw[thick] (0:2cm) arc (0:160:2cm);
	\draw[thick] (0:2cm) arc (0:-160:2cm);
	\draw[thick, dotted] (170:2.5cm) -- (-170:2.5cm);
	\draw[thick] (160:2cm) .. controls (165:2cm) and (175:2cm) .. (170:2.5cm)
		(-170:2.5cm) .. controls (-175:2cm) and (-165:2cm) .. (-160:2cm);

\end{scope}
\end{tikzpicture}
\caption{Disk with two marked points and corresponding real blowup}\label{fig:rb}
\end{figure}
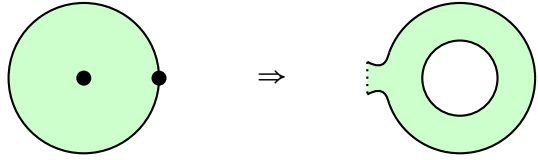
  
\end{rem}
     
\begin{ex}\label{ex:S,M-polygon}
As a simple but important case, our definition of a marked surface incudes $(P_{n+1}, M)$ where
$P_{n+1}, n\geq 2$, is a convex $(n+1)$-gon in the plane, and $M$ is its set of vertices. Via a
homeomorphism with the closed disk, this is a smooth manifold with boundary. 
We have $\mathfrak{G}(P_{n+1}, M) = \ZZ/(n+1)$. 
\end{ex}

\begin{Defi}\label{def:simple-arcs}
A {\em simple curve} on a marked surface $(S, M)$ is a continuous map $\gamma:[0,1]\to M$ with the
following properties:
\begin{enumerate}
\item The endpoints $\gamma(0),\gamma(1)$ lie in $M$. They can coincide.

\item Except for possible coincidence of the endpoints, $\gamma$ does not intersect itself,
nor $M$.  

\item If the endpoints coincide, $\gamma(0)=\gamma(1)=x$, then $\gamma$ gives a nontrivial element of
the fundamental group $\pi_1(S-M\cup\{x\}, x)$. 
\end{enumerate}
A {\em (simple) arc} on $(S, M)$ is an equivalence class of simple curves under isotopies and
reversal of parametrization. We denote by $\AAA(S,M)$ the set of arcs.
An {\em oriented arc} on $(S, M)$ is an equivalence class of simple curves under isotopies.
We denote by $\Delta(S,M)$ the set  of oriented arcs. 
\end{Defi} 
    
This  definition differs from \cite[Def. 2.2]{fomin-shapiro-thurston} in that we allow, as arcs, segments joining adjacent marked 
points on the same boundary component. This does not affect the validity of the results we need, while allowing
for a more suggestive interpretation of $\Delta(S,M)$, as the following examples show. 

\begin{exas} (a) In the situation of Example \ref{ex:S,M-polygon}, 
$\Delta(P_{n+1},M)$ is identified with the root system of type $A_n$. 

\vskip .2cm

(b) Let $S=T^2$ be a torus and $M$ consist of one point, denoted 0. 
An oriented  arc $\gamma$ has a homology class $[\gamma]\in H_2(T^2, \ZZ)=\ZZ^2$
which must be a primitive vector of the lattice $\ZZ^2$.  
Thus $\Delta(S,M)$ is identified
with the set of primitive vectors in $\ZZ^2$, and $\AAA(S,M)$ is identified with    $\PP^1(\QQ)$. 
\end{exas}
     
\subsubsection{Triangulations as systems of arcs}\label{par:triangulations-arcs}
Two arcs are called {\em compatible} if there are simple curves representing them which do not
intersect in $S-M$. It is known \cite[Prop. 2.5]{fomin-shapiro-thurston} that any collection of
pairwise compatible arcs can be represented by a collection of simple curves which pairwise do not
intersect in $S-M$. 
An {\em ideal triangulation} of $(S,M)$ is  defined as a maximal collection of pairwise
compatible arcs. 
    
Pairwise non-intersecting curves from a maximal collection cut $S$ into ``ideal triangles with
vertices in $M$",  which are regions $\sigma$ diffeomorphic to the interior of the  standard
plane triangle $P_3$. Each such $\sigma$ comes with a canonical 3-element set
$\on{Vert}(\sigma)$ of ``intrinsic vertices" (or ``corners") which is equipped with a cyclic
order via the orientation of $S$.  Note that different elements of $\on{Vert}(\sigma)$ may
correspond to the same element of $M$, i.e., the vertices (and even edges) of a triangle can
become identified, see Fig.\ref{fig:vert-sigma1}.
Similarly, each arc $a$ comes with a 2-element set $\on{Vert}(a)$ of intrinsic vertices
(``half-edges") which can become identified in $S$, if $a$ is a loop. 
  
\begin{figure}\label{fig:vert-sigma1}
\begin{tikzpicture}[scale=0.5, baseline=(current  bounding  box.center)]] 

\node (A) at (0,-3){};

\node (B) at (-5,0){};

\node (C) at (0,3){}; 

\fill (A) circle (0.1);

\fill (B) circle (0.1);

\fill (C) circle (0.1);

\draw[line width=.1mm] (0,-3) -- (-5,0) -- (0,3) -- (0,-3); 

\node at (0.8,-3){$A$};
\node at (0.8,3){$C$};
\node at (-5,.8){$B$};

\node at (-2,0){$\sigma$};

\node (B1) at (10,0){};
\fill (B1) circle (0.1);

 \node (C1) at (7,2){};
\fill (C1) circle (0.1);

 \node (A1) at (7,-2){};
\fill (A1) circle (0.1);

\draw[line width=.1mm] (7,2) .. controls (14,2) and (14,-2) .. (7,-2);

\draw[line width=.1mm] (7,2) -- (10,0) -- (7,-2); 

\node at (6.5,2.5){$C$};
\node at (6.5,-2.5){$A$};
\node at (10.7,0){$B$}; 
\node at (10,1){$\sigma$}; 

\node (B2) at (20,0){}; 
\fill (B2) circle (0.1);
\draw (B2) circle (3.5);

\node(A2) at (16.5,0){};
\fill (A2) circle (0.1); 
\draw (16.5,0) -- (20,0); 

\node at (20.5,0.5){$B$}; 

\node at (17.1 , 0.6){$C$}; 
\node at (17.1 , -0.6){$A$}; 

\node at (21, -2){$\sigma$}; 


\end{tikzpicture}
\caption{$\on{Vert}(\sigma)=\{A,B,C\}$ in all cases.}\label{fig:vert-sigma}
\end{figure}
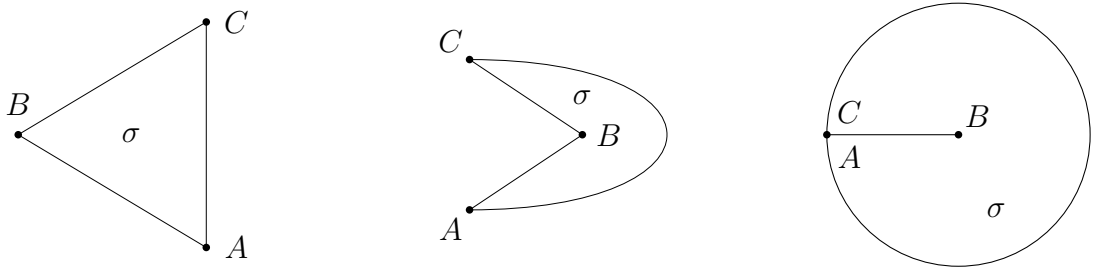

\subsubsection{Triangulations, tesselations, and spanning graphs}

An ideal triangulation $\Tc$ can be encoded by its {\em dual graph} $\Gamma_\Tc$ obtained by putting
one vertex $v_\sigma$ inside each triangle $\sigma$ and joining the $v_\sigma$ by edges
corresponding to common edges of the triangles, see \eqref{eq:dual graph}. As this construction
allows a uniform treatment of all degenerate cases, we recall precise definitions.

\begin{Defi} 
\begin{enumerate}
\item[(a)] A {\em graph}  $\Gamma$ is a  finite, 1-dimensional CW-complex without isolated points.
	For a vertex $v\in\on{Vert}(\Gamma)$ we denote $\on{Ed}(v)$ the set of germs of edges at $v$
	(a loop beginning and ending at $v$ gives rise to two germs of edges at $v$). The
	cardinality of $\on{Ed}(v)$ is called the {\em valency} of $v$. The set of 1-valent vertices
	is denoted by $\partial \Gamma$.  A graph is called {\em 3-valent} if all vertices have
	valence 1 or 3. 

\item[(b)]  Let $(S,M)$ be a   marked surface.  A {\em spanning graph} for $(S,M)$ is an embedded
	graph $\Gamma\subset S-M$ such that  $\partial\Gamma\subset \partial S$ and both maps 
	\[
	\Gamma\lra S-M, \quad \partial\Gamma\lra \partial S-M 
	\] 
	are homotopy equivalences. 

\end{enumerate}
\end{Defi}

\begin{prop}\label{prop:triangulations-graphs} Let $(S,M)$ be a marked surface. 
Forming the dual graph defines a bijection $\Tc\mapsto\Gamma_\Tc$ between ideal triangulations of
$(S,M)$ and isotopy classes of 3-valent spanning graphs for $(S,M)$.  \qed
\end{prop}
  
We will further call a {\em tessellation} of $(S,M)$ an isotopy class of (not necessarily 3-valent)
spanning graphs $\Gamma$ for $(S,M)$.
Each such $\Gamma$ encodes a decomposition of $S$ into curvilinear polygons with vertices in $M$,
one polygon for each vertex $v\in\Gamma$ of valence $\geq 3$.

\subsubsection{Ribbon graphs, Stasheff polytopes, and the tessellation complex} \label{subsub:rib}

A {\em ribbon graph} is a graph $\Gamma$ together with a choice of a total cyclic order on each
set $\on{Ed}(v)$, $v\in\on{Vert}(\Gamma)$. As any graph embedded into an oriented surface, a spanning
graph for $(S,M)$ has a natural ribbon structure. 
     
Conversely, a ribbon graph $\Gamma$ gives rise to an oriented surface with boundary
$\Sigma_\Gamma$   as follows.  Each vertex $v$ of $\Gamma$ corresponds to a ribbon corolla as
illustrated in Figure \ref{fig:corolla}. 
Further, each edge $e$ of $\Gamma$ corresponds to a ribbon strip as illustrated in Figure
\ref{fig:edge}. 
\begin{figure}[ht]
\begin{minipage}[b]{0.5\linewidth}
\centering
\begin{tikzpicture}[>=latex,scale=1.5]
\begin{scope}[decoration={
    markings,
    mark=at position 0.2 with {\arrow{>}},
    mark=at position 0.8 with {\arrow{>}}}
    ] 

\fill[fill=green!20!white] ({1*360/5+10}:1cm) .. controls ({(1+0.5)*360/5}:0.25cm) and ({(1+0.5)*360/5}:0.25cm) .. ({(1+1)*360/5-10}:1cm) 
-- ({(2*360)/5+10}:1cm) .. controls ({(2+0.5)*360/5}:0.25cm) and ({(2+0.5)*360/5}:0.25cm) .. ({3*360/5-10}:1cm) 
-- ({(3*360)/5+10}:1cm) .. controls ({(3+0.5)*360/5}:0.25cm) and ({(3+0.5)*360/5}:0.25cm) .. ({4*360/5-10}:1cm) 
-- ({(4*360)/5+10}:1cm) .. controls ({(4+0.5)*360/5}:0.25cm) and ({(4+0.5)*360/5}:0.25cm) .. ({5*360/5-10}:1cm) 
-- ({(5*360)/5+10}:1cm) .. controls ({(5+0.5)*360/5}:0.25cm) and ({(5+0.5)*360/5}:0.25cm) .. ({1*360/5-10}:1cm) 
--cycle;

\foreach \s in {1,...,5}
{
\draw[postaction={decorate}] ({\s*360/5+10}:1cm) .. controls ({(\s+0.5)*360/5}:0.25cm) and
({(\s+0.5)*360/5}:0.25cm) .. ({(\s+1)*360/5-10}:1cm);
}

\path (0:0cm) node (o) {$v$};
\path (72+36:0.7cm) node (b0) {$b_0$};
\path (2*72+36:0.7cm) node (b1) {$b_1$};
\path (3*72+36:0.7cm) node (b2) {$b_2$};
\path (4*72+36:0.7cm) node (b3) {$b_3$};
\path (36:0.7cm) node (b4) {$b_4$};

\draw[->, >=latex] (-90:0.3cm) arc (-90:180:0.3cm);

\end{scope}
\end{tikzpicture}
	\caption{Ribbon corolla corresponding to a vertex $v$ of $\Gamma$ of valency $5$}
	\label{fig:corolla}
\end{minipage}
\hspace{0.5cm}
\begin{minipage}[b]{0.45\linewidth}
\centering
\begin{tikzpicture}[>=latex,scale=1.5]
\begin{scope}[decoration={
    markings,
    mark=at position 0.5 with {\arrow{>}}}
    ] 

\fill[fill=green!20!white] 
({2*360/5-10}:1cm) -- ({2*360/5+10}:1cm) -- ({2*360/5+170}:1cm) -- ({2*360/5+190}:1cm) --
cycle;

\draw[postaction={decorate}] ({2*360/5+10}:1cm) -- ({2*360/5+170}:1cm);
\draw[postaction={decorate}] ({2*360/5+190}:1cm) -- ({2*360/5-10}:1cm);

\path (0:0cm) node (o) {$e$};
\path (2*72-90:.4cm) node (o) {$a_0$};
\path (2*72+90:.4cm) node (o) {$a_1$};

\end{scope}
\end{tikzpicture}
\caption{Ribbon strip corresponding to an edge $e$ of $\Gamma$}
	\label{fig:edge}
\end{minipage}
\end{figure}
The ribbon strips are then glued to the ribbon corollas according to the incidence relations provided by $\Gamma$.
As a result of this procedure, we obtain an oriented surface with boundary which we denote by
$\Sigma_{\Gamma}$. 
See, e.g., \cite{penner-book}, Ch. 1, \S 1.3 for the case of graphs without 1-valent vertices. If
$\Gamma$ does have 1-valent vertices, they give rise to open ends at the boundary of
$\Sigma_\Gamma$. If $\Gamma$ is a spanning graph for a stable marked surface $(S,M)$, then
$\Sigma_\Gamma$ is diffeomorphic to the real blowup $\widetilde S_M$ from Remark
\ref{rem:real-blow-up}. 
                
Note that for a ribbon tree $T$ we have a canonical cyclic order on $\partial T$, since in this case
the ribbon structure gives an isotopy class of embeddings into $\RR^2$.  Let $\Gamma, \Gamma'$ be
ribbon graphs. A {\em contraction} $p:\Gamma\to\Gamma'$ is a surjective cellular map such that for
any vertex $v'\in\Gamma'$ the preimage $p^{-1}(v')$ is a sub-tree in $\Gamma$, and the induced map
$\partial(p^{-1}(v'))\to \on{Ed}(v')$ is a bijection preserving the cyclic order.   
       
We denote by $K_n$ the $n$th {\em Stasheff polytope}. Thus, the vertices of $K_n$ are in bijection with the following three
canonically identified sets:
\begin{itemize}

\item[(S)] Bracketings of a product of
$n$ factors. 

\item[(S$'$)] Triangulations of the $(n+1)$-gon $P_{n+1}$.

 \item[(S$''$)] Planar 3-valent trees with $(n+1)$ ``tails" (1-valent vertices) labelled cyclically from $0$ to $n$. 
 \end{itemize}
The bijection between (S$'$) and (S$''$) is a particular case of Proposition
\ref{prop:triangulations-graphs}. 
 More generally, faces of $K_n$ of arbitrary dimension are labelled by
planar trees with the same  tails as in (S') but not necessarily 3-valent.
For instance, edges correspond
to trees with one 4-valent vertex and the rest being 3- (or 1-)valent.
The maximal face corresponds  to a ``corolla"  (tree with  one 
vertex of valency $n+1$). 
Note that both (S') and (S'') allow for the definition of $K_I$ for any
finite cyclic ordinal $I$, with $K_n$ corresponding to $I=\la n \ra$. 
We put $K_I=\pt$ for $|I|\leq 3$. 
      
For a ribbon graph $\Gamma$, we define its {\em Stasheff polytope} as
\[
	K_\Gamma = \prod_{v\in\on{Vert}(\Gamma)} K_{\on{Ed}(v)}. 
\]
A contraction $p:\Gamma\to\Gamma'$ gives rise to an embedding $k_p: K_\Gamma\hookrightarrow
K_{\Gamma'}$, which realizes $K_\Gamma$ as a face of $K_{\Gamma'}$. 
        
Let now $(S,M)$ be an arbitrary marked surface and  denote by $\Sigma(S,M)$ the set if isotopy
classes of all, not necessarily $3$-valent, spanning graphs for $(S,M)$. This set is partially
ordered by {\em degeneration}: we say that $\Gamma \leq \Gamma'$, if $\Gamma'$ can be obtained from
$\Gamma$ by collapsing some edges to points. In this case, the subgraph shrunk to each vertex of
$\Gamma'$ is a tree, so that we have a contraction $\Gamma \to \Gamma'$ and the corresponding face
embedding $K_\Gamma\hookrightarrow K_{\Gamma'}$. 
      
\begin{Defi}
The {\em tessellation complex} $K_{S,M}$ is a CW-complex glued from the cells $K_\Gamma$,
$\Gamma\in\Sigma(S,M)$, using the face identifications induced by degenerations. 
\end{Defi}
      
In particular, vertices of $K_{S,M}$ correspond to ideal triangulations of $(S,M)$, edges correspond
to ``flips" on 4-gons, as in Example \ref{ex:octahedron-flip}, and so on.  More precisely, a cell of
$K_{S,M}$, i.e., an isotopy class of  spanning graphs $\Gamma$, can be seen as encoding a {\em
tessellation} of $(S,M)$, i.e., a decomposition of $S$ into curvilinear polygons with vertices in
$M$, see \cite{penner-book}, Ch.1, Th. 1.25.
       
The mapping class $\mathfrak{G}(S,M)$ acts on $K_{S,M}$ by automorphisms. Crucial for us will be the
following result due to Harer (Theorem 1.1 and Theorem 2.1 in \cite{harer}).


\begin{prop}\label{prop:modular-contractible}
The CW-complex $K_{S,M}$ is contractible. \qed
\end{prop}

\subsection{Cyclic membrane spaces}\label{subset:construction-TFT}

Let $\Cb$ be a combinatorial model category. Given a cyclic object $X \in \Cb_\Lambda$ we define,
for every cyclic set $D \in \Set_\Lambda$, the {\em cyclic membrane space} to be 
\[
 	\llb D,X\rlb   = \pro^\Cb_{\{\Lambda^n\to D\}} X_n, 
\]
As in the simplicial case, this construction can be expressed
in terms of the right Kan extension 
	$\Upsilon_*: \; \Cb_\Lambda \lra \Cb_{\Set_\Lambda}$
along the Yoneda embedding $\Upsilon: \Lambda \to \Set_{\Lambda}$, so that we have a natural
isomorphism
\[
 	\llb D,X\rlb \cong (\Upsilon_*X)(D).
\]
Again, we can derive these constructions, defining the {\em derived cyclic
membrane space} to be 
\[
 	\llb D,X\rlb_R  = \hopro^\Cb_{\{\Lambda^n\to D\}} X_n,
\]
obtaining the description
\[
 	\llb D,X\rlb_R \simeq (R\Upsilon_*X)(D)
\]
in terms of the right homotopy Kan extension along the Yoneda embedding. In particular, we have
$\llb D,X\rlb_R \simeq \llb D, \widetilde X\rlb$,
where $X\to\widetilde X$ is an injectively fibrant replacement of $X$.

\subsubsection{From a triangulated surface to a cyclic membrane} 

Let $(S,M)$ be a stable marked surface, and let $\Tc$ be a triangulation of $(S,M)$.  We denote by
$\Tc_1$ and $\Tc_2$ the set of arcs and triangles of $\Tc$. Recall from \S
\ref{par:triangulations-arcs} that each arc $a\in\Tc_1$ has a 2-element set $\on{Vert}(a)$ of
``intrinsic endpoints", which, as any 2-element set, can be canonically considered as a cyclic
ordinal.  Recall further that each triangle $\sigma\in\Tc_2$ has a 3-element set
$\on{Vert}(\sigma)$ of ``intrinsic vertices" which is made into a cyclic ordinal by the orientation
of $S$. 
Whenever an arc $a$ is a side of a triangle $\sigma$ (notation $a\subset\sigma$), we have an
embedding $u_{a,\sigma}: \on{Vert}(a)\to\on{Vert}(\sigma)$ which can be considered as a morphism in
the category $\bf\Lambda$. In particular, we have an embedding of the cyclic simplices
\[
	(u_{a,\sigma})_*: \Lambda^{\on{Vert}(a)} \lra \Lambda^{\on{Vert}(\sigma)}.
\]
Let $\Tc_{[1,2]}$ be the {\em incidence category} of $\Tc$, with the set of objects being
$\Tc_1\sqcup \Tc_2$, and non-identity morphisms given by  inclusions $a\subset\sigma$. 
Let $U_\Tc:\Tc_{[1,2]}\to\Set_\Lambda$ be the functor sending:
\begin{itemize}
\item an object $a$ to $\Lambda^{\on{Vert}(a)}$,
\item an object $\sigma$ to $\Lambda^{\on{Vert}(\sigma)}$,
\item a morphism $a\subset\sigma$ to the morphism $(u_{a,\sigma})_*$.
\end{itemize}
The {\em cyclic membrane}
corresponding to $\Tc$ is  defined as the colimit
\be
\Lambda^\Tc = \ind^{\Set_\Lambda} U_\Tc
\ee
mimicking the way $S$ itself is glued out of triangles $\sigma\in\Tc_2$ identified along arcs $a\in\Tc_1$. 


\begin{rem}
The geometric realization $|\Lambda^\Tc|$ of (the simplicial set corresponding to) the cyclic set
$\Lambda^\Tc$ is a 3-dimensional manifold with boundary.  As showed in \cite{dwyer-hopkins-kan}, the
realization of any cyclic set has a natural $S^1$-action. In our case $|\Lambda^\Tc|$  is an
$S^1$-bundle over $S$ which is obtained from the tangent circle bundle by performing a surgery at
each point of $M$ (in particular, if $S$ is a compact surface of genus $g$, then the degree of the
bundle is $2-2g-|M|$). Among other things, this means that $|\Lambda^\Tc|$ is independent on $\Tc$,
up to homeomorphism. 
If $S$ is equipped with a holomorphic structure, then  $|\Lambda^\Tc|$ can be identified with the
circle bundle corresponding to the holomorphic line bundle
\[
T_S(\log M) \,=\, (\Omega^1_S(\log M))^*
\]
whose sections are holomorphic vector fields on $S$ vanishing on $M$. This fact can be obtained by
carefully analyzing the case when $S$ is a triangle (with a complex structure) and $M$ is the set of
its 3 vertices. In this case by \cite{drinfeld}, the interior of $|\Lambda^M|=|\Lambda^2|$ is the
space of cyclically monotone embeddings $M\to S^1$. The Riemann Mapping Theorem identifies this
space with the space of biholomorphisms $f$  from $S$ to the unit disk $D=\{|z|\leq 1\}$ (such $f$
is uniquely determined by the images of three points on the boundary). Another way of determining
$f$ is by prescribing an interior point $s\in S$ (sent by $f$ to $0$) and a tangent direction at $s$
(sent by $d_sf$ to the tangent direction of $\RR_+$ at $0$). This provides an identification of the
interior of $|\Lambda^M|$ with the tangent circle bundle of the interior of $S$. We omit further
details. 
\end{rem}

\subsubsection{From a ribbon graph to a cyclic membrane}  

We provide a dual description of the association $\Tc \mapsto \Lambda^{\Tc}$ in terms of ribbon
graphs which easily allows us to generalize it to more general polygonal subdivisions.

Let $\Gamma$ be a ribbon graph. For a vertex $v$ of $\Gamma$ let $B(v)$ be the set of oriented arcs
comprising the local boundary of the ribbon corolla of $v$.  See Figure \ref{fig:corolla} where this
set is denoted $\{b_0,\dots,b_n\}$. Note that $B(v)$ has a natural cyclic order inherited from that
on $\on{Ed}(v)$, the set of half-edges edges incident to $v$.
More precisely, $B(v)=\on{Ed}(v)^*$ is the set of interstices in $\on{Ed}(v)$. 
We denote by $\Lambda^{B(v)}$ the cyclic simplex corresponding to $B(v)$. 
  
For an edge $e$   of $\Gamma$,  let $B(e)$ be the $2$-element set of components of the ribbon strip
$e$. See  Figure \ref{fig:edge} where this set is denoted $\{a_0, a_1\}$. Invariantly,
$B(e)=\on{Vert}(e)^*$, where $\on{Vert}(e)$ is the 2-element set of endpoints of $e$ (considered
distinct even if $e$ is a loop).  
As any $2$-element set, $B(e)$ has a unique  total cyclic order.
We associate to $e$ the cyclic $1$-simplex $\Lambda^{B(e)}$. 

For a flag $(v,e)$ consisting of a vertex and an edge of $\Gamma$, we have an inclusion of cyclic
ordinals $u_{v,e}: B(e)\to B(v)$ and the corresponding embedding of the cyclic simplexes
\[
	(u_{v,e})_*: \Lambda^{B(e)}\to\Lambda^{B(v)}. 
\]
Let $\Gamma_{[0,1]}$be the incidence category of $\Gamma$ with the set of objects being
$\on{Vert}(\Gamma)\sqcup\on{Ed}(\Gamma)$ and morphisms being incidence inclusions. 
As before, we get a functor $U_\Gamma: \Gamma_{[0,1]}^\op\to\Set_\Lambda$ sending $e$ to
$\Lambda^{B(e)}$, $v$ to $\Lambda^{B(v)}$ and an incidence $v\subset e$ to $u_{v,e*}$. We define the
cyclic membrane corresponding to $\Gamma$ as
 \[
	\Lambda^\Gamma = \ind^{\Set_\Lambda} U_\Gamma.
\]

\begin{ex}
If $\Tc$ is a triangulation of $(S,M)$ and $\Gamma$ is its dual ribbon graph, then
$\on{Vert}(\Gamma)=\Tc_2$, $\on{Ed}(\Gamma)=\Tc_1$, the functor $U_\Gamma$ is identified with
$U_\Tc$, and $\Lambda^\Gamma$ with $\Lambda^\Tc$. 
\end{ex}
 
Let $p: \Gamma \to \Gamma'$ be a contraction of ribbon graphs. For any vertex $v'$ of $\Gamma$, we
have a canonical map
\[
\coprod\limits_{v \in p^{-1}(v')} \Lambda^{B(v)} \to \Lambda^{B(v')}
\]
of cyclic sets. These maps induce an inclusion of cyclic membranes $\Lambda^\Gamma \to
\Lambda^{\Gamma'}$. Since this association is functorial, we obtain the following result.

\begin{prop} \label{prop:cycmem} The cyclic membrane construction 
	\[
	\Rib \lra \Set_{\Lambda},\; \Gamma \mapsto \Lambda^\Gamma
	\]
	extends to a functor on the category of Ribbon graphs with contractions as morphisms.\qed
\end{prop}

\subsubsection{Mapping class group actions} 

Let $X \in \Cb_{\Lambda}$ be a cyclic object and $\Gamma$ a Ribbon graph. We define
\[
RX_{\Gamma}  :=  \llb \Lambda^\Gamma, X \rlb_R, 
\]
where $\Lambda^\Gamma$ denotes the cyclic membrane corresponding to $\Gamma$. 
%
%
Similarly, given a marked surface $(S,M)$ and a triangulation $\T$ of $S$ with set of vertices $M$,
we denote by $RX_\Tc = \llb \Lambda^\Gamma, X \rlb_R$ the corresponding derived membrane space. 

 
\begin{thm}\label{thm-2-Segal-TFT}
Let $X$ be a cyclic $2$-Segal object in $\Cb_{\Lambda}$. Then the functor
\[
RX:\; \Rib \lra \C,\; \Gamma \mapsto RX_\Gamma
\]
maps contractions of ribbon graphs to weak equivalences in $\Cb$.
\end{thm}
\begin{proof} It suffices to show that a contraction $p: \Gamma \to \Gamma'$ of a single edge $e$ of
	$\Gamma$ to a vertex $v$ of $\Gamma'$ induces a weak equivalence. Without restriction we
	assume that $X$ is injectively fibrant. Assume that the edge $e$ is incident to vertices of
	valency $m+1$ and $n+1$, respectively. Then the map of cyclic membranes $\Lambda^\Gamma \to
	\Lambda^{\Gamma'}$ induced by $p$ is a pushout of the map
	\[
	\Lambda^{\{0,n,\dots,n+m\}} \coprod_{\Lambda^{\{0,n\}}} \Lambda^{\{0,1,\dots,n\}} \to
	\Lambda^{\{0,1,\dots,n+m\}}.
	\]
	Evaluating $X$ on this map, we obtain 
	\[
		X_{\{0,1,\dots,n+m\}} \to X_{\{0,n,\dots,n+m\}} \times_{X_{\{0,n\}}} X_{\{0,1,\dots,n\}}
	\]
	which is the $2$-Segal map corresponding to the subdivision of a $(n+m+1)$-gon into a
	$(n+1)$-gon and a $(m+1)$-gon along the edge $\{0,n\}$. Thus it is a weak equivalence, and
	hence, due to the fibrancy assumption, a trivial fibration in $\Cb$. Since the map
	\[
	RX_\Gamma \to RX_{\Gamma'}
	\]
	induced by $p$ is a pullback of the above $2$-Segal map it is a weak equivalence as
	well.
\end{proof}

Let $(S,M)$ be a marked surface and let $\Sigma(S,M)$ be the partially ordered set of isotopy
classes of spanning graphs for $(S,M)$ where, as in \S \ref{subsec:back-triang-surf}, the order is
given by degeneration. The geometric realization $|\Sigma(S,M)|$ of the poset $\Sigma(S,M)$ is
homeomorphic to the tessellation complex $K_{S,M}$ from \S \ref{subsub:rib}, and hence contractible
by Proposition \ref{prop:modular-contractible}. We obtain the following immediate consequences.

\begin{cor}\label{cor:indep-triang1}
Let $X$ be a cyclic $2$-Segal object in $\Cb$ and let $(S,M)$ be a stable marked surface. 
Then the object $RX_\Gamma$, $\Gamma \in \Sigma(S,M)$, is, up to weak equivalence, independent on
the choice of $\Gamma$ and defines therefore a unique isomorphism class of objects in $\on{Ho}(\Cb)$
depending only on $(S,M)$. 
\end{cor}
\begin{proof} This follows from Theorem \ref{thm-2-Segal-TFT}, since $|\Sigma(S,M)|$ is connected.
\end{proof}

\begin{cor}\label{cor:indep-triang2}
Let $X$ be a cyclic $2$-Segal object in $\Cb$ and let $(S,M)$ be a stable marked surface. 
The diagram 
\[
	\Sigma(S,M) \to \on{Ho}(\Cb),\; \Gamma \mapsto RX_{\Gamma}
\]
admits a colimit, denoted by $RX_{(S,M)}$, which is, for every $\Gamma \in \Sigma(S,M)$, equipped with a canonical isomorphism $RX_{\Gamma} \cong
RX_{(S,M)}$ in $\Ho(\Cb)$. The mapping class group $\mathfrak{G}(S,M)$ acts on
$RX_{(S,M)}$ by automorphisms in $\on{Ho}(\Cb)$.
\end{cor}
\begin{proof}
	This follows from Theorem \ref{thm-2-Segal-TFT} and the fact that $|\Sigma(S,M)|$ is
	connected and simply connected.
\end{proof}

\begin{defi}\label{defi:dermem} We call the object $RX_{(S,M)}$ the {\em derived membrane space of the surface $(S,M)$ in
	$X$}.
\end{defi}

\begin{rem} In the above results, we have only used the $1$-connectedness of $|\Sigma(S,M)|$. The
	contractibility of $|\Sigma(S,M)|$ amounts to the statement that there exists a {\em
	coherent} action of the mapping class group. One way to make this precise for closed
	surfaces is to consider the full subcategory $\Rib_3$ of $\Rib$ spanned by
	stable, connected Ribbon graphs with each vertex of valence $\ge 3$. It is well-known (see, e.g.,
	\cite{igusa-reidemeister}), that we have a weak equivalence of topological spaces
	\[
	  |\N(\Rib_3)| \simeq \coprod_{(S,M)} B \mathfrak{G}(S,M)
	\]
	where $(S,M)$ ranges over stable closed marked surfaces.
	Theorem \ref{thm-2-Segal-TFT} implies that the functor
	\[
	  \Rib \lra \Cb, \Gamma \mapsto R X_{\Gamma}
	\]
	maps all morphisms in $\Rib$ to weak equivalences. Passing to nerves, we obtain a map
	\[
	  |\N(\Rib)| \lra |\N(W)|
	\]
	where $W$ denotes the subcategory of weak equivalences in $\Cb$. This map encodes, for each stable
	marked surface $(S,M)$, the choice of an object of $\Cb$ together with a coherent action of
	the mapping class group $\mathfrak{G}(S,M)$.
	A more refined analysis in the context of $\infty$-categories will be given in \cite{HSS2}.
\end{rem}

\section{Application: Fukaya categories}
  
\subsection{Topological Fukaya categories}

We apply the theory of cyclic membrane spaces to the cocyclic $2$-coSegal object $\E = \E^{\bullet}$ from
Theorem \ref{thm:E-cosegal}, considered as a cyclic $2$-Segal object in $(\dgcatt, \Mor)^{\op}$. We
use the upper index notation $L\E^{\Gamma}$ to denote $R(\E^{\op})_{\Gamma}$, and similarly for
other types of derived membrane spaces. In particular, we write $L\E^{(S,M)}$ for
$R(\E^{\op})_{(S,M)}$ from Definition \ref{defi:dermem}.

\begin{defi} Let $(S,M)$ be a stable marked oriented surface. We call the derived membrane object 
\[
	\F^{(S,M)} = L\E^{(S,M)} 
\]
the {\em topological coFukaya category of $(S,M)$}.
Given a $2$-periodic perfect dg-category $\Ac$, we call 
\[
	\IRHom(\F^{(S,M)}, \Ac)
\]
the {\em topological Fukaya category of $(S,M)$ with coefficients in $\Ac$}. We introduce special
notation for the Morita dual
\[
	\F_{(S,M)} = \IRHom(\F^{(S,M)}, \Perft_{\k})
\]
which is simply called the {\em topological Fukaya category of $(S,M)$}.
\end{defi}

As immediate consequences of the general theory of derived cyclic membranes, we obtain the following
main results.

\begin{thm}\label{thm:descent}
	Let $(S,M)$ be a stable marked oriented surface, and let $\Gamma$ be a spanning graph for $(S,M)$. 
	Then we have canonical isomorphisms in $\Hmot$
	\begin{align*}
		\F^{(S,M)} & \simeq L\E^{\Gamma} \simeq \hoind_{\{\Lambda^n \to \Lambda^{\Gamma}\}} \E^n\\
	        \F_{(S,M)} & \simeq R\E_{\Gamma} \simeq \hopro_{\{\Lambda^n \to
			\Lambda^{\Gamma}\}} \E_n,
	\end{align*}
	where the homotopy limits are taken in $(\dgcatt, \Mor)$.
\end{thm}
\begin{proof} Corollary \ref{cor:indep-triang2}.
\end{proof}

Therefore, while the definition of the topological (co)Fukaya category does not depend on any choice
of a triangulation of $(S,M)$, we may chose a triangulation to compute it via the descent
isomorphisms of Theorem \ref{thm:descent}.

\begin{thm}\label{thm:mapping}
	Let $(S,M)$ be a stable marked oriented surface. The topological (co)Fukaya categories
	$\F^{(S,M)}$ and $\F_{(S,M)}$ admit a canonical action of the mapping class group of
	$(S,M)$ via automorphisms in $\Hmot$.
\end{thm}
\begin{proof} Corollary \ref{cor:indep-triang2}.
\end{proof}

The name ``topological Fukaya category" for $\F_{(S,M)}$ is justified as follows.
First, if $(S,M) = (\sigma, \on{Vert}(\sigma))$
 is a triangle, then elements of the cyclic set $M^*$ are identified with
edges of the triangle $\sigma$. For two such edges $i,j\in M^*$ the indecomposable object
$E_{ij}\in \E_M = \E^{M^*}$ is then visualized by an oriented simple arc $\alpha$
beginning at an interior point of the edge $i$ and ending at an interior point of the edge $j$
(such arcs form one isotopy class). Let us denote this object $E_\alpha$. 
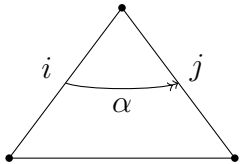
\[
\begin{tikzpicture}
\usetikzlibrary{calc}
\def\centerarc[#1](#2)(#3:#4:#5)  { \draw[#1] ($(#2)+({#5*cos(#3)},{#5*sin(#3)})$) arc (#3:#4:#5); }

    \node (left) at (-1.5,0) {};
 \fill (left) circle (0.05);
    \node (middle) at (0,2) {};
 \fill (middle) circle (0.05);
    \node (right) at (1.5,0) {};
 \fill (right) circle (0.05);

    \draw[postaction={decorate}] (left.center) -- (middle.center);
      \draw[postaction={decorate}] (middle.center) -- (right.center);
       \draw[postaction={decorate}] (left.center) -- (right.center);
       
       \node at (-1,1.2){$i$};
       \node at (1,1.2){$j$};

\draw[->>] (-0.75,1) .. controls (-0.5,0.9) and (0.5,0.9) .. (0.75,1);

\node at (0, 0.7){$\alpha$};

\end{tikzpicture}
\]
Next, let $\Tc$ be a triangulation of a stable marked surface $(S,M)$, and
$\Gamma$ be the corresponding spanning graph. The homotopy limit
defining $R\E_\Tc=R\E_\Gamma$  can be computed by reducing it
to a homotopy fiber product which is then computed by using 
the concept of the path object $P(\Ac) \to \Ac \times \Ac$ of a dg
category $\Ac$ from \cite{tabuada}, see the proof of Proposition 
\ref{prop:universal1}. Explicitly, this means that local arcs in individual
triangles as above can be combined, in the homotopy limit, yielding a large 
supply of aggregate objects: 
\begin{enumerate}
	\item {\em open Lagrangians:} isotopy classes of oriented immersed arcs $\beta$
	 which begin and end on $\partial S$ and avoid $M$. In the real blowup
	 picture they correspond to arcs beginning and ending on open ends of the 
	 blown up surface. 
	  
  	\item {\em closed Lagrangians:} isotopy classes of oriented closed immersed curves
  	$\beta$ avoiding $\partial S \cup M$ and equipped with a flat $\k^*$-principal bundle 
	(completely classified by its monodromy).
      	In the real blowup picture $\beta$ is a closed oriented curve inside the blown up surface. 
\end{enumerate}
More precisely, the object $E_\beta$ corresponding to an arc or curve $\beta$ as above
is obtained by gluing together the objects $E_\alpha\in \E_\sigma$ for $\sigma$
being a triangle of $\Tc$ and $\alpha$ being a component of $\beta\cap\sigma$. 
The monodromy for closed curves appears because of the $\k^*$-freedom in identifying
the images of $E_\alpha$ and $E_{\alpha'}$  in $\E_{\on{Vert}(b)}$ for adjacent parts 
$\alpha, \alpha'$ of $\beta$ cut by adjacent triangles with common edge $b$.

\subsection{Examples}

\label{subsection:fukayaexamples}

As an illustration, we show how our construction relates to some examples from Kontsevich's list
\cite[Pictures]{kontsevich-sym-homol}. 
In each case we first exhibit the surface Postnikov system generating the
  coFukaya category $\F^{(S,M)}$, then refine this to a homotopy colimit presentation
  via  Theorem  \ref{thm:descent}, and, finally,  explain how to
  identify $\Fc^{(S,M)}$ or  $\Fc_{(S,M)}$ with an algebro-geometric derived category. 
 We assume the ground field $\k$ to be algebraically closed.

\subsubsection{The affine line}

Let $(S,M)$ be a disk with one interior and one boundary marked point. 
The Ribbon graph $\Gamma$ displayed in Figure \ref{fig:annulus1} is a spanning graph in $(S,M)$ whose corresponding
real blowup is given by an annulus with an open end on one of its boundary components.

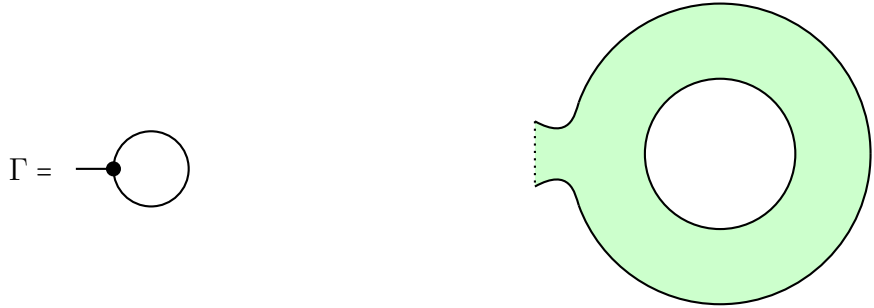
\begin{figure}[ht]
\begin{minipage}[b]{0.5\linewidth}
\centering
\begin{tikzpicture}[baseline=-10ex,>=latex,scale=0.5]
\begin{scope}

  	\draw (180:1.5cm) node[fill=white,left] {$\Gamma =\quad $};
	\draw[thick] (0:1cm) arc (0:360:1cm);
	\draw[thick] (180:1cm) -- (180:2cm);
	\fill (180:1cm) circle [radius=0.2cm];

\end{scope}
\end{tikzpicture}
\end{minipage}
\begin{minipage}[b]{0.5\linewidth}
\centering
\begin{tikzpicture}[>=latex,scale=1.0]
\begin{scope}[decoration={
    markings,
    mark=at position 0.5 with {\arrow{>}}}
    ] 

\fill[fill=green!20!white] (0:1cm) -- (0:2cm)
	arc (0:360:2cm) -- (0:1cm)
	arc (360:0:1cm) -- cycle;

\fill[fill=green!20!white] (160:2cm) .. controls (165:2cm) and (175:2cm) .. (170:2.5cm)
		-- (-170:2.5cm) .. controls (-175:2cm) and (-165:2cm) .. (-160:2cm)
		-- cycle;

	\draw[thick] (0:1cm) arc (0:360:1cm); 
	\draw[thick] (0:2cm) arc (0:160:2cm);
	\draw[thick] (0:2cm) arc (0:-160:2cm);
	\draw[thick, dotted] (170:2.5cm) -- (-170:2.5cm);
	\draw[thick] (160:2cm) .. controls (165:2cm) and (175:2cm) .. (170:2.5cm)
		(-170:2.5cm) .. controls (-175:2cm) and (-165:2cm) .. (-160:2cm);

\end{scope}
\end{tikzpicture}
\end{minipage}
\caption{Ribbon graph $\Gamma$ with corresponding real blowup}\label{fig:annulus1}
\end{figure}

\noindent {\bf  Surface Postnikov system:}
\be\label{eq:SPS-A1}
A\buildrel \alpha\over\lra A\lra C\lra \Sigma A
\ee
(One distinguished triangle with two terms being the same.) 

\vskip .2cm

\noindent {\bf Homotopy colimit presentation:}
\[
	\F^{(S,M)} \simeq \E^2 \coprod^{\h}_{\E^1 \coprod \E^1} \E^1 \simeq \A^2 \coprod^{\h}_{\A^1 \coprod \A^1} \A^1
 \simeq \A^2 \coprod_{\A^1 \coprod \A^1} \A^1  =:  \EuScript L	
	\]
 where we use the Morita equivalences $\A^n \to \E^n$ to simplify the homotopy pushout
 and then identify it with the usual pushout using the
  fact that 
   the map
$\A^1 \coprod \A^1 \to \A^2$ is a Morita cofibration.
The usual push out, denoted $\EuScript L$, is the 
 $\Zt$-graded $\k$-linear category freely generated by the quiver with one
vertex and one loop. It  corresponds to the part $\{A\buildrel\alpha\over\to A\}$ in 
\eqref{eq:SPS-A1}. 

\vskip .2cm

\noindent {\bf Algebro-geometric picture:} 
  Passing to the perfect envelope we obtain
\[
  \F^{(S,M)} \simeq \Perft_{\EuScript L} \simeq \Perft(\AA^1_\k) , 
\]
the $\Zt$-folded category of perfect complexes on the affine line.
Dually, we obtain the category
\[
	\F_{(S,M)} \simeq \IRHom(\Perft_{\EuScript L},\Perft_{\k}) \simeq \IRHom(\EuScript L, \Perft_{\k})
\]
which can be identified with the full subcategory of $\Perf^{(2)}(\AA^1_\k)$ consisting of those
complexes whose cohomology is compactly supported.

The indecomposable objects of $\F_{(S,M)}$ are given by shifts of skyscraper sheaves on $\AA^1$.
The corresponding Lagrangians can be explicitly visualized in the annulus. The skyscraper sheaf of
length $n$ with support at the origin in $\AA^1$ corresponds to the curve which starts at the open
end wraps $n$ times around the annulus and ends at the open end. A skyscraper sheaf of length $n$
with support at a point $\lambda \in \AA^1$ with $\lambda \ne 0$, corresponds to a closed curve
which wraps around the annulus $n$-times and is equipped with the flat $\k^*$-principal bundle
with monodromy $\lambda$. Shifting an indecomposable object amounts to changing the orientation
of the corresponding object.

Dually, we can visualize the generator of the category $\F^{(S,M)}$ corresponding to the
vertex of the quiver $\EuScript L$. It corresponds to an arc which connects both boundary components
of the annulus. 

\subsubsection{The projective line}
\label{subsub:projline}

Let $(S,M)$ be an annulus with one marked point on each boundary component.
Figure \ref{fig:annulus2} depicts two spanning Ribbon graphs and the real blow up.
\begin{figure}[ht]
\begin{minipage}[b]{0.5\linewidth}
\centering
\begin{tikzpicture}[baseline=-10ex,>=latex,scale=0.5]
\begin{scope}
	\draw[thick] (0:1cm) arc (0:360:1cm);
	\draw[thick] (180:1cm) -- (180:2cm);
	\draw[thick] (180:1cm) -- (180:0cm);
	\fill (180:1cm) circle [radius=0.2cm];
\end{scope}
\begin{scope}[shift={(6,0)}]

  	\draw (180:3cm) node[fill=white,left] {$\simeq$};
	\draw[thick] (-150:1.5cm) arc (-150:150:1.5cm);
	\draw[thick] (-150:1.5cm) -- (150:1.5cm);
	\draw[thick] (-150:1.5cm) -- (-100:0.5cm);
	\draw[thick] (150:1.5cm) -- (155:2.7cm);
	\fill (-150:1.5cm) circle [radius=0.2cm];
	\fill (150:1.5cm) circle [radius=0.2cm];

\end{scope}
\end{tikzpicture}
\end{minipage}
\begin{minipage}[b]{0.5\linewidth}
\centering
\begin{tikzpicture}[>=latex,scale=1.0]
\begin{scope}[decoration={
    markings,
    mark=at position 0.5 with {\arrow{>}}}
    ] 

\fill[fill=green!20!white] (0:1cm) -- (0:2cm)
	arc (0:360:2cm) -- (0:1cm)
	arc (360:0:1cm) -- cycle;

\fill[fill=green!20!white] (160:2cm) .. controls (165:2cm) and (175:2cm) .. (170:2.5cm)
		-- (-170:2.5cm) .. controls (-175:2cm) and (-165:2cm) .. (-160:2cm)
		-- cycle;

\fill[fill=green!20!white] (150:1cm) .. controls (160:1cm) and (160:0.8cm) .. (130:0.6cm)
		     -- (-130:0.6cm) .. controls (-160:0.8cm) and (-160:1cm) .. (-150:1cm) -- (180:2cm)
		     -- cycle;

	\draw[thick] (0:1cm) arc (0:150:1cm);
	\draw[thick] (0:1cm) arc (0:-150:1cm);
	\draw[thick] (0:2cm) arc (0:160:2cm);
	\draw[thick] (0:2cm) arc (0:-160:2cm);
	\draw[thick, dotted] (170:2.5cm) -- (-170:2.5cm);
  	\draw (180:2.4cm) node[left] {$P$};
	\draw[thick, dotted] (130:0.6cm) -- (-130:0.6cm);
  	\draw (180:0.45cm) node[right] {$Q$};
	\draw[thick] (160:2cm) .. controls (165:2cm) and (175:2cm) .. (170:2.5cm)
		(-170:2.5cm) .. controls (-175:2cm) and (-165:2cm) .. (-160:2cm);
	\draw[thick] (150:1cm) .. controls (160:1cm) and (160:0.8cm) .. (130:0.6cm)
		     (-130:0.6cm) .. controls (-160:0.8cm) and (-160:1cm) .. (-150:1cm);

\end{scope}
\end{tikzpicture}
\end{minipage}
\caption{Spanning Ribbon graphs and real blowup}\label{fig:annulus2}
\end{figure}
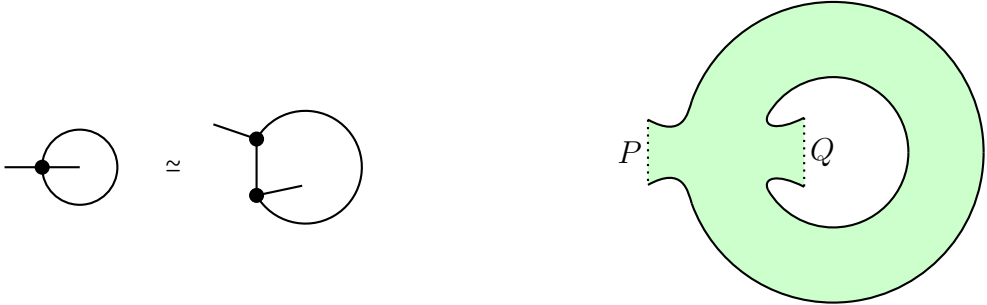

\noindent {\bf Surface Postnikov system:} 
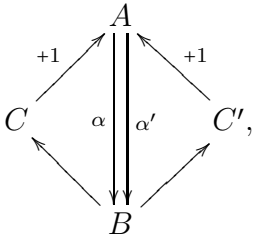
\be\label{eq:SPS-P1}
\xymatrix{
&A\ar@<.5ex>[dd]^{\alpha'}\ar@<-.5ex>[dd]_{\alpha}& \\
C\ar[ur]^{+1}&&C', \ar[ul]_{+1} \\
&B\ar[ul]  \ar[ur]&
}
\ee
(Two distinguished triangles with two vertices in common.)

\vskip .2cm

\noindent {\bf Homotopy colimit presentation:} 
 \[
	\F^{(S,M)} \simeq (\E^2 \coprod^{\h}_{\E^1} \E^2) \coprod^{\h}_{\E^1 \coprod \E^1} \E^1 
              \simeq (\A^2 \coprod_{\A^1} \A^2) \coprod_{\A^1 \coprod \A^1} \A^1 =:\K.
\]
where we use a similar argument to simplify the homotopy colimit. The
   $\Zt$-graded $\k$-linear category $\K$ is  freely generated by the Kronecker
quiver with two vertices and two parallel arrows. It corresponds to the part
$\bigl\{\xymatrix{
A\ar@<.5ex>[r]^\alpha \ar@<-.5ex>[r]_{\alpha'}&B
}\bigr\}$ in \eqref{eq:SPS-P1}. 

\vskip .2cm

\noindent {\bf Algebro-geometric picture:} Using the classical equivalence $b$  of Beilinson \cite{beilinson},
we identify $\Fc^{(S,M)}$
 \be\label{eq:beilinson}
 \begin{gathered}
	\F^{(S,M)} \simeq \Perft_{\K} =   D^{(2)}(\K) \buildrel b\over\lra   \Perft(\PP_\k^1)=  D^{(2)} (\PP^1_\k),\\
\xymatrix{
\{V^\bullet_A \ar@<.5ex>[r]^{\alpha} \ar@<-.5ex>[r]_{\alpha'} &V^\bullet_B\}
}
\,\, \buildrel b\over\longmapsto \,\,\, \P^\bullet := \on{Cone}\bigl\{ V^\bullet_A\otimes \Oc_{\PP^1}(-1) 
\buildrel \alpha t_0 + \alpha' t_1 \over \lra 
V^\bullet_B\otimes \Oc_{\PP^1}\bigr\} 
 \end{gathered}
\ee
with the $\Zt$-folding of the dg-category of perfect complexes on the projective line $\PP^1_\k= \on{Proj} \k[t_0, t_1]$. Since
this category is smooth and proper, and hence dualizable with respect to Morita duality, we obtain that
\[
	\F_{(S,M)} \simeq \IRHom(\Perft_{\K},\Perft_{\k}) \simeq \Perft_{\K^{\op}} \simeq \Perft_{\K} \simeq \F^{(S,M)}. 
\]
Again, we exhibit all indecomposable objects of $\Perft(\PP_\k^1)$ explicitly as objects
of $\F_{(S,M)}$, visualized as immersed Lagrangians in the annulus. The line bundle $\O(n)$, $n \in
\ZZ$, corresponds to an oriented arc starting at the open end $P$, wrapping $n$ times
counterclockwise around the annulus and ending at the open end $Q$. The skyscraper sheaf of length
$n$ supported at a finite nonzero point $\lambda$ in $\PP^1_k$ corresponds to a closed oriented
curve wrapping around the annulus $n$ times, equipped with the flat $\k^*$-principal bundle
corresponding to $\lambda$. Skyscraper sheaves of length $n$ supported at $0$ respectively $\infty$ correspond to oriented
curves beginning and ending at $P$ respectively $Q$ wrapping around the annulus $n$ times.

Dually, we can visualize the generators of $\F^{(S,M)}$ corresponding to the vertices of the
Kronecker quiver. 

\subsubsection{The nodal cubic curve}

Let $(S,M)$ be the  torus with one marked point. We use 
the spanning Ribbon graph $\Gamma$ displayed in Figure \ref{fig:torus}. 
 
\vskip .2cm
 
\noindent {\bf Surface Postnikov system:}
 \begin{equation}
   \xymatrix{
     A\ar@<.5ex>[r]^{\alpha} \ar@<-.5ex>[r]_{\alpha'} & B
     \ar@<.5ex>[r]^{\beta} \ar@<-.5ex>[r]_{\beta'} &C 
     \ar@<.5ex>[r]^{\gamma} \ar@<-.5ex>[r]_{\gamma'}&\Sigma A, 
   }
\end{equation}
(Two distinguished triangles on the same 3 objects, having  arrows in the same direction.)
It differs from \eqref{eq:SPS-P1} in that $C\simeq \on{Cone}(\alpha)$ is identified with $C'\simeq \on{Cone}(\alpha')$. 
 
\vskip .2cm
 
\noindent {\bf Homotopy colimit presentation:} 
\begin{equation}
    \F_{(S,M)} \simeq \Perft(\PP_\k^1) \times^{\h}_{\Perft_{\k} \times \Perft_{\k}}
    \Perft_{\k}.
\end{equation}

\noindent {\bf Algebro-geometric picture:} In the Beilinson equivalence \eqref{eq:beilinson},
the dg-vector spaces gives by the 
cones of $\alpha$ and $\alpha'$ are  the fibers of the perfect complex $\P^\bullet$ at $0$ and $\infty$.
So an identification of these  cones is a datum of descent to the nodal cubic curve $C= \PP^1_\k/(0\sim\infty)$. 
This can be extended to a Morita equivalence
\[
	\F_{(S,M)} \simeq \Perft(C),
\] 
see \cite{sibilla}. Theorem \ref{thm:mapping} implies the existence of an action of $SL(2,\ZZ)$ on $\Perft(C)$.
\begin{figure}[ht]
\begin{minipage}[b]{0.5\linewidth}
\centering
\begin{tikzpicture}[baseline=-10ex,>=latex,scale=0.5]
\begin{scope}

  	\draw (180:1.5cm) node[fill=white,left] {$\Gamma =\quad $};
	\draw[thick] (0:1cm) arc (0:360:1cm);
	\fill (180:1cm) circle [radius=0.2cm];
	\draw[thick] (180:1cm) arc (-90:-10:1cm);
	\draw[thick] (180:1cm) arc (-90:-350:1cm);

\end{scope}
\end{tikzpicture}
\end{minipage}
\begin{minipage}[b]{0.5\linewidth}
\centering
\begin{tikzpicture}[>=latex,scale=1.0]
\begin{scope}[shift={(-1.1,1.2)},decoration={
    markings,
    mark=at position 0.5 with {\arrow{>}}}
    ] 
\fill[fill=green!20!white] (0:1cm) -- (0:1.5cm)
	arc (0:360:1.5cm) -- (0:1cm)
	arc (360:0:1cm) -- cycle;

	\draw[thick] (-70:1cm) arc (-70:230:1cm);
	\draw[thick] (-80:1.5cm) arc (-80:240:1.5cm);

\end{scope}
\begin{scope}[decoration={
    markings,
    mark=at position 0.5 with {\arrow{>}}}
    ] 

\fill[fill=green!20!white] (0:1cm) -- (0:1.5cm)
	arc (0:360:1.5cm) -- (0:1cm)
	arc (360:0:1cm) -- cycle;

\fill[fill=green!20!white] (150:1cm) .. controls (160:1cm) and (170:1cm) .. ([shift={(-1.1,1.2)}] -70:1cm)
	-- ([shift={(-1.1,1.2)}] -80:1.5cm) .. controls (200:1cm) and (200:1cm) .. (-150:1cm)
	-- (-160:1.5cm) .. controls (190:1.5cm) and (190:1.5cm) .. ([shift={(-1.1,1.2)}] 240:1.5cm)
	-- ([shift={(-1.1,1.2)}] 230:1cm) .. controls (170:1.5cm) and (170:1.5cm) .. (160:1.5cm) --
	cycle;

	\draw[thick] (0:1cm) arc (0:150:1cm);
	\draw[thick] (0:1cm) arc (0:-150:1cm);
	\draw[thick] (0:1.5cm) arc (0:160:1.5cm);
	\draw[thick] (0:1.5cm) arc (0:-160:1.5cm);

	\draw[thick] (150:1cm) .. controls (160:1cm) and (170:1cm) .. ([shift={(-1.1,1.2)}] -70:1cm);
	\draw[thick] ([shift={(-1.1,1.2)}] -80:1.5cm) .. controls (200:1cm) and (200:1cm) .. (-150:1cm);
	\draw[thick] (-160:1.5cm) .. controls (190:1.5cm) and (190:1.5cm) .. ([shift={(-1.1,1.2)}]
	240:1.5cm);
	\draw[thick] ([shift={(-1.1,1.2)}] 230:1cm) .. controls (170:1.5cm) and (170:1.5cm) ..
	(160:1.5cm);

\end{scope}
\end{tikzpicture}
\end{minipage}
\caption{Ribbon graph $\Gamma$ and real blowup}\label{fig:torus}
\end{figure}
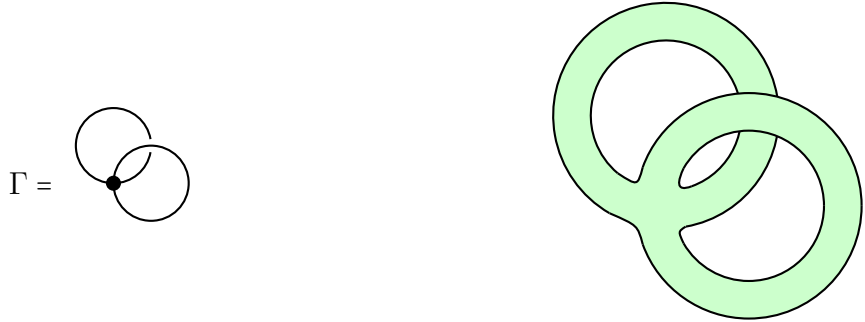
Using the methods of \cite{sibilla}, the topological Fukaya categories associated to the Ribbon
graph
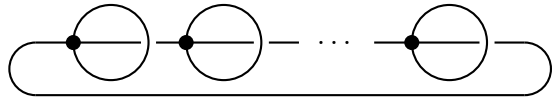
\[
\begin{tikzpicture}[>=latex,scale=0.5]
\begin{scope}
	\draw[thick] (0,0) circle (1);
	\fill (-1,0) circle (0.2);
	\draw[thick] (3,0) circle (1);
	\fill (2,0) circle (0.2);
	\draw (6,0) node {$\dots$};
	\draw[thick] (9,0) circle (1);
	\fill (8,0) circle (0.2);
	\draw[thick] (-2,0) -- (-1,0);
	\draw[thick] (-1,0) -- (0.8,0);
	\draw[thick] (1.2,0) -- (2,0);
	\draw[thick] (2.0,0) -- (3.8,0);
	\draw[thick] (4.2,0) -- (5,0);
	\draw[thick] (7,0) -- (8,0);
	\draw[thick] (8,0) -- (9.8,0);
	\draw[thick] (10.2,0) -- (11,0);
	\draw[thick] (11,0) arc (90:-90:0.7);
	\draw[thick] (11,-1.4) -- (-2,-1.4);
	\draw[thick] (-2,-1.4) arc (270:90:0.7);

\end{scope}
\end{tikzpicture}
\]
with $n$ vertices can be identified with the $\Zt$-folding of the category of perfect complexes on a cycle
of $n$ projective lines which, by Theorem \ref{thm:mapping}, comes equipped with an action of the
mapping class group of the $n$-punctured torus. Note that, in \cite{burban-kreussler,sibilla},
actions of central extensions of these mapping class groups have been constructed on the
$\ZZ$-graded variants of the above categories. When passing to $2$-periodizations of the
respective categories, these actions factor through the mapping class group actions which we
construct.

\subsubsection{The union of two lines}

Let $(S,M)$ be a   sphere with 3 marked points. 
Two (equivalent) spanning Ribbon graphs $\Gamma, \Gamma'$ of $(S,M)$ 
are displayed in Figure \ref{fig:sphere3} with the
corresponding real blowup given by a disk with two open interior disks removed.

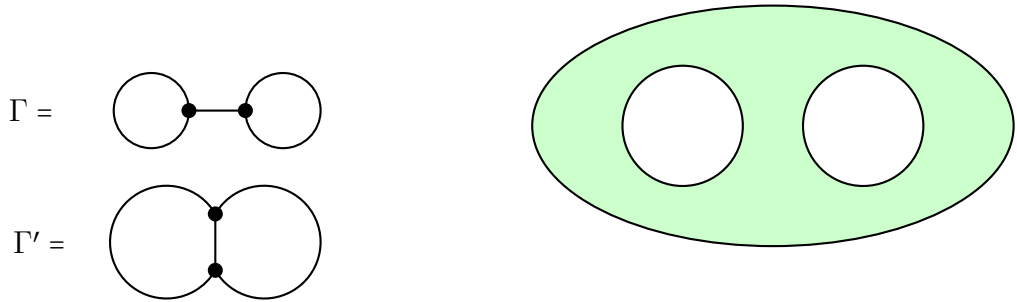
\begin{figure}[ht]
\begin{minipage}[b]{0.5\linewidth}
\centering
\begin{tikzpicture}[baseline=-10ex,>=latex,scale=0.5]
\begin{scope}

  	\draw (180:1.5cm) node[fill=white,left] {$\Gamma =\quad $};
	\draw[thick] (180:0cm) circle [radius=1cm];
	\fill (0:1cm) circle [radius=0.2cm];
	\draw[thick] (0:1cm) -- (0:2.5cm);
	\fill (0:2.5cm) circle [radius=0.2cm];
	\draw[thick] (0:3.5cm) circle [radius=1cm];

\end{scope}

\begin{scope}[shift={(3, -3.5)}]

  	\draw (180:5cm) node[fill=white,left] {$\Gamma' = $};
	\draw[thick] (-150:1.5cm) arc (-150:150:1.5cm);
	\draw[thick] (150:1.5cm) arc (30:330:1.5cm);
	\draw[thick] (-150:1.5cm) -- (150:1.5cm);
	\fill (-150:1.5cm) circle [radius=0.2cm];
	\fill (150:1.5cm) circle [radius=0.2cm];

\end{scope}

\end{tikzpicture}
\end{minipage}
\begin{minipage}[b]{0.5\linewidth}
\centering
\begin{tikzpicture}[>=latex,scale=0.8]
\begin{scope}
\fill[fill=green!20!white] (0:0) ellipse (4 and 2);
\fill[fill=white] (180:1.5) circle (1);
\fill[fill=white] (0:1.5) circle (1);

\draw[thick] (0:0) ellipse (4 and 2);
\draw[thick] (180:1.5) circle (1);
\draw[thick] (0:1.5) circle (1);

\end{scope}
\end{tikzpicture}
\end{minipage}
\caption{Ribbon graphs $\Gamma, \Gamma'$ with corresponding real blowup}\label{fig:sphere3}
\end{figure}

Using the   Ribbon graph $\Gamma'$ we obtain the following. 

\vskip .2cm
 
\noindent {\bf Surface Postnikov system:}
\be\label{eq:SPS-node}
\xymatrix{ A \ar@<.5ex>[r]^{\alpha}&B \ar@<.5ex>[l]^{\alpha'}
\ar@<.5ex>[r]^{\beta}&C \ar@<.5ex> ^{\beta'}[l]
\ar@<.5ex>[r]^{\gamma}&\Sigma A 
\ar@<.5ex>[l]^{\gamma'},
 }
\ee
(Two distinguished triangles on the same 3 objects, with arrows going in the opposite directions.)

\vskip .2cm

\noindent {\bf Homotopy colimit presentation:}
 \begin{equation}\label{eq:cofukpushout}
	\F^{(S,M)} \simeq \E^2 \coprod^{\h}_{\E^1 \coprod \E^1 \coprod \E^1} \E^2.
\end{equation}
\noindent {\bf Algebro-geometric picture:} In \eqref{eq:SPS-node}, the endomorphisms
\[
y_B = \alpha \alpha', \,\, x_B=\beta'\beta \,\,\in\,\, \End(B)
\]
satisfy $x_By_B=y_Bx_B=0$ and so make $B$ into a module (in $\Fc^{(S,M)}$) over the algebra
$R = \k[x,y]/(xy)$. Further, defining
\[
y_A = \alpha'\alpha, x_A = (\Sigma\gamma)(\Sigma\gamma') \in\End(A), \quad
y_C=\gamma'\gamma,  x_C= \beta\beta' \in\End(C),
\]
we make $B$ and $C$ into $R$-modules as well, so that the arrows in \eqref{eq:SPS-node}
commute with the $R$-action. Using  \eqref{eq:SPS-node} as a tilting object
and analyzing more carefully the morphisms of complexes, 
we construct a Morita equivalence 
$\Fc^{(S,M)} \simeq D^{(2)}_{\text{fg}}(R)$ with the 2-periodification of the bounded derived category of
finitely generated $R$-modules. 
On the level of objects, the equivalence takes $A \mapsto R/(x)$, $B \mapsto R$, and $C \mapsto R/(y)$, and the
surface Postnikov system \eqref{eq:SPS-node} in $\F^{(S,M)}$ corresponds to the system  in $D(R)$ given by
\begin{equation}
	  \label{eq:SPS-xy}
	  \xymatrix{ R/(x) \ar@<.5ex>[r]^{y}&R \ar@<.5ex>[l]^{\on{pr}}
	  \ar@<.5ex>[r]^{\on{pr}}&R/(y) \ar@<.5ex> ^{x}[l]
	  \ar@<.5ex>[r]^{}&\Sigma R/(x).
	  \ar@<.5ex>[l]^{}
}
\end{equation}
The mapping class group of a sphere with 3 marked points is the symmetric group
$S_3$ (cf. \cite{farb-margalit}). Therefore, $S_3$   acts on  $D^{(2)}_{\text{fg}}(R)$
by equivalences of triangulated categories. 
The action on objects is given by
\begin{align*}
	(12):& \quad (R/(x), R, R/(y)) \mapsto (\Sigma R/(x), R/(y), R)\\
	(23):& \quad (R/(x), R, R/(y)) \mapsto (R, R/(x), \Sigma R/(y)).
\end{align*}
For example, the cycle $(123)$ induces a two-step rotation of the above distinguished triangles.
Note that all three terms of the surface Postnikov system \eqref{eq:SPS-xy} have the same endomorphism
ring 
\[
	\Hom_R(R,R) \cong \Ext^\bullet_R(R/(x), R/(x)) \cong \Ext^\bullet_R (R/(y), R/(y)) \cong R,
\]
in $D^{(2)}(R)$ due to the $2$-periodic folding.

\section{Application: Waldhausen S-construction}\label{sec:sconstruction}

In \cite{HSS1}, we showed that the Waldhausen S-construction of a stable $\infty$-category is a
$2$-Segal space. Generalizing results of \cite{lurie.algebra}, it is shown in \cite{faonte} that the
differential graded nerve $\Ndg(\A)$ of a perfect dg category $\A$ is a stable $\infty$-category.

In this section, we show that, given a $2$-periodic perfect dg category $\A$, the Waldhausen
S-construction of the stable $\infty$-category $\Ndg(\A)$ is weakly equivalent to the simplicial
space $\Map(\E^{\bullet},\A)$.
An immediate consequence is the following result, predicted on a heuristic basis in \cite{HSS1}.

\begin{thm} The Waldhausen S-construction of a $2$-periodic perfect dg category admits a canonical
	cyclic structure.
\end{thm}

We recall the variant of the Waldhausen S-construction given in \cite{HSS1}. For $n \ge 0$, let
$\J^n$ be the nerve of the category $\Fun([1], [n])$ corresponding to the poset formed by ordered
pairs $(0\leq i\leq j\leq n)$, with $(i,j)\leq (k,l)$ iff $i\leq k$ and $j\leq l$. 

\begin{defi} \label{defi:exactinftywald} Let $\C$ be a stable $\infty$-category.
	We define 
	\[
	\SW_n \C \subset \Fun(\J^n, \C)_{\Kan}
	\]
 	to be the simplicial subset given by those simplices whose vertices are $\J^n$-diagrams $F$ satisfying the following conditions:
	\begin{enumerate}
	\item For all $0 \le i \le n$, the object $F(i,i)$ is a zero object in $\C$.
	\item For any $0 \le j \le k \le n$, the square 
	 \[
	 \xymatrix{
	 F(0,j) \ar[r] \ar[d] & F(0,k) \ar[d]\\
	 F(j,j) \ar[r] & F(j,k)
	 }
	 \] 
	in $\C$ is coCartesian.
\end{enumerate}
By construction, $\SW_n \C$ is functorial in $[n]$ and defines a simplicial
space $\SW \C$, which we call the {\em Waldhausen S-construction} or {\em Waldhausen space} of $\C$.
\end{defi}

To bridge between dg categories and $\infty$-categories, we use the Quillen adjunction 
\begin{equation}\label{eq:dg-infty}
	\dg: (\sSet,\on{Joyal}) \longleftrightarrow (\dgcat,\Qeq) : \Ndg
\end{equation}
as introduced in \cite{lurie.algebra}. Here, $\dgcat$ is equipped with the quasi-equivalence model
structure of Tabuada and $\sSet$ is equipped with the quasi-category model structure of Joyal.
Further, we have the Quillen adjunction
\begin{equation}\label{eq:dg-2periodic}
	P: (\dgcat, \Qeq) \longleftrightarrow (\dgcatt,\Qeq): F
\end{equation}
from \eqref{eq:adjunction:periodic}, where $P$ is given by folding the mapping complexes $2$-periodically and $F$ is the 
functor which forgets $2$-periodicity. In what follows, we leave the application of $F$ implicit. 

There is a canonical dg functor $U^n:\dg(\J^n) \to \E^n$ extending the assignment
\[
	(i,j) \mapsto E_{ij}.
\]
which, under \eqref{eq:dg-infty}, is adjoint to the functor $\J^n \to \Ndg(\E^n)$ which is nontrivial
only on the $1$-skeleton of $\J^n$ and sends the edge $(i,j) \le (k,l)$ of $\J^n$ to the edge of
$\Ndg(\E^n)$ given by the closed morphism of degree $0$
\[
	\begin{pmatrix}z^{k-i} & 0\\ 0 & z^{l-j} \end{pmatrix} : E_{ij} {\lra} E_{kl}.
\]
By \cite[17.4.15]{hirschhorn}, the above Quillen adjunctions induce weak equivalences of mapping spaces
\[
	\Map_{(\dgcatt, \Qeq)}(P(\dg(\J^n)), \A) \simeq \Map_{(\dgcat, \Qeq)}(\dg(\J^n), \A) \simeq
	\Map_{(\sSet,\on{Joyal})}(\J^n, \Ndg(\A)).
\]
Further, we have the formula
\[
	\Map_{(\sSet,\on{Joyal})}(\J^n, \Ndg(\A)) \simeq \Fun(\J^n, \Ndg(\A))_{\Kan}
\]
for the mapping spaces with respect to the Joyal model structure on $\sSet$.
Hence, pullback along the functor $U^n$ gives a natural map of simplicial sets
\[
	(U^n)^*: \Map^{(2)}(\E^n, \A) \lra \Fun(\J^n, \Ndg(\A))_{\Kan}.
\]
The following proposition shows that the functor $U^n$ is the {\em universal Waldhausen
$n$-simplex}.

\begin{prop} For every $n\ge 0$, the pullback map $(U^n)^*$ factors into
	\[
		\xymatrix{
			\Map^{(2)}(\E^n, \A) \ar[dr]_{f_n} \ar[rr]^{(U^n)^*} & & \Fun(\J^n, \Ndg(\A))_{\Kan}\\
			& \SW_n \Ndg(\A) \ar@{^{(}->}[ur] & 
		}
	\]
	where the map $f_n$ is a weak equivalence. In other words, the functor $(U^n)^*$ is
	a weak equivalence onto the union of those connected components $\Fun(\J^n,
	\Ndg(\A))_{\Kan}$ which satisfy conditions (1) and (2) of Definition
	\ref{defi:exactinftywald}. Further, the resulting weak equivalences $\{f_n\}$ assemble to provide
	a weak equivalence of simplicial spaces
	\[
		\Map^{(2)}(\E^{\bullet}, \A) \overset{\simeq}{\lra} \SW_{\bullet} \Ndg(\A).
	\]
\end{prop}
\begin{proof}
We utilize the $\infty$-categorical theory of Kan extensions as developed in \cite{lurie.htt}.
Consider the functor $i: \Delta^{n-1} \to \J^n$ given by the apparent extension of the association
\[
	k \mapsto [0,k]
\]
on objects. 
Consider the commutative diagram of simplicial sets
\begin{equation}\label{eq:bigdiag}
	\xymatrix{
		&&&\\
		\SW_n \Ndg(\A) \ar@{^{(}->}[r] \ar[dr]_{q} & 
		\Fun(\J^n, \Ndg(\A))_{\Kan} \ar[d]_{i^*} & \ar[l]_-{\simeq}\Map(\dg(\J^n), \A)\ar[d] &
		\ar[l]\ar[ld]^{p} \Map^{(2)}(\E^n, \A) \ar@{-->}@/_10ex/[lll] \ar@/_4ex/[ll]_-{(U^n)^*}\\
		& \Fun(\Delta^{n-1}, \Ndg(\A))_{\Kan} &\ar[l]_-{\simeq} \Map(\dg(\Delta^{n-1}), \A) &.
	}
\end{equation}
The map $q$ is a weak equivalence by Proposition 7.3.6(1) in \cite{HSS1}. The map $p$ is a
weak equivalence, since it is given by pulling back along the composite
\[
	\dg(\Delta^{n-1}) \overset{p_1}{\lra} \A^{n} \overset{r_{n}^\triangleleft}{\lra} \E^n
\]
where $p_1$ is a quasi-equivalence and $r_{n}^\triangleleft$ is the Morita equivalence from Proposition
\ref{prop:r-n-morita}. We claim that the map $(U^n)^*$ factors as indicated by the dashed arrow in
\eqref{eq:bigdiag}. Assuming the existence of this factorization, we obtain a commutative diagram
\begin{equation}
	\xymatrix{
		\SW_n \Ndg(\A)\ar[d]_{\simeq} & \ar[l]_{f_n} \ar[dl]^{\simeq} \Map^{(2)}(\E^n, \A)\\
		 \Fun(\Delta^{n-1}, \Ndg(\A))_{\Kan}&.
	}
\end{equation}
which shows that $f_n$ is a weak equivalence.

To show the claimed factorization, note that, since $\SW_n \Ndg(\A)$ is a union of
connected components of $\Fun(\J^n, \Ndg(\A))_{\Kan}$, it suffices to show that the factorization
exists after passing to $\pi_0$, i.e., morphism sets in the respective homotopy categories. Choosing
a cofibrant replacement $\widetilde{\E^n} \to \E^n$, we express pullback by $U^n$ as the composite
\begin{equation}\label{eq:composite}
	[\E^n, \A] \to [\widetilde{\E^n}, \A] \to [\dg(\J^n), \A] \cong [\J^n, \Ndg(\A)] 
\end{equation}
which allows us to represent every element of $[\widetilde{\E^n}, \A]$ by a dg functor
$\widetilde{\E^n} \to \A$. Using the method of \cite[1.3.2]{lurie.algebra} one verifies that the
squares 
\begin{equation}\label{eq:pushout}
	\xymatrix{
		E_{0j} \ar[r]\ar[d] & E_{0k} \ar[d]\\
		E_{jj} \ar[r] & E_{jk}
	}
\end{equation}
in $\E^n$, become homotopy pushout squares in the model category $(\E^n)^{\op}\mod$. The claim now follows from
the general comparison result \cite[4.2.4.1]{lurie.htt} between homotopy limits in simplicial categories and
$\infty$-categorical limits, again using the methods of \cite[1.3.2]{lurie.algebra}: the homotopy
pushout square \eqref{eq:pushout} becomes an $\infty$-categorical pushout square of
$\Ndg( \Perf(\E^n))$. This implies that its lift in $\Ndg( \Perf(\widetilde{\E^n}))$ is a pushout square
which maps to a pushout square in $\A$ (since dg functors between perfect dg categories
preserve finite homotopy colimits). This implies that the image of the composite
\eqref{eq:composite} satisfies condition (2) of Definition \ref{defi:exactinftywald}. It also
satisfies condition (1) by applying the same argument to the (homotopy) zero objects $E_{ii}$ of
$(\E^n)^{\op}\mod$.
\end{proof}

\end{document}